\documentclass[11pt]{article}
\usepackage{amssymb,amsfonts}
\usepackage{amsmath,amsthm,amsxtra}

\usepackage{graphicx}
\usepackage{amscd}
\usepackage{ulem}
\usepackage{makeidx}
\usepackage{fancybox}

\newcommand{\nnn}{{\mathbf N}}
\newcommand{\qqq}{{\mathbf Q}}

\newcommand{\zzz}{{\mathbf Z}}

\newtheorem{prop}{Proposition}[section]
\newtheorem{lemma}{Lemma}[section]

\newtheorem{note}{Note}[section]

\numberwithin{equation}{section}

\begin{document}

\title{Mock theta functions and indefinite modular forms II}

\author{\footnote{12-4 Karato-Rokkoudai, Kita-ku, Kobe 651-1334, 
Japan, \qquad
wakimoto.minoru.314@m.kyushu-u.ac.jp, \hspace{5mm}
wakimoto@r6.dion.ne.jp 
}{ Minoru Wakimoto}}

\date{\empty}

\maketitle

\begin{center}
Abstract
\end{center}

In this paper, we compute the Zwegers's modification of the mock theta 
functions $\Phi^{[m,0] \, \ast}$ and study the modular transformation 
properties of the indefinite modular forms which appear in the explicit 
formula for the modified functions $\widetilde{\Phi}^{[m,0] \, \ast}$. 

\tableofcontents

\section{Introduction}

In the previous paper \cite{W2022c} we studied indefinite modular forms 
obtained from the Zwegers' modification of the mock theta function 
$\Phi^{(-)[m, \frac12]}$ for $m \in \frac12 \nnn_{\rm odd}$.
In the current paper, we study the case $m \in \nnn$ where, 
in order that the Zwegers' modification works, we consider
the function
\begin{equation}
\Phi^{(\pm)[m,s] \, \ast} \,\ := \,\ \Phi^{(\pm)[m,s]}_1 + \Phi^{(\pm)[m,s]}_2
\label{m1:eqn:2022-924d}
\end{equation}
with $\Phi^{(\pm)[m,s]}_i$ \,\ $(i=1,2)$ defined in \cite{W2022c}. 

\medskip

Comparing the results in the current paper with those in the previous 
paper \cite{W2022c}, there appear strange difference 
between the case $m \in \nnn$ and 
the case $m \in \frac12 \nnn_{\rm odd}$.
In the case $m \in \frac12 \nnn_{\rm odd}$ in \cite{W2022c}, 
we need 3 series of functions $g^{(i)[m,p]}_k(\tau)$ \, 
$(i=1,2,3)$ to get $SL_2(\zzz)$-invariant family for each 
$m \in \frac12 \nnn_{\rm odd}$ whereas, in the case 
$m \in \nnn$ in the current paper, only one series of functions 
$g^{(1)[m,p] \, \ast}_k(\tau)$ span $SL_2(\zzz)$-invariant spaces
for each $m \in \nnn$.

\medskip

This paper is organized as follows.

In section \ref{sec:preliminaries}, we make preparation on basic properties 
of $\Phi^{(\pm)[m,s] \, \ast}$.
In section \ref{sec:Phi(+:0)ast}, we derive the explicit formula for  
$\Phi^{[m,0] \, \ast}(\tau, z_1,z_2,0)$ by using the Kac-Peterson's 
identity just in the similar way with \cite{W2022e}.  
In section \ref{sec:Phi(+:0)ast:z1-z2=2atau+2b}, we deduce the formula 
for $\Phi^{[m,0] \, \ast}(\tau, z_1+p\tau,z_2-p\tau,0)$ in the case 
when 
\begin{equation}
(z_1,z_2)=
\Big(\frac{z}{2}+\frac{\tau}{2}-\frac12, \, 
\frac{z}{2}-\frac{\tau}{2}+\frac12\Big) 
\hspace{5mm} {\rm and} \hspace{5mm}
(z_1,z_2)=
\Big(\frac{z}{2}+\frac{\tau}{2}, \, \frac{z}{2}-\frac{\tau}{2}\Big).
\label{m1:eqn:2022-1001a}
\end{equation}
In section \ref{sec:Phi;ast:add:(m0):z1-z2=2atau}, we compute the 
Zwegers's correction function $\Phi^{[m,0] \, \ast}_{\rm add}(\tau, z_1,z_2,0)$ 
for $(z_1,z_2)$ given by \eqref{m1:eqn:2022-1001a}.  
In section \ref{sec:Phi(+:0)ast:(z1+ptau:z2-ptau)}, using the 
results obtained in \S \ref{sec:Phi;ast:add:(m0):z1-z2=2atau},
we make detail investigation on the relation between 
$\Phi^{[m,0] \, \ast}(\tau, z_1+p\tau,z_2-p\tau,0)$ and 
the modified function 
$\widetilde{\Phi}^{[m,0] \, \ast}(\tau, z_1+p\tau,z_2-p\tau,0)$
when $(z_1,z_2)$ is given by \eqref{m1:eqn:2022-1001a}.  
In section \ref{sec:Phi:tilde:ast}, we obtain the explicit formula 
for $\widetilde{\Phi}^{[m,0] \, \ast}(\tau, z_1,z_20)$
in the case when $(z_1,z_2)$ is given by \eqref{m1:eqn:2022-1001a}.
In section \ref{sec:Phi:tilde:ast:modular}, we introduce functions 
$\Xi^{(i)[m,p] \, \ast}(\tau,z)$ and $\Upsilon^{(i)[m,p] \, \ast}(\tau,z)$
for $p \in \zzz$ such that $0 \leq p \leq 2m$ and $i \in \{1,2\}$, 
and compute the modular transformation properties of these functions.
In section \ref{sec:g(i):ast}, we define the functions 
$G^{(i)[m,p] \ast}(\tau,z)$ and $g^{(i)[m,p] \, \ast}_k(\tau)$ by
\begin{eqnarray*}
G^{(i)[m,p] \ast}(\tau,z) &:=& 
\Xi^{(i)[m,p] \, \ast}(\tau,z)-\Upsilon^{(i)[m,p] \, \ast}(\tau,z)
\\[2mm] 
&=&
\sum_{\substack{k \, \in \, \zzz \\[1mm] 0 \, \leq \, k \, \leq \, m}}
g^{(i)[m,p] \, \ast}_k(\tau) [\theta_{k,m}+\theta_{-k,m}](\tau, z)
\end{eqnarray*}
and compute modular transformation of these functions.
In section \ref{sec:g(mp):indefinite}, using the relation between 
$g^{(1)[m,p] \, \ast}_k(\tau)$ and $g^{(2)[m,p] \, \ast}_k(\tau)$, 
we define the indefinite modular forms $g^{[m,p] \, \ast}_k(\tau)$
and obtain their modular transformation properties.

\section{Preliminaries}
\label{sec:preliminaries}

\begin{lemma} 
\label{m1:lemma:2022-917a}
Let $m \in \frac12 \nnn$, $s \in \frac12 \zzz$ and $a \in \zzz$. Then 
\begin{equation}
\Phi^{(\pm)[m,s] \ast}(\tau, z_1+a\tau, z_2-a\tau,t)
=
(\pm 1)^a \, e^{2\pi ima(z_1-z_2)}q^{ma^2}
\Phi^{(\pm)[m,s-2am] \ast}(\tau, z_1, z_2,t)
\label{m1:eqn:2022-917a}
\end{equation}
\end{lemma}

\begin{proof} 
This follows immediately from Lemma 2.3 in \cite{W2022c} and 
definition of $\Phi^{(\pm)[m,s] \ast}$.
\end{proof}

\vspace{0mm}

\begin{lemma} 
\label{m1:lemma:2022-917b}
Let $m \in \frac12 \nnn$, $s \in \frac12 \zzz$ and 
$a \in \zzz_{\geq 0}$. Then 
{\allowdisplaybreaks
\begin{eqnarray}
& & \hspace{-10mm}
\Phi^{(\pm)[m,s] \, \ast}(\tau, \, z_1+a\tau, \, z_2-a\tau, \, 0)
\nonumber
\\[0mm]
&=&
(\pm 1)^a e^{2\pi ima(z_1-z_2)}q^{ma^2} \, \bigg\{
\Phi^{(\pm)[m,s] \, \ast}(\tau, z_1, z_2,0)
\nonumber
\\[0mm]
& & + 
\sum_{\substack{k \, \in \, \zzz \\[1mm] 1 \, \leq \, k \, \leq \, 2am}} 
\hspace{-3mm}
e^{-\pi i(k-s)(z_1-z_2)} \, 
q^{-\frac{1}{4m}(k-s)^2} \, 
\big[\theta_{k-s,m}^{(\pm)}+\theta_{-(k-s),m}^{(\pm)}\big]
(\tau, z_1+z_2)\bigg\}
\label{m1:eqn:2022-917b}
\end{eqnarray}}
\end{lemma}

\begin{proof} This lemma can be shown in the similar way with the 
proof of Lemma 2.4 in \cite{W2022c} as folloes.
By Lemma 2.1 in \cite{W2022b} with Remark 2.1 in \cite{W2022c}, 
we have
{\allowdisplaybreaks
\begin{eqnarray*}
& & \hspace{-5mm}
\Phi^{(\pm)[m,s-a] \, \ast}(\tau, z_1,z_2,0) \,\ - \,\ 
\Phi^{(\pm)[m,s] \, \ast}(\tau, z_1,z_2,0)
\\[1mm]
&= &
\sum_{\substack{k \, \in \, \zzz \\[1mm] -a \leq k \leq -1}}
e^{\pi i(s+k)(z_1-z_2)} \, 
q^{- \frac{(s+k)^2}{4m}} \, \big[
\theta^{(\pm)}_{s+k, \, m} + \theta^{(\pm)}_{-(s+k), \, m}\big]
(\tau, z_1+z_2)
\end{eqnarray*}}
Letting \, $a \rightarrow  2am$, we have
{\allowdisplaybreaks
\begin{eqnarray}
& & \hspace{-5mm}
\Phi^{(\pm)[m,s-2am] \, \ast}(\tau, z_1,z_2,0) \, - \, 
\Phi^{(\pm)[m,s] \, \ast}(\tau, z_1,z_2,0)
\nonumber
\\[1mm]
&= &
\sum_{\substack{k \, \in \, \zzz \\[1mm] -2am \leq k \leq -1}}
e^{\pi i(s+k)(z_1-z_2)} 
q^{- \frac{(s+k)^2}{4m}} \big[
\theta^{(\pm)}_{s+k, \, m} + \theta^{(\pm)}_{-(s+k), \, m}\big]
(\tau, z_1+z_2)
\nonumber
\\[1mm]
& \hspace{-5mm}
\underset{\substack{\\[0.5mm] \uparrow \\[1mm] k \, \rightarrow \, -k
}}{=} \hspace{-5mm}
& \,\ 
\sum_{\substack{k \, \in \, \zzz \\[1mm] 1 \leq k \leq 2am}}
e^{\pi i(s-k)(z_1-z_2)} \, 
q^{- \frac{(s-k)^2}{4m}} 
\big[\theta^{(\pm)}_{k-s, \, m} + \theta^{(\pm)}_{-(k-s), \, m}
\big](\tau, z_1+z_2)
\label{m1:eqn:2022-917c}
\end{eqnarray}}
Substituting this equation \eqref{m1:eqn:2022-917c} into 
\eqref{m1:eqn:2022-917a}, we obtain \eqref{m1:eqn:2022-917b},
proving Lemma \ref{m1:lemma:2022-917b}.
\end{proof}

\vspace{1mm}

\begin{lemma} 
\label{m1:lemma:2022-917c}
Let $m \in \frac12 \nnn$, $s \in \frac12 \zzz$ and $p \in \zzz$ 
such that $mp \in \zzz$. Then
\begin{enumerate}
\item[{\rm 1)}] \,\ if \, $p \geq 0$,
{\allowdisplaybreaks
\begin{eqnarray*}
& & \hspace{-10mm}
\Phi^{(\pm)[m,s] \, \ast}(\tau, \, z_1, \, z_2+p\tau, \, 0) 
\,\ = \,\ 
e^{-2\pi impz_1} \, \Phi^{(\pm)[m,s] \, \ast}(\tau, \, z_1, \, z_2, \, 0) 
\\[2mm]
& &- \,\ 
e^{-2\pi impz_1} \sum_{k=0}^{mp-1}
e^{\pi i(s+k)(z_1-z_2)}
q^{-\frac{(s+k)^2}{4m}} \, 
\big[\theta^{(\pm)}_{s+k,m}+ \theta^{(\pm)}_{-(s+k),m}\big]
(\tau, z_1+z_2)
\end{eqnarray*}}
\item[{\rm 2)}] \,\ if \, $p < 0$,
{\allowdisplaybreaks
\begin{eqnarray*}
& & \hspace{-10mm}
\Phi^{(\pm)[m,s] \, \ast}(\tau, \, z_1, \, z_2+p\tau, \, 0) 
\,\ = \,\ 
e^{-2\pi impz_1} \, \Phi^{(\pm)[m,s] \, \ast}(\tau, \, z_1, \, z_2, \, 0) 
\\[2mm]
& &
+ \,\ e^{-2\pi impz_1} 
\sum_{\substack{k \, \in \zzz \\[1mm] mp \, \leq \, k \, < \, 0}}
e^{\pi i(s+k)(z_1-z_2)}
q^{-\frac{(s+k)^2}{4m}} \, 
\big[\theta^{(\pm)}_{s+k,m}+ \theta^{(\pm)}_{-(s+k),m}\big]
(\tau, z_1+z_2)
\end{eqnarray*}}
\end{enumerate}
\end{lemma}

\begin{proof} 
These formulas follow immediately from Lemma 2.5 in \cite{W2022b} and 
definition of $\Phi^{(\pm)[m,s] \ast}$.
\end{proof}

\vspace{1mm}

Following formulas for theta functions can be shown easily by using 
Lemma 1.1 in \cite{W2022c} and Note 1.1 in \cite{W2022c},
and will be used in the proof of Lemmas \ref{m1:lemma:2022-918a} 
and \ref{m1:lemma:2022-918b}.

\vspace{1mm}

\begin{note}
\label{m1:note:2022-917b}
For $m \in \nnn$ and $p \in \zzz$, the following formulas hold:
\begin{enumerate}
\item[{\rm 1)}] \, $\theta_{\frac12, m+\frac12}^{(-)}
\Big(\tau, \, \dfrac{m(2p+1)\tau-m}{m+\frac12}\Big)
\,\ = \,\ 
e^{-\frac{\pi im}{2m+1}} \, q^{-\frac{m^2}{2(2m+1)}(2p+1)^2} \, 
\theta_{2mp+(m+\frac12), \, m+\frac12}^{(-)}(\tau,0)$
\item[{\rm 2)}] \, $\theta_{\frac12, m+\frac12}^{(-)}
\Big(\tau, \, \dfrac{m(2p+1)\tau}{m+\frac12}\Big)
\,\ = \,\ 
q^{-\frac{m^2}{2(2m+1)}(2p+1)^2} \, 
\theta_{2mp+(m+\frac12), \, m+\frac12}^{(-)}(\tau,0)$
\end{enumerate}
\end{note}

\begin{note}
\label{m1:note:2022-917c}
For $m \in \frac12\nnn$ and $p \in \zzz$, the following formulas hold:
\begin{enumerate}
\item[{\rm 1)}]
\begin{enumerate}
\item[{\rm (i)}] $\theta_{-\frac12, m+\frac12}^{(-)}
\Big(\tau, \,\ z+\dfrac{(2p+1)\tau-1}{2m+1}\Big) 
\, = \,\
e^{\frac{\pi i}{2(2m+1)}} \, q^{- \frac{1}{16(m+\frac12)}(2p+1)^2} 
e^{-\frac{\pi i}{2}(2p+1)z} \, \theta_{p, m+\frac12}(\tau,z)$
\item[{\rm (ii)}] $\theta_{\frac12, m+\frac12}^{(-)}
\Big(\tau, \,\ z-\dfrac{(2p+1)\tau-1}{2m+1}\Big) 
\, = \,\
e^{\frac{\pi i}{2(2m+1)}} \, q^{- \frac{1}{16(m+\frac12)}(2p+1)^2} 
e^{\frac{\pi i}{2}(2p+1)z} \, \theta_{-p, m+\frac12}(\tau,z)$
\end{enumerate}
\item[{\rm 2)}]
\begin{enumerate}
\item[{\rm (i)}] \,\ $\theta_{-\frac12, m+\frac12}^{(-)}
\Big(\tau, \,\ z+\dfrac{(2p+1)\tau}{2m+1}\Big) 
\, = \,\
q^{- \frac{1}{16(m+\frac12)}(2p+1)^2} \, 
e^{-\frac{\pi i}{2}(2p+1)z} \, \theta_{p, m+\frac12}^{(-)}(\tau,z)$
\item[{\rm (ii)}] \,\ $\theta_{\frac12, m+\frac12}^{(-)}
\Big(\tau, \,\ z-\dfrac{(2p+1)\tau}{2m+1}\Big) 
\, = \,\
q^{- \frac{1}{16(m+\frac12)}(2p+1)^2} \, 
e^{\frac{\pi i}{2}(2p+1)z} \, \theta_{-p, m+\frac12}^{(-)}(\tau,z)$
\end{enumerate}
\end{enumerate}
\end{note}

\vspace{1mm}

\begin{note}
\label{m1:note:2022-917d}
For $p \in \zzz$, the following formulas hold:
\begin{enumerate}
\item[{\rm 1)}] 
\begin{enumerate}
\item[{\rm (i)}] \quad $\vartheta_{11}
\Big(\tau, \, \dfrac{z}{2}+\dfrac{(2p+1)\tau-1}{2}\Big)
\,\ = \,\ 
q^{-\frac18(2p+1)^2} \, e^{-\frac{\pi i}{2}(2p+1)z} \, 
\theta_{0, \frac12}(\tau,z)$
\item[{\rm (ii)}] \quad $\vartheta_{11}
\Big(\tau, \, \dfrac{z}{2}-\dfrac{(2p+1)\tau-1}{2}\Big)
\,\ = \,\ - \, 
q^{-\frac18(2p+1)^2} \, e^{\frac{\pi i}{2}(2p+1)z} \, 
\theta_{0, \frac12}(\tau,z)$
\end{enumerate}
\item[{\rm 2)}] 
\begin{enumerate}
\item[{\rm (i)}] \quad $\vartheta_{11}
\Big(\tau, \,\ \dfrac{z}{2}+\Big(p+\dfrac12\Big)\tau\Big)
\,\ = \,\ 
- \, i \, (-1)^p \, 
q^{-\frac18(2p+1)^2} \, 
e^{-\frac{\pi i}{2}(2p+1)z} \, \theta_{0, \frac12}^{(-)}(\tau,z)$
\item[{\rm (ii)}] \quad $\vartheta_{11}
\Big(\tau, \,\ \dfrac{z}{2}-\Big(p+\dfrac12\Big)\tau\Big)
\,\ = \,\ 
i \, (-1)^p \, 
q^{-\frac18(2p+1)^2} \, 
e^{\frac{\pi i}{2}(2p+1)z} \, \theta_{0, \frac12}^{(-)}(\tau,z)$
\end{enumerate}
\end{enumerate}
\end{note}

\section{Explicit formula for $\Phi^{[m,0] \, \ast}(\tau, z_1, z_2,0)$}
\label{sec:Phi(+:0)ast}

\begin{prop} 
\label{m1:prop:2022-917a}
For $m \in \frac12 \nnn$, 
$\Phi^{[m,0] \, \ast}(\tau, z_1, z_2, 0)$ is given by the 
following formula:
{\allowdisplaybreaks
\begin{eqnarray}
& & \hspace{-10mm}
\theta_{\frac12, \, m+\frac12}^{(-)}
\Big(\tau, \,\ \frac{m(z_1-z_2)}{m+\frac12} \Big)
\Phi^{[m,0] \, \ast}(\tau, z_1, z_2, 0) 
\nonumber
\\[3mm]
&=&
- \,\ \bigg[
\sum_{\substack{j, \, k \, \in \zzz \\[1mm] 0 \, < \, k \, \leq \, 2mj}}
-
\sum_{\substack{j, \, k \, \in \zzz \\[1mm] 2mj \, < \, k \, \leq \, 0}}
\bigg] \, 
(-1)^j 
q^{(m+\frac12)(j+\frac{1}{4(m+\frac12)})^2} 
q^{-\frac{k^2}{4m}}
\nonumber
\\[2mm]
& & \hspace{10mm}
\times \,\ 
e^{2\pi im(j+\frac{1}{4(m+\frac12)})(z_1-z_2)} \, 
e^{-\pi ik(z_1-z_2)} \, 
\big[\theta_{k,m}+\theta_{-k,m}\big](\tau, z_1+z_2) \hspace{10mm}
\nonumber
\\[3mm]
& & \hspace{-5mm}
+ \,\ i \, \eta(\tau)^3 \Bigg\{- \, 
\frac{ \displaystyle 
\theta_{-\frac12, \, m+\frac12}^{(-)}
\Big(\tau, \,\ z_1+z_2+\frac{z_1-z_2}{2m+1}\Big)}{\vartheta_{11}(\tau, z_1)}
+ \frac{\displaystyle 
\theta_{\frac12, \, m+\frac12}^{(-)}
\Big(\tau, \,\ z_1+z_2-\frac{z_1-z_2}{2m+1}\Big)}{\vartheta_{11}(\tau, z_2)}
\Bigg\}
\nonumber
\\[-1mm]
& &
\label{m1:eqn:2022-917d}
\end{eqnarray}}
\end{prop}

\begin{proof} Lettig $z_1=z_2$ in the formulas (3.1a) and (3.1b) 
in \cite{W2022d} and making their sum, we have 
{\allowdisplaybreaks
\begin{eqnarray*}
& &
\sum_{j \in \zzz} (-1)^j
q^{(m+\frac12)j^2+\frac12 j} \, 
e^{-2\pi ijm(z_1-z_3)} \, \big[
\underbrace{\Phi^{[m,0]}_1+\Phi^{[m,0]}_2}_{\substack{|| \\[0mm] 
{\displaystyle \Phi^{[m,0] \, \ast}
}}}
\big](\tau, \, z_1, \, -z_3-2j\tau, \, 0)
\\[2mm]
&=&
e^{-\pi iz_1}
\sum_{k \in \zzz} (-1)^k
q^{(m+\frac12)k^2-\frac12 k} \, 
e^{2\pi ikm(z_1-z_3)} \, 
\Phi^{(-)[\frac12,\frac12]}_1(\tau, \, z_1, \, -z_1-2k\tau, \, 0)
\\[1mm]
& & \hspace{-4.5mm}
+ \,\ e^{-\pi i z_3}
\sum_{k \in \zzz} (-1)^k
q^{(m+\frac12)k^2-\frac12 k} \, 
e^{-2\pi ikm(z_1-z_3)} \, 
\Phi^{(-)[\frac12,\frac12]}_1(\tau, \, z_3, \, -z_3-2k\tau, \, 0)
\end{eqnarray*}}
Letting $(j,k)\rightarrow (-j, -k)$ and changing the notation 
$-z_3 \rightarrow z_2$, this formula becomes:
{\allowdisplaybreaks
\begin{eqnarray}
& &
\sum_{j \in \zzz} (-1)^j
q^{(m+\frac12)j^2-\frac12 j} \, 
e^{2\pi ijm(z_1+z_2)} \, 
\Phi^{[m,0] \, \ast}(\tau, \, z_1, \, z_2+2j\tau, \, 0)
\nonumber
\\[2mm]
& & \hspace{-10mm}
= \,\ e^{-\pi iz_1}
\sum_{k \in \zzz} (-1)^k
q^{(m+\frac12)k^2+\frac12 k} \, 
e^{-2\pi ikm(z_1+z_2)} \, 
\Phi^{(-)[\frac12,\frac12]}_1(\tau, \, z_1, \, -z_1+2k\tau, \, 0)
\nonumber
\\[2mm]
& & \hspace{-9mm}
+ \,\ e^{\pi i z_2}
\sum_{k \in \zzz} (-1)^k
q^{(m+\frac12)k^2+\frac12 k} \, 
e^{2\pi ikm(z_1+z_2)} \, 
\Phi^{(-)[\frac12,\frac12]}_1(\tau, \, -z_2, \, z_2+2k\tau, \, 0)
\label{m1:eqn:2022-917e}
\end{eqnarray}}
First we compute the LHS of this equation \eqref{m1:eqn:2022-917e}:
{\allowdisplaybreaks
\begin{eqnarray*}
& & \hspace{-10mm}
\text{LHS of \eqref{m1:eqn:2022-917e}} 
\\[2mm]
&=& 
\underbrace{q^{- \frac{1}{16(m+\frac12)}} 
\sum_{j \in \zzz_{\geq 0}}
(-1)^j q^{(m+\frac12)(j-\frac{1}{4(m+\frac12)})^2} 
e^{2\pi ijm(z_1+z_2)}
\Phi^{[m,0] \, \ast}(\tau, \, z_1, \, z_2+2j\tau, \, 0)}_{\rm (I)}
\\[1mm]
& & 
+ \, 
\underbrace{q^{- \frac{1}{16(m+\frac12)}} 
\sum_{j \in \zzz_{< 0}}
(-1)^j q^{(m+\frac12)(j-\frac{1}{4(m+\frac12)})^2} 
e^{2\pi ijm(z_1+z_2)}
\Phi^{[m,0] \, \ast}(\tau, \, z_1, \, z_2+2j\tau, \, 0)}_{\rm (II)}
\end{eqnarray*}}
where (I) and (II) are computed by using Lemma \ref{m1:lemma:2022-917c} as folloows:
{\allowdisplaybreaks
\begin{eqnarray*}
& & \hspace{-10mm}
{\rm (I)} \, = \, 
q^{- \frac{1}{16(m+\frac12)}} 
\sum_{j \in \zzz_{\rm \geq 0}} (-1)^j 
q^{(m+\frac12)(j-\frac{1}{4(m+\frac12)})^2} 
e^{-2\pi ijm(z_1-z_2)}
\Phi^{[m,0] \, \ast}(\tau, z_1, z_2, 0) 
\\[1mm]
& & \hspace{-5mm}
- \,\ q^{- \frac{1}{16(m+\frac12)}} 
\sum_{j \in \zzz_{\rm \geq 0}} \sum_{k=0}^{2mj-1}
(-1)^j 
q^{(m+\frac12)(j-\frac{1}{4(m+\frac12)})^2} \hspace{-1mm}
q^{-\frac{k^2}{4m}} 
e^{-2\pi ijm(z_1-z_2)}
e^{\pi ik(z_1-z_2)}
\\[0mm]
& & 
\times \,\ 
\big[\theta_{k,m}+\theta_{-k,m}\big](\tau, z_1+z_2)
\\[3mm]
& & \hspace{-10mm}
{\rm (II)} \, = \, 
q^{- \frac{1}{16(m+\frac12)}} 
\sum_{j \in \zzz_{< 0}} (-1)^j 
q^{(m+\frac12)(j-\frac{1}{4(m+\frac12)})^2} 
e^{-2\pi ijm(z_1-z_2)}
\Phi^{[m,0] \, \ast}(\tau, z_1, z_2, 0) 
\\[1mm]
& & \hspace{-8mm}
+ \,\ q^{- \frac{1}{16(m+\frac12)}} 
\sum_{\substack{j, \, k \, \in \zzz \\[1mm] 2mj \, \leq \, k \, < \, 0}}
(-1)^j 
q^{(m+\frac12)(j-\frac{1}{4(m+\frac12)})^2} \hspace{-1mm}
q^{-\frac{k^2}{4m}} 
e^{-2\pi ijm(z_1-z_2)}
e^{\pi ik(z_1-z_2)}
\\[0mm]
& &
\times \,\ 
\big[\theta_{k,m}+\theta_{-k,m}\big](\tau, z_1+z_2)
\end{eqnarray*}}
Then we have  
\begin{subequations}
{\allowdisplaybreaks
\begin{eqnarray}
& & \hspace{-10mm}
\text{LHS of \eqref{m1:eqn:2022-917e}} \,\ = \,\ {\rm (I)}+ {\rm (II)}
\nonumber
\\[2mm]
&=& 
q^{- \frac{1}{16(m+\frac12)}} 
\underbrace{\sum_{j \in \zzz} (-1)^j 
q^{(m+\frac12)(j-\frac{1}{4(m+\frac12)})^2} 
e^{-2\pi ijm(z_1-z_2)}}_{\substack{|| \\[0mm] {\displaystyle 
e^{\frac{- \pi im}{2m+1}(z_1-z_2)} \, 
\theta_{\frac12, \, m+\frac12}^{(-)}
\Big(\tau, \,\ \frac{m}{m+\frac12}(z_1-z_2) \Big)
}}}
\Phi^{[m,0] \, \ast}(\tau, z_1, z_2, 0) 
\nonumber
\\[1mm]
& & \hspace{-5mm}
- \,\ q^{- \frac{1}{16(m+\frac12)}} 
\bigg[
\sum_{\substack{j, \, k \, \in \zzz \\[1mm] 0 \, \leq \, k \, < \, 2mj}}
-
\sum_{\substack{j, \, k \, \in \zzz \\[1mm] 2mj \, \leq \, k \, < \, 0}}
\bigg]
\nonumber
\\[2mm]
& &
\times \,\ (-1)^j 
q^{(m+\frac12)(j-\frac{1}{4(m+\frac12)})^2} \hspace{-1mm}
q^{-\frac{k^2}{4m}} 
e^{-2\pi ijm(z_1-z_2)}
e^{\pi ik(z_1-z_2)}
\big[\theta_{k,m}+\theta_{-k,m}\big](\tau, z_1+z_2)
\nonumber
\\[0mm]
& &
\label{m1:eqn:2022-917f1}
\end{eqnarray}}
Next we compute the RHS of \eqref{m1:eqn:2022-917e} by using 
Lemma 2.1 in \cite{W2022e}:
{\allowdisplaybreaks
\begin{eqnarray}
& & \hspace{-10mm}
\text{RHS of \eqref{m1:eqn:2022-917e}}
\nonumber 
\\[2mm]
&=&
e^{-\pi iz_1}
\sum_{k \in \zzz} (-1)^k
q^{(m+\frac12)k^2+\frac12 k} \, 
e^{-2\pi ikm(z_1+z_2)} 
\underbrace{\Phi^{(-)[\frac12,\frac12]}_1(\tau, \, z_1, \, -z_1+2k\tau, \, 0)
}_{\substack{|| \\[-2mm] {\displaystyle 
- \, i \, e^{-2\pi ikz_1} \, \frac{\eta(\tau)^3}{\vartheta_{11}(\tau, z_1)}
}}}
\nonumber
\\[2mm]
& & \hspace{-3mm}
+ \,\ e^{\pi i z_2}
\sum_{k \in \zzz} (-1)^k
q^{(m+\frac12)k^2+\frac12 k} \, 
e^{2\pi ikm(z_1+z_2)} 
\underbrace{\Phi^{(-)[\frac12,\frac12]}_1(\tau, \, -z_2, \, z_2+2k\tau, \, 0)
}_{\substack{|| \\[-2mm] {\displaystyle 
- \, i \, e^{2\pi ikz_2} \, \frac{\eta(\tau)^3}{\vartheta_{11}(\tau, -z_2)}
}}}
\nonumber
\\[0mm]
&=&
- \,\ i \, q^{-\frac{1}{16(m+\frac12)}} \, 
e^{-\frac{\pi im}{2(m+\frac12)} (z_1-z_2)} 
\nonumber
\\[2mm]
& &
\times \,\ 
\underbrace{\sum_{k \in \zzz} (-1)^k \, 
q^{(m+\frac12)(k+\frac{1}{4(m+\frac12)})^2} \, 
e^{-2\pi i(m+\frac12)(k+\frac{1}{4(m+\frac12)}) \cdot 
\frac{m(z_1+z_2)+z_1}{m+\frac12}}}_{\substack{|| \\[-2mm] 
{\displaystyle \hspace{14mm}
\theta_{\frac12, \, m+\frac12}^{(-)}
\Big(\tau, \,\ - \, \frac{m(z_1+z_2)+z_1}{m+\frac12}\Big)
}}}
\frac{\eta(\tau)^3}{\vartheta_{11}(\tau, z_1)}
\nonumber
\\[2mm]
& & \hspace{-3mm}
+ \,\ i \, q^{-\frac{1}{16(m+\frac12)}} \, 
e^{-\frac{\pi im}{2(m+\frac12)} (z_1-z_2)} \, 
\nonumber
\\[2mm]
& &
\times \,\ 
\underbrace{\sum_{k \in \zzz} (-1)^k \, 
q^{(m+\frac12)(k+\frac{1}{4(m+\frac12)})^2} \, 
e^{2\pi i(m+\frac12)(k+\frac{1}{4(m+\frac12)}) \cdot 
\frac{m(z_1+z_2)+z_2}{m+\frac12}}}_{\substack{|| \\[-2mm] 
{\displaystyle \hspace{14mm}
\theta_{\frac12, \, m+\frac12}^{(-)}
\Big(\tau, \,\ \frac{m(z_1+z_2)+z_2}{m+\frac12}\Big)
}}}
\frac{\eta(\tau)^3}{\vartheta_{11}(\tau, z_2)}
\nonumber
\\[3mm]
&=&
- \,\ i \, q^{-\frac{1}{16(m+\frac12)}} \, 
e^{-\frac{\pi im}{2m+1} (z_1-z_2)} 
\theta_{-\frac12, \, m+\frac12}^{(-)}
\Big(\tau, \,\ z_1+z_2+\frac{z_1-z_2}{2m+1}\Big)
\frac{\eta(\tau)^3}{\vartheta_{11}(\tau, z_1)}
\nonumber
\\[2mm]
& & 
+ \,\ i \, q^{-\frac{1}{16(m+\frac12)}} \, 
e^{-\frac{\pi im}{2m+1} (z_1-z_2)} 
\theta_{\frac12, \, m+\frac12}^{(-)}
\Big(\tau, \,\ z_1+z_2-\frac{z_1-z_2}{2m+1}\Big)
\frac{\eta(\tau)^3}{\vartheta_{11}(\tau, z_2)}
\label{m1:eqn:2022-917f2}
\end{eqnarray}}
\end{subequations}
Then by \eqref{m1:eqn:2022-917f1} and \eqref{m1:eqn:2022-917f2},
we have
{\allowdisplaybreaks
\begin{eqnarray*}
& & 
q^{- \frac{1}{16(m+\frac12)}} 
e^{\frac{- \pi im}{2m+1}(z_1-z_2)} \, 
\theta_{\frac12, \, m+\frac12}^{(-)}
\Big(\tau, \,\ \frac{m(z_1-z_2)}{m+\frac12} \Big)
\Phi^{[m,0] \, \ast}(\tau, z_1, z_2, 0) 
\\[1mm]
& & \hspace{-4mm}
- \,\ q^{- \frac{1}{16(m+\frac12)}} \, \bigg[
\sum_{\substack{j, \, k \, \in \zzz \\[1mm] 0 \, \leq \, k \, < \, 2mj}}
-
\sum_{\substack{j, \, k \, \in \zzz \\[1mm] 2mj \, \leq \, k \, < \, 0}}
\bigg]
\\[2mm]
& &
\times \,\ (-1)^j 
q^{(m+\frac12)(j-\frac{1}{4(m+\frac12)})^2} \hspace{-1mm}
q^{-\frac{k^2}{4m}} \, 
e^{\pi i(k-2jm)(z_1-z_2)} \, 
\big[\theta_{k,m}+\theta_{-k,m}\big](\tau, z_1+z_2)
\\[3mm]
&=&
- \,\ i \, q^{-\frac{1}{16(m+\frac12)}} \, 
e^{-\frac{\pi im}{2m+1}(z_1-z_2)}
\theta_{-\frac12, \, m+\frac12}^{(-)}
\Big(\tau, \,\ z_1+z_2+\frac{z_1-z_2}{2m+1}\Big)
\frac{\eta(\tau)^3}{\vartheta_{11}(\tau, z_1)}
\\[2mm]
& & 
+ \,\ i \, q^{-\frac{1}{16(m+\frac12)}} \, 
e^{-\frac{\pi im}{2m+1}(z_1-z_2)}
\theta_{\frac12, \, m+\frac12}^{(-)}
\Big(\tau, \,\ z_1+z_2-\frac{z_1-z_2}{2m+1}\Big)
\frac{\eta(\tau)^3}{\vartheta_{11}(\tau, z_2)}
\end{eqnarray*}}
Multiplying \, $q^{\frac{1}{16(m+\frac12)}} \, 
e^{\frac{\pi im}{2m+1}(z_1-z_2)}$ \, to both sides, we have 
\begin{subequations}
{\allowdisplaybreaks
\begin{eqnarray}
& & \hspace{-10mm}
\theta_{\frac12, \, m+\frac12}^{(-)}
\Big(\tau, \,\ \frac{m(z_1-z_2)}{m+\frac12} \Big)
\Phi^{[m,0] \, \ast}(\tau, z_1, z_2, 0) 
\nonumber
\\[2mm]
&=&
\bigg[
\sum_{\substack{j, \, k \, \in \zzz \\[1mm] 0 \, \leq \, k \, < \, 2mj}}
-
\sum_{\substack{j, \, k \, \in \zzz \\[1mm] 2mj \, \leq \, k \, < \, 0}}
\bigg] \, 
(-1)^j 
q^{(m+\frac12)(j-\frac{1}{4(m+\frac12)})^2} \hspace{-1mm}
q^{-\frac{k^2}{4m}}
\nonumber
\\[2mm]
& & \hspace{10mm}
\times \,\ 
e^{-2\pi im(j-\frac{1}{4(m+\frac12)})(z_1-z_2)} \, 
e^{\pi ik(z_1-z_2)} \, 
\big[\theta_{k,m}+\theta_{-k,m}\big](\tau, z_1+z_2) \hspace{10mm}
\nonumber
\\[2mm]
& &
- \,\ i \, \theta_{-\frac12, \, m+\frac12}^{(-)}
\Big(\tau, \,\ z_1+z_2+\frac{z_1-z_2}{2m+1}\Big)
\frac{\eta(\tau)^3}{\vartheta_{11}(\tau, z_1)}
\nonumber
\\[2mm]
& &
+ \,\ i \, 
\theta_{\frac12, \, m+\frac12}^{(-)}
\Big(\tau, \,\ z_1+z_2-\frac{z_1-z_2}{2m+1}\Big)
\frac{\eta(\tau)^3}{\vartheta_{11}(\tau, z_2)}
\label{m1:eqn:2022-917g1}
\end{eqnarray}}
The 1st term in the RHS of this equation \eqref{m1:eqn:2022-917g1}
is rewritten by putting $j=-j'$ and $k=-k'$ as follows: 
{\allowdisplaybreaks
\begin{eqnarray}
& & \hspace{-10mm}
\text{the 1st term in the RHS of \eqref{m1:eqn:2022-917g1}}
\nonumber
\\[2mm]
&=& 
\bigg[
\sum_{\substack{j', \, k' \, \in \zzz \\[1mm] 2mj' < \, k' \leq \, 0}}
-
\sum_{\substack{j', \, k' \, \in \zzz \\[1mm] 0 \, < \, k' \leq \, 2mj'}}
\bigg] \, 
(-1)^{j'} 
q^{(m+\frac12)(j'+\frac{1}{4(m+\frac12)})^2} \hspace{-1mm}
q^{-\frac{k' {}^2}{4m}}
\nonumber
\\[2mm]
& & 
\times \,\ 
e^{2\pi im(j'+\frac{1}{4(m+\frac12)})(z_1-z_2)} \, 
e^{-\pi ik'(z_1-z_2)} \, 
\big[\theta_{-k',m}+\theta_{k',m}\big](\tau, z_1+z_2) \hspace{10mm}
\label{m1:eqn:2022-917g2}
\end{eqnarray}}
\end{subequations}
Then by \eqref{m1:eqn:2022-917g1} and \eqref{m1:eqn:2022-917g2}
we obtain \eqref{m1:eqn:2022-917d}, proving 
Proposition \ref{m1:prop:2022-917a}.
\end{proof}

\vspace{1mm}

In order to rewrite the formula \eqref{m1:eqn:2022-917d}, we use 
the following:

\vspace{1mm}

\begin{note} 
\label{m1:note:2022-917a}
For $m \in \frac12 \nnn$, the following formula holds:
{\allowdisplaybreaks
\begin{eqnarray}
& & \hspace{-10mm}
\bigg[
\sum_{\substack{j, \, k \, \in \zzz \\[1mm] 0 \, < \, k \, \leq \, 2mj}}
-
\sum_{\substack{j, \, k \, \in \zzz \\[1mm] 2mj \, < \, k \, \leq \, 0}}
\bigg]
(-1)^j \, 
q^{(m+\frac12)(j+\frac{1}{4(m+\frac12)})^2} \, 
q^{-\frac{k^2}{4m}} \,\ 
e^{2\pi im(j+\frac{1}{4(m+\frac12)})(z_1-z_2)}
\nonumber
\\[2mm]
& & \hspace{10mm}
\times \,\ e^{-\pi ik(z_1-z_2)} \, 
\big[\theta_{k,m}+\theta_{-k,m}\big](\tau, z_1+z_2)
\nonumber
\\[2mm]
&=&
\bigg[
\sum_{\substack{j, \, r \, \in \zzz \\[1mm] 0 \, < \, r \, \leq \, j}}
-
\sum_{\substack{j, \, r \, \in \zzz \\[1mm] j \, < \, r \, \leq \, 0}}
\bigg]
\sum_{\substack{s \, \in \zzz \\[1mm] 0 \, \leq \, s \, \leq \, m}}
(-1)^j \, 
q^{(m+\frac12)(j+\frac{1}{4(m+\frac12)})^2} \, 
q^{-\frac{(2mr-s)^2}{4m}} 
\nonumber
\\[2mm]
& & 
\times \,\ 
e^{2\pi im(j+\frac{1}{4(m+\frac12)})(z_1-z_2)} \, 
e^{-\pi i(2mr-s)(z_1-z_2)} \, 
\big[\theta_{s,m}+\theta_{-s,m}\big](\tau, z_1+z_2) \hspace{20mm}
\nonumber
\\[3mm]
& & \hspace{-5mm}
+ \,\ 
\bigg[
\sum_{\substack{j, \, r \, \in \zzz \\[1mm] 0 \, \leq \, r \, < \, j}}
-
\sum_{\substack{j, \, r \, \in \zzz \\[1mm] j \, \leq \, r \, < \, 0}}
\bigg]
\sum_{\substack{s \, \in \zzz \\[1mm] 0 \, < \, s \, < \, m}}
(-1)^j \, 
q^{(m+\frac12)(j+\frac{1}{4(m+\frac12)})^2} \, 
q^{-\frac{(2mr+s)^2}{4m}} 
\nonumber
\\[2mm]
& & 
\times \,\ 
e^{2\pi im(j+\frac{1}{4(m+\frac12)})(z_1-z_2)} \, 
e^{-\pi i(2mr+s)(z_1-z_2)} \, 
\big[\theta_{s,m}+\theta_{-s,m}\big](\tau, z_1+z_2)
\label{m1:eqn:2022-917h}
\end{eqnarray}}
\end{note}

\begin{proof} In order to prove \eqref{m1:eqn:2022-917h}, 
we need only to note the following for $j \in \zzz$ :
{\allowdisplaybreaks
\begin{eqnarray*}
\Big\{ k \in \zzz \,\ ; \,\ 0 \, < \, k \, \leq \, 2mj\Big\} 
&=&
\bigg\{k = 2mr-s \,\ ; \,\ r, \, s \, \in \zzz \quad {\rm and} \,\ 
\begin{array}{l}
0 < r \leqq j \\[0mm]
0 \leq s < 2m
\end{array} \bigg\}
\\[2mm]
\Big\{ k \in \zzz \,\ ; \,\ 2mj \, < \, k \, \leq \, 0 \Big\} 
&=&
\bigg\{k = 2mr-s \,\ ; \,\ r, \, s \, \in \zzz \quad {\rm and} \,\ 
\begin{array}{l}
j < r \leq 0 \\[0mm]
0 \leq s < 2m
\end{array} \bigg\}
\end{eqnarray*}}
Then the LHS of \eqref{m1:eqn:2022-917h} is rewitten as follows:
$$ \hspace{-30mm}
\text{LHS of \eqref{m1:eqn:2022-917h}} \,\ = \,\ {\rm (I)} +{\rm (II)}
$$
where 
\begin{subequations}
{\allowdisplaybreaks
\begin{eqnarray}
{\rm (I)} &:=&
\bigg[
\sum_{\substack{j, \, r \, \in \zzz \\[1mm] 0 \, < \, r \, \leq \, j}}
-
\sum_{\substack{j, \, r \, \in \zzz \\[1mm] j \, < \, r \, \leq \, 0}}
\bigg]
\sum_{\substack{s \, \in \zzz \\[1mm] 0 \, \leq \, s \, \leq \, m}}
(-1)^j \, 
q^{(m+\frac12)(j+\frac{1}{4(m+\frac12)})^2} \, 
q^{-\frac{(2mr-s)^2}{4m}} 
\nonumber
\\[2mm]
& & 
\times \,\ 
e^{2\pi im(j+\frac{1}{4(m+\frac12)})(z_1-z_2)} \, 
e^{-\pi i(2mr-s)(z_1-z_2)} \, 
\big[\theta_{-s,m}+\theta_{s,m}\big](\tau, z_1+z_2)
\label{m1:eqn:2022-917j1}
\\[3mm]
{\rm (II)} &:=&
\bigg[
\sum_{\substack{j, \, r \, \in \zzz \\[1mm] 0 \, < \, r \, \leq \, j}}
-
\sum_{\substack{j, \, r \, \in \zzz \\[1mm] j \, < \, r \, \leq \, 0}}
\bigg]
\sum_{\substack{s \, \in \zzz \\[1mm] m \, < \, s \, < \, 2m}}
(-1)^j \, 
q^{(m+\frac12)(j+\frac{1}{4(m+\frac12)})^2} \, 
q^{-\frac{(2mr-s)^2}{4m}} 
\nonumber
\\[2mm]
& & 
\times \,\ 
e^{2\pi im(j+\frac{1}{4(m+\frac12)})(z_1-z_2)} \, 
e^{-\pi i(2mr-s)(z_1-z_2)} \, 
\big[\theta_{-s,m}+\theta_{s,m}\big](\tau, z_1+z_2) \hspace{10mm}
\nonumber
\end{eqnarray}}
Putting \, $s' \, := 2m-s$, we have 
$$
m \, < \, s \, < \, 2m \quad \Longleftrightarrow \quad 0 \, < \, s' \, < \, m 
$$
so (II) is rewritten as follows:
{\allowdisplaybreaks
\begin{eqnarray}
& & \hspace{-10mm}
{\rm (II)} \, = \, 
\bigg[
\sum_{\substack{j, \, r \, \in \zzz \\[1mm] 0 \, < \, r \, \leq \, j}}
-
\sum_{\substack{j, \, r \, \in \zzz \\[1mm] j \, < \, r \, \leq \, 0}}
\bigg]
\sum_{\substack{s' \, \in \zzz \\[1mm] 0 \, < \, s' \, < \, m}}
(-1)^j \, 
q^{(m+\frac12)(j+\frac{1}{4(m+\frac12)})^2} \, 
q^{-\frac{(2m(r-1)+s')^2}{4m}} 
\nonumber
\\[2mm]
& & 
\times \,\ 
e^{2\pi im(j+\frac{1}{4(m+\frac12)})(z_1-z_2)} \, 
e^{-\pi i(2m(r-1)+s')(z_1-z_2)} \, 
\big[\theta_{s',m}+\theta_{-s',m}\big](\tau, z_1+z_2)
\nonumber
\\[2mm]
& \hspace{-3mm}
\underset{\substack{\\[0.5mm] \uparrow \\[1mm] r-1 \, = \, r'
}}{=} \hspace{-3mm} &
\bigg[
\sum_{\substack{j, \, r' \, \in \zzz \\[1mm] 0 \, \leq \, r' \, < \, j}}
-
\sum_{\substack{j, \, r' \, \in \zzz \\[1mm] j \, \leq \, r' \, < \, 0}}
\bigg]
\sum_{\substack{s' \, \in \zzz \\[1mm] 0 \, < \, s' \, < \, m}}
(-1)^j \, 
q^{(m+\frac12)(j+\frac{1}{4(m+\frac12)})^2} \, 
q^{-\frac{(2mr'+s')^2}{4m}} 
\nonumber
\\[2mm]
& & 
\times \,\ 
e^{2\pi im(j+\frac{1}{4(m+\frac12)})(z_1-z_2)} \, 
e^{-\pi i(2mr'+s')(z_1-z_2)} \, 
\big[\theta_{s',m}+\theta_{-s',m}\big](\tau, z_1+z_2)
\label{m1:eqn:2022-917j2}
\end{eqnarray}}
\end{subequations}
Then by \eqref{m1:eqn:2022-917j1} and \eqref{m1:eqn:2022-917j2}, we have 
$$
{\rm (I)}+{\rm (II)} \,\ = \,\ \text{RHS of \eqref{m1:eqn:2022-917h}} \, ,
\hspace{10mm}
\text{proving Note \ref{m1:note:2022-917a}}.
$$

\vspace{-6mm}

\end{proof}

\vspace{1mm}

By Note \ref{m1:note:2022-917a}, the formula \eqref{m1:eqn:2022-917d}
in Proposition \ref{m1:prop:2022-917a} is rewritten as follows:

\vspace{1mm}

\begin{prop} \quad 
\label{m1:prop:2022-917b}
For $m \in \frac12 \nnn$, 
$\Phi^{[m,0] \, \ast}(\tau, z_1, z_2, 0)$ is given by the 
following formula:
{\allowdisplaybreaks
\begin{eqnarray}
& & \hspace{-10mm}
\theta_{\frac12, \, m+\frac12}^{(-)}
\Big(\tau, \,\ \frac{m(z_1-z_2)}{m+\frac12} \Big)
\Phi^{[m,0] \, \ast}(\tau, z_1, z_2, 0) 
\nonumber
\\[3mm]
&=&
i \, \eta(\tau)^3 \Bigg\{- \, 
\frac{ \displaystyle 
\theta_{-\frac12, \, m+\frac12}^{(-)}
\Big(\tau, \,\ z_1+z_2+\frac{z_1-z_2}{2m+1}\Big)}{\vartheta_{11}(\tau, z_1)}
+ \frac{\displaystyle 
\theta_{\frac12, \, m+\frac12}^{(-)}
\Big(\tau, \,\ z_1+z_2-\frac{z_1-z_2}{2m+1}\Big)}{\vartheta_{11}(\tau, z_2)}
\Bigg\}
\nonumber
\\[3mm]
& & \hspace{-5mm}
- \,\ 
\bigg[\sum_{\substack{j, \, r \, \in \zzz \\[1mm] 0 \, < \, r \, \leq \, j}}
-
\sum_{\substack{j, \, r \, \in \zzz \\[1mm] j \, < \, r \, \leq \, 0}}
\bigg]
\sum_{\substack{s \, \in \zzz \\[1mm] 0 \, \leq \, s \, \leq \, m}}
(-1)^j \, 
q^{(m+\frac12)(j+\frac{1}{4(m+\frac12)})^2} \, 
q^{-\frac{(2mr-s)^2}{4m}} 
\nonumber
\\[3mm]
& & 
\times \,\ 
e^{2\pi im(j+\frac{1}{4(m+\frac12)})(z_1-z_2)} \, 
e^{-\pi i(2mr-s)(z_1-z_2)} \, 
\big[\theta_{s,m}+\theta_{-s,m}\big](\tau, z_1+z_2) \hspace{20mm}
\nonumber
\\[3mm]
& & \hspace{-5mm}
- \,\ 
\bigg[\sum_{\substack{j, \, r \, \in \zzz \\[1mm] 0 \, \leq \, r \, < \, j}}
-
\sum_{\substack{j, \, r \, \in \zzz \\[1mm] j \, \leq \, r \, < \, 0}}
\bigg]
\sum_{\substack{s \, \in \zzz \\[1mm] 0 \, < \, s \, < \, m}}
(-1)^j \, 
q^{(m+\frac12)(j+\frac{1}{4(m+\frac12)})^2} \, 
q^{-\frac{(2mr+s)^2}{4m}} 
\nonumber
\\[3mm]
& & 
\times \,\ 
e^{2\pi im(j+\frac{1}{4(m+\frac12)})(z_1-z_2)} \, 
e^{-\pi i(2mr+s)(z_1-z_2)} \, 
\big[\theta_{s,m}+\theta_{-s,m}\big](\tau, z_1+z_2)
\label{m1:eqn:2022-917k}
\end{eqnarray}}
\end{prop}

\section{$\Phi^{[m,0] \, \ast}(\tau, z_1, z_2,t)$ 
$\sim$ the case $z_1-z_2=2a\tau+2b$}
\label{sec:Phi(+:0)ast:z1-z2=2atau+2b}

\subsection{$\Phi^{[m,0] \, \ast}(\tau, z_1, z_2,t)$ 
$\sim$ the case $z_1-z_2=(1+2p)\tau-1$}
\label{subsec:Phi(+:0)ast:z1-z2=(1+2p)tau-1}

\begin{lemma} 
\label{m1:lemma:2022-918a}
For $m \in \nnn$ and $p \in \zzz$, the following formula holds:
{\allowdisplaybreaks
\begin{eqnarray}
& & \hspace{-7mm}
\theta_{2mp+m+\frac12, m+\frac12}^{(-)}(\tau,0) \, 
\Phi^{[m,0] \, \ast}\Big(\tau, \,\ 
\frac{z}{2}+\frac{\tau}{2}-\frac12+p\tau, \,\ 
\frac{z}{2}-\frac{\tau}{2}+\frac12-p\tau, \,\ 0\Big) 
\nonumber
\\[2mm]
&=&
q^{\frac{m}{4}(2p+1)^2} \, 
\frac{\eta(\tau)^3}{\theta_{0,\frac12}(\tau,z)} \cdot 
\big[\theta_{p, m+\frac12}+ \theta_{-p, m+\frac12}\big](\tau,0)
\nonumber
\\[3mm]
& & \hspace{-5mm}
- \,\ 
\bigg[\sum_{\substack{j, \, r \, \in \, \frac12 \, \zzz_{\rm odd} \\[1mm]
0 \, \leq \, r \, < \, j}}
-
\sum_{\substack{j, \, r \, \in \, \frac12 \, \zzz_{\rm odd} \\[1mm]
j \, \leq \, r \, < \, 0}}\bigg]
\sum_{\substack{k \, \in \zzz \\[1mm] 0 \, < \, k \, < \, m}}
(-1)^{j-\frac12+k} \, 
q^{(m+\frac12)(j+\frac{2pm}{2m+1})^2}
\nonumber
\\[-1mm]
& & \hspace{50mm}
\times \,\ 
q^{-\frac{1}{4m} \, [2mr+k+2mp]^2
\, + \, \frac{m}{4}(2p+1)^2} \, 
\big[\theta_{k,m}+\theta_{-k,m}\big](\tau, z)
\nonumber
\\[3mm]
& & \hspace{-5mm}
- \,\ \bigg[
\sum_{\substack{j, \, r \, \in \, \frac12 \, \zzz_{\rm odd} \\[1mm]
0 \, \leq \, r \, \leq \, j}}
-
\sum_{\substack{j, \, r \, \in \, \frac12 \, \zzz_{\rm odd} \\[1mm]
j \, < \, r \, < \, 0}}\bigg]
\sum_{\substack{k \, \in \zzz \\[1mm] 0 \, \leq \, k \, \leq \, m}}
(-1)^{j-\frac12+k} \, 
q^{(m+\frac12)(j+\frac{2pm}{2m+1})^2}
\nonumber
\\[-1mm]
& & \hspace{50mm}
\times \,\ 
q^{-\frac{1}{4m} \, [2mr-k+2mp]^2
\, + \, \frac{m}{4}(2p+1)^2} \, 
\big[\theta_{k,m}+\theta_{-k,m}\big](\tau, z)
\nonumber
\\[3mm]
& & \hspace{-5mm}
+ \,\ 
\theta_{2mp+m+\frac12, m+\frac12}^{(-)}(\tau,0)
\sum_{\substack{k \, \in \zzz \\[1mm] 0 \, \leq \, k \, \leq \, m}} 
(-1)^k
q^{-\frac{1}{4m}k^2+\frac{k}{2}(2p+1)} \, 
\big[\theta_{k,m}+\theta_{-k,m}\big](\tau, z)
\label{m1:eqn:2022-918a}
\end{eqnarray}}
\end{lemma}

\begin{proof} Letting $\left\{
\begin{array}{lcl}
z_1 &=& \frac{z}{2}+\frac{\tau}{2}-\frac12+p\tau \\[1mm]
z_2 &=& \frac{z}{2}-\frac{\tau}{2}+\frac12-p\tau
\end{array}\right. $ namely 
$\left\{
\begin{array}{ccl}
z_1+z_2 &=& z \\[1mm]
z_1-z_2 &=& (2p+1)\tau-1
\end{array}\right. $ in the formula \eqref{m1:eqn:2022-917k} in 
Proposition \ref{m1:prop:2022-917b}, we have
{\allowdisplaybreaks
\begin{eqnarray}
& & \hspace{-10mm}
\theta_{\frac12, \, m+\frac12}^{(-)}
\Big(\tau, \,\ \frac{m((2p+1)\tau-1)}{m+\frac12} \Big) \, 
\Phi^{[m,0] \, \ast}(\tau, z_1, z_2, 0) 
\nonumber
\\[3mm]
&=&
i \, \eta(\tau)^3 \Bigg\{
\underbrace{- \, \frac{ \displaystyle 
\theta_{-\frac12, \, m+\frac12}^{(-)}
\Big(\tau, \,\ z+\frac{(2p+1)\tau-1}{2m+1}\Big)}{\vartheta_{11}(\tau, z_1)}
+ \frac{\displaystyle 
\theta_{\frac12, \, m+\frac12}^{(-)}
\Big(\tau, \,\ z-\frac{(2p+1)\tau-1}{2m+1}\Big)}{\vartheta_{11}(\tau, z_2)}
}_{\substack{ \hspace{6.2mm} || \,\ put \\[0.5mm] {\rm (I)}
}}
\Bigg\}
\nonumber
\\[3mm]
& & \hspace{-5mm}
- \,\ \bigg[
\sum_{\substack{j, \, r \, \in \zzz \\[1mm] 0 \, < \, r \, \leq \, j}}
-
\sum_{\substack{j, \, r \, \in \zzz \\[1mm] j \, < \, r \, \leq \, 0}}
\bigg]
\sum_{\substack{s \, \in \zzz \\[1mm] 0 \, \leq \, s \, \leq \, m}}
(-1)^j \, 
q^{(m+\frac12)(j+\frac{1}{4(m+\frac12)})^2} \, 
q^{-\frac{(2mr-s)^2}{4m}} 
\nonumber
\\[2mm]
& & 
\times \,\ 
e^{2\pi im(j+\frac{1}{4(m+\frac12)})((2p+1)\tau-1)}
e^{-\pi i(2mr-s)((2p+1)\tau-1)}
\big[\theta_{s,m}+\theta_{-s,m}\big](\tau, z) \hspace{20mm}
\nonumber
\\[3mm]
& & \hspace{-5mm}
- \,\ \bigg[
\sum_{\substack{j, \, r \, \in \zzz \\[1mm] 0 \, \leq \, r \, < \, j}}
-
\sum_{\substack{j, \, r \, \in \zzz \\[1mm] j \, \leq \, r \, < \, 0}}
\bigg]
\sum_{\substack{s \, \in \zzz \\[1mm] 0 \, < \, s \, < \, m}}
(-1)^j \, 
q^{(m+\frac12)(j+\frac{1}{4(m+\frac12)})^2} \, 
q^{-\frac{(2mr+s)^2}{4m}} 
\nonumber
\\[3mm]
& & 
\times \,\ 
e^{2\pi im(j+\frac{1}{4(m+\frac12)})((2p+1)\tau-1)}
e^{-\pi i(2mr+s)((2p+1)\tau-1)}
\big[\theta_{s,m}+\theta_{-s,m}\big](\tau, z)
\nonumber
\\[3mm]
&=& i \, \eta(\tau)^3 \, \times \, {\rm (I)}
\nonumber
\\[3mm]
& & \hspace{-5mm}
- \,\ e^{-\frac{\pi im}{2m+1}} \, \bigg[
\sum_{\substack{j, \, r \, \in \zzz \\[1mm] 0 \, < \, r \, \leq \, j}}
-
\sum_{\substack{j, \, r \, \in \zzz \\[1mm] j \, < \, r \, \leq \, 0}}
\bigg]
\sum_{\substack{s \, \in \zzz \\[1mm] 0 \, \leq \, s \, \leq \, m}}
(-1)^{j+s} \, 
q^{(m+\frac12)(j+\frac{1}{4(m+\frac12)})^2 
\, + \, m(j+\frac{1}{4(m+\frac12)})(2p+1)}
\nonumber
\\[2mm]
& &
\times \,\ 
q^{-\frac{(2mr-s)^2}{4m}-\frac12(2mr-s)(2p+1)} \, 
\big[\theta_{s,m}+\theta_{-s,m}\big](\tau, z) \hspace{20mm}
\nonumber
\\[2mm]
& & \hspace{-5mm}
- \,\ e^{-\frac{\pi im}{2m+1}} \, \bigg[
\sum_{\substack{j, \, r \, \in \zzz \\[1mm] 0 \, \leq \, r \, < \, j}}
-
\sum_{\substack{j, \, r \, \in \zzz \\[1mm] j \, \leq \, r \, < \, 0}}
\bigg]
\sum_{\substack{s \, \in \zzz \\[1mm] 0 \, < \, s \, < \, m}}
(-1)^{j+s} \, 
q^{(m+\frac12)(j+\frac{1}{4(m+\frac12)})^2
\, + \, m(j+\frac{1}{4(m+\frac12)})(2p+1)}
\nonumber
\\[2mm]
& & 
\times \,\ 
q^{-\frac{(2mr+s)^2}{4m}-\frac12(2mr+s)(2p+1)} \, 
\big[\theta_{s,m}+\theta_{-s,m}\big](\tau, z)
\label{m1:eqn:2022-918b1}
\end{eqnarray}}
The LHS of this equation \eqref{m1:eqn:2022-918b1} becomes 
by Note \ref{m1:note:2022-917b} as follows:  
$$
\text{LHS of \eqref{m1:eqn:2022-918b1}} \, = \, 
e^{-\frac{\pi im}{2m+1}}
q^{-\frac{m^2}{2(2m+1)}(2p+1)^2}
\theta_{2mp+m+\frac12, m+\frac12}^{(-)}(\tau,0) \cdot 
\Phi^{[m,0] \, \ast}(\tau, z_1, z_2, 0) 
$$
Also (I) is computed by using Notes \ref{m1:note:2022-917c}
and \ref{m1:note:2022-917d} as follows:
{\allowdisplaybreaks
\begin{eqnarray*}
& & \hspace{-7mm}
{\rm (I)} \, = \, - \, \frac{ \displaystyle 
\theta_{-\frac12, \, m+\frac12}^{(-)}
\Big(\tau, \,\ z+\frac{(2p+1)\tau-1}{2m+1}\Big)}{\displaystyle 
\vartheta_{11}\Big(\tau, \, \frac{z}{2}+\frac{(2p+1)\tau-1}{2}\Big)}
\, + \, 
\frac{\displaystyle 
\theta_{\frac12, \, m+\frac12}^{(-)}
\Big(\tau, \,\ z-\frac{(2p+1)\tau-1}{2m+1}\Big)}{\displaystyle 
\vartheta_{11}\Big(\tau, \, \frac{z}{2}-\frac{(2p+1)\tau-1}{2}\Big)}
\\[3mm]
&=&
- \,\ \frac{
e^{\frac{\pi i}{2(2m+1)}} \, q^{-\frac{1}{16(m+\frac12)}(2p+1)^2}
e^{-\frac{\pi i}{2}(2p+1)z}
\theta_{p, m+\frac12}(\tau,0)
}{
q^{-\frac18(2p+1)^2} e^{-\frac{\pi i}{2}(2p+1)z}
\theta_{0,\frac12}(\tau,z)}
\\[2mm]
& &
+ \,\ \frac{
e^{\frac{\pi i}{2(2m+1)}} \, q^{-\frac{1}{16(m+\frac12)}(2p+1)^2}
e^{\frac{\pi i}{2}(2p+1)z}
\theta_{-p, m+\frac12}(\tau,0)
}{- \, 
q^{-\frac18(2p+1)^2} e^{\frac{\pi i}{2}(2p+1)z}
\theta_{0,\frac12}(\tau,z)}
\\[2mm]
&=&
- \, e^{\frac{\pi i}{2(2m+1)}} \, 
q^{\frac{m}{8(m+\frac12)}(2p+1)^2} \, 
\frac{1}{\theta_{0,\frac12}(\tau,z)} \cdot 
\big[\theta_{p, m+\frac12}+ \theta_{-p, m+\frac12}\big](\tau,0)
\end{eqnarray*}}
Then substituting these into \eqref{m1:eqn:2022-918b1} and multiplying 
$e^{\frac{\pi im}{2m+1}}$ and rewriting the 2nd terms in the RHS of 
\eqref{m1:eqn:2022-918b1} by using 
\begin{equation}
\sum_{\substack{j, \, r \, \in \zzz \\[1mm] 0 \, < \, r \, \leq \, j}}
-
\sum_{\substack{j, \, r \, \in \zzz \\[1mm] j \, < \, r \, \leq \, 0}}
\,\ = \,\ 
\bigg[
\sum_{\substack{j, \, r \, \in \zzz \\[1mm] 0 \, \leq \, r \, \leq \, j}}
-
\sum_{\substack{j, \, r \, \in \zzz \\[1mm] j \, < \, r \, < \, 0}} 
\bigg]
\,\ - \,\ 
\sum_{\substack{r \, = \, 0 \\[1mm] j \, \in \, \zzz}} 
\label{m1:eqn:2022-918c}
\end{equation}
the above formula \eqref{m1:eqn:2022-918b1} becomes as follows:
{\allowdisplaybreaks
\begin{eqnarray}
& & \hspace{-10mm}
q^{-\frac{m^2}{2(2m+1)}(2p+1)^2}
\theta_{2mp+m+\frac12, m+\frac12}^{(-)}(\tau,0) \, 
\Phi^{[m,0] \, \ast}(\tau, z_1, z_2, 0) 
\nonumber
\\[3mm]
&=&
q^{\frac{m}{8(m+\frac12)}(2p+1)^2} \, 
\frac{\eta(\tau)^3}{\theta_{0,\frac12}(\tau,z)} \cdot 
\big[\theta_{p, m+\frac12}+ \theta_{-p, m+\frac12}\big](\tau,0)
\nonumber
\\[3mm]
& & \hspace{-5mm}
- \,\ \bigg[
\sum_{\substack{j, \, r \, \in \zzz \\[1mm] 0 \, \leq \, r \, \leq \, j}}
-
\sum_{\substack{j, \, r \, \in \zzz \\[1mm] j \, < \, r \, < \, 0}}
\bigg]
\sum_{\substack{s \, \in \zzz \\[1mm] 0 \, \leq \, s \, \leq \, m}}
(-1)^{j+s} \, 
q^{(m+\frac12)(j+ \frac12+\frac{2pm}{2m+1})^2
\, - \, \frac{m^2}{2(2m+1)}(2p+1)^2}
\nonumber
\\[2mm]
& &
\times \,\ 
q^{-\frac{1}{4m} \, [2m(r+\frac12)-s+2mp]^2
\, + \, \frac{m}{4}(2p+1)^2} \, 
\big[\theta_{s,m}+\theta_{-s,m}\big](\tau, z) \hspace{20mm}
\nonumber
\\[2mm]
& &\hspace{-5mm}
+ \,\ 
\underbrace{\sum_{j \in \zzz} (-1)^j \, 
q^{(m+\frac12)(j+\frac12+\frac{2pm}{2m+1})^2}
}_{\substack{|| \\[-2mm] {\displaystyle 
\theta_{2mp+m+\frac12, m+\frac12}^{(-)}(\tau,0)
}}}
q^{-\frac{m^2}{2(2m+1)}(2p+1)^2}
\nonumber
\\[2mm]
& &
\times \,\
\sum_{\substack{s \, \in \zzz \\[1mm] 0 \, \leq \, s \, \leq \, m}} 
(-1)^s q^{-\frac{1}{4m}s^2+\frac{s}{2}(2p+1)} \, 
\big[\theta_{s,m}+\theta_{-s,m}\big](\tau, z) 
\nonumber
\\[3mm]
& & \hspace{-5mm}
- \,\ \bigg[
\sum_{\substack{j, \, r \, \in \zzz \\[1mm] 0 \, \leq \, r \, < \, j}}
-
\sum_{\substack{j, \, r \, \in \zzz \\[1mm] j \, \leq \, r \, < \, 0}}
\bigg]
\sum_{\substack{s \, \in \zzz \\[1mm] 0 \, < \, s \, < \, m}}
(-1)^{j+s} \, 
q^{(m+\frac12)(j+ \frac12+\frac{2pm}{2m+1})^2
\, - \, \frac{m^2}{2(2m+1)}(2p+1)^2}
\nonumber
\\[3mm]
& & 
\times \,\ 
q^{-\frac{1}{4m} \, [2m(r+\frac12)+s+2mp]^2
\, + \, \frac{m}{4}(2p+1)^2} \, 
\big[\theta_{s,m}+\theta_{-s,m}\big](\tau, z)
\label{m1:eqn:2022-918b2}
\end{eqnarray}}
Multiplying $q^{\frac{m^2}{2(2m+1)}(2p+1)^2}$ and replacing 
$\sum_{j, r \in \zzz}$ with $\sum_{j', r' \in \frac12 \zzz_{\rm odd}}$ 
by putting $j+\frac12=j'$ and $r+\frac12=r'$ in the above formula 
\eqref{m1:eqn:2022-918b2}, we obtain
{\allowdisplaybreaks
\begin{eqnarray*}
& & \hspace{-7mm}
\theta_{2mp+m+\frac12, m+\frac12}^{(-)}(\tau,0) \, 
\Phi^{[m,0] \, \ast}(\tau, z_1, z_2, 0) 
\\[2mm]
&=&
q^{\frac{m}{4}(2p+1)^2} \, 
\frac{\eta(\tau)^3}{\theta_{0,\frac12}(\tau,z)} \cdot 
\big[\theta_{p, m+\frac12}+ \theta_{-p, m+\frac12}\big](\tau,0)
\\[3mm]
& & \hspace{-5mm}
- \,\ \bigg[
\sum_{\substack{j', \, r' \, \in \, \frac12 \, \zzz_{\rm odd} \\[1mm]
0 \, \leq \, r' \, \leq \, j'
}}
-
\sum_{\substack{j', \, r' \, \in \, \frac12 \, \zzz_{\rm odd} \\[1mm]
j' \, < \, r' \, < \, 0}}\bigg]
\sum_{\substack{k \, \in \zzz \\[1mm] 0 \, \leq \, k \, \leq \, m}}
(-1)^{j'-\frac12+k} \, 
q^{(m+\frac12)(j'+\frac{2pm}{2m+1})^2}
\\[-1mm]
& & \hspace{50mm}
\times \,\ 
q^{-\frac{1}{4m} \, [2mr'-k+2mp]^2
\, + \, \frac{m}{4}(2p+1)^2} \, 
\big[\theta_{k,m}+\theta_{-k,m}\big](\tau, z)
\\[3mm]
& & \hspace{-5mm}
+ \,\ 
\theta_{2mp+m+\frac12, m+\frac12}^{(-)}(\tau,0)
\sum_{\substack{k \, \in \zzz \\[1mm] 0 \, \leq \, k \, \leq \, m}} 
(-1)^k
q^{-\frac{1}{4m}k^2+\frac{k}{2}(2p+1)} \, 
\big[\theta_{k,m}+\theta_{-k,m}\big](\tau, z)
\\[3mm]
& & \hspace{-5mm}
- \,\ \bigg[
\sum_{\substack{j', \, r' \, \in \, \frac12 \, \zzz_{\rm odd} \\[1mm]
0 \, \leq \, r' \, < \, j'}}
-
\sum_{\substack{j', \, r' \, \in \, \frac12 \, \zzz_{\rm odd} \\[1mm]
j' \, \leq \, r' \, < \, 0}}
\bigg]
\sum_{\substack{k \, \in \zzz \\[1mm] 0 \, < \, k \, < \, m}}
(-1)^{j'-\frac12+k} \, 
q^{(m+\frac12)(j'+\frac{2pm}{2m+1})^2}
\\[-1mm]
& & \hspace{50mm}
\times \,\ 
q^{-\frac{1}{4m} \, [2mr'+k+2mp]^2
\, + \, \frac{m}{4}(2p+1)^2} \, 
\big[\theta_{k,m}+\theta_{-k,m}\big](\tau, z) \hspace{20mm}
\end{eqnarray*}}
proving Lemma \ref{m1:lemma:2022-918a}.
\end{proof}

\subsection{$\Phi^{[m,0] \, \ast}(\tau, z_1, z_2,t)$ 
$\sim$ the case $z_1-z_2=(1+2p)\tau$}
\label{subsec:Phi(+:0)ast:z1-z2=(1+2p)tau}

\begin{lemma} 
\label{m1:lemma:2022-918b}
For $m \in \nnn$ and $p \in \zzz$, the following formula holds:
{\allowdisplaybreaks
\begin{eqnarray}
& & \hspace{-9mm}
\theta_{2mp+m+\frac12, m+\frac12}^{(-)}(\tau,0) \, 
\Phi^{[m,0] \, \ast}\Big(\tau, \,\ 
\frac{z}{2}+\frac{\tau}{2}+p\tau, \,\ 
\frac{z}{2}-\frac{\tau}{2}+p\tau, \,\ 0 \Big)
\nonumber
\\[2mm]
&=&
(-1)^p \, q^{\frac{m}{4}(2p+1)^2} \, 
\frac{\eta(\tau)^3}{\theta_{0,\frac12}^{(-)}(\tau,z)} \cdot 
\big[\theta_{p, m+\frac12}^{(-)}+ \theta_{-p, m+\frac12}^{(-)}\big](\tau,0)
\nonumber
\\[3mm]
& & \hspace{-5mm}
- \,\ \bigg[
\sum_{\substack{j, \, r \, \in \, \frac12 \, \zzz_{\rm odd} \\[1mm]
0 \, \leq \, r \, < \, j}}
-
\sum_{\substack{j, \, r \, \in \, \frac12 \, \zzz_{\rm odd} \\[1mm]
j \, \leq \, r \, < \, 0}}
\bigg]
\sum_{\substack{k \, \in \zzz \\[1mm] 0 \, < \, k \, < \, m}}
(-1)^{j-\frac12} \, 
q^{(m+\frac12)(j+\frac{2pm}{2m+1})^2}
\nonumber
\\[-1mm]
& & \hspace{40mm}
\times \,\ 
q^{-\frac{1}{4m} \, [2mr+k+2mp]^2
\, + \, \frac{m}{4}(2p+1)^2} \, 
\big[\theta_{k,m}+\theta_{-k,m}\big](\tau, z)
\nonumber
\\[3mm]
& & \hspace{-5mm}
- \,\ \bigg[
\sum_{\substack{j, \, r \, \in \, \frac12 \, \zzz_{\rm odd} \\[1mm]
0 \, \leq \, r \, \leq \, j}}
-
\sum_{\substack{j, \, r \, \in \, \frac12 \, \zzz_{\rm odd} \\[1mm]
j \, < \, r \, < \, 0}}\bigg]
\sum_{\substack{k \, \in \zzz \\[1mm] 0 \, \leq \, k \, \leq \, m}}
(-1)^{j-\frac12} \, 
q^{(m+\frac12)(j+\frac{2pm}{2m+1})^2}
\nonumber
\\[-1mm]
& & \hspace{40mm}
\times \,\ 
q^{-\frac{1}{4m} \, [2mr-k+2mp]^2
\, + \, \frac{m}{4}(2p+1)^2} \, 
\big[\theta_{k,m}+\theta_{-k,m}\big](\tau, z)
\nonumber
\\[3mm]
& & \hspace{-5mm}
+ \,\ 
\theta_{2mp+m+\frac12, m+\frac12}^{(-)}(\tau,0) \, 
\sum_{\substack{k \, \in \zzz \\[1mm] 0 \, \leq \, k \, \leq \, m}} 
q^{-\frac{1}{4m}k^2+\frac{k}{2}(2p+1)} \, 
\big[\theta_{k,m}+\theta_{-k,m}\big](\tau, z)
\label{m1:eqn:2022-918d}
\end{eqnarray}}
\end{lemma}

\begin{proof} Letting $\left\{
\begin{array}{lcl}
z_1 &=& \frac{z}{2}+\frac{\tau}{2}+p\tau \\[0mm]
z_2 &=& \frac{z}{2}-\frac{\tau}{2}-p\tau
\end{array}\right. $ namely $\left\{
\begin{array}{ccl}
z_1+z_2 &=& z \\[1mm]
z_1-z_2 &=& (2p+1)\tau
\end{array}\right. $ in the formula \eqref{m1:eqn:2022-917k} in 
Proposition \ref{m1:prop:2022-917b}, we have
{\allowdisplaybreaks
\begin{eqnarray}
& & \hspace{-10mm}
\theta_{\frac12, \, m+\frac12}^{(-)}
\Big(\tau, \,\ \frac{m(2p+1)\tau}{m+\frac12} \Big) \, 
\Phi^{[m,0] \, \ast}(\tau, z_1, z_2, 0) 
\nonumber
\\[3mm]
&=&
i \, \eta(\tau)^3 \Bigg\{
\underbrace{- \, \frac{ \displaystyle 
\theta_{-\frac12, \, m+\frac12}^{(-)}
\Big(\tau, \,\ z+\frac{(2p+1)\tau}{2m+1}\Big)}{\vartheta_{11}(\tau, z_1)}
+ \frac{\displaystyle 
\theta_{\frac12, \, m+\frac12}^{(-)}
\Big(\tau, \,\ z-\frac{(2p+1)\tau}{2m+1}\Big)}{\vartheta_{11}(\tau, z_2)}
}_{\rm (I)}
\Bigg\}
\nonumber
\\[3mm]
& & \hspace{-5mm}
- \,\ \bigg[
\sum_{\substack{j, \, r \, \in \zzz \\[1mm] 0 \, < \, r \, \leq \, j}}
-
\sum_{\substack{j, \, r \, \in \zzz \\[1mm] j \, < \, r \, \leq \, 0}}
\bigg]
\sum_{\substack{s \, \in \zzz \\[1mm] 0 \, \leq \, s \, \leq \, m}}
(-1)^j \, 
q^{(m+\frac12)(j+\frac{1}{4(m+\frac12)})^2} \, 
q^{-\frac{(2mr-s)^2}{4m}} 
\nonumber
\\[3mm]
& & 
\times \,\ 
e^{2\pi im(j+\frac{1}{4(m+\frac12)})(2p+1)\tau} \, 
e^{-\pi i(2mr-s)(2p+1)\tau} \, 
\big[\theta_{s,m}+\theta_{-s,m}\big](\tau, z) \hspace{20mm}
\nonumber
\\[3mm]
& & \hspace{-5mm}
- \,\ \bigg[
\sum_{\substack{j, \, r \, \in \zzz \\[1mm] 0 \, \leq \, r \, < \, j}}
-
\sum_{\substack{j, \, r \, \in \zzz \\[1mm] j \, \leq \, r \, < \, 0}}
\bigg]
\sum_{\substack{s \, \in \zzz \\[1mm] 0 \, < \, s \, < \, m}}
(-1)^j \, 
q^{(m+\frac12)(j+\frac{1}{4(m+\frac12)})^2} \, 
q^{-\frac{(2mr+s)^2}{4m}} 
\nonumber
\\[3mm]
& & 
\times \,\ 
e^{2\pi im(j+\frac{1}{4(m+\frac12)})(2p+1)\tau} \, 
e^{-\pi i(2mr+s)(2p+1)\tau} \, 
\big[\theta_{s,m}+\theta_{-s,m}\big](\tau, z)
\nonumber
\\[3mm]
&=& i \, \eta(\tau)^3 \, \times \, {\rm (I)}
\nonumber
\\[3mm]
& & \hspace{-5mm}
- \,\ \bigg[
\sum_{\substack{j, \, r \, \in \zzz \\[1mm] 0 \, < \, r \, \leq \, j}}
-
\sum_{\substack{j, \, r \, \in \zzz \\[1mm] j \, < \, r \, \leq \, 0}}
\bigg]
\sum_{\substack{s \, \in \zzz \\[1mm] 0 \, \leq \, s \, \leq \, m}}
(-1)^j \, 
q^{(m+\frac12)(j+\frac{1}{4(m+\frac12)})^2 
\, + \, m(j+\frac{1}{4(m+\frac12)})(2p+1)}
\nonumber
\\[2mm]
& &
\times \,\ 
q^{-\frac{(2mr-s)^2}{4m}-\frac12(2mr-s)(2p+1)} \, 
\big[\theta_{s,m}+\theta_{-s,m}\big](\tau, z) \hspace{20mm}
\nonumber
\\[3mm]
& & \hspace{-5mm}
- \,\ \bigg[
\sum_{\substack{j, \, r \, \in \zzz \\[1mm] 0 \, \leq \, r \, < \, j}}
-
\sum_{\substack{j, \, r \, \in \zzz \\[1mm] j \, \leq \, r \, < \, 0}}
\bigg]
\sum_{\substack{s \, \in \zzz \\[1mm] 0 \, < \, s \, < \, m}}
(-1)^j \, 
q^{(m+\frac12)(j+\frac{1}{4(m+\frac12)})^2
\, + \, m(j+\frac{1}{4(m+\frac12)})(2p+1)}
\nonumber
\\[3mm]
& & 
\times \,\ 
q^{-\frac{(2mr+s)^2}{4m}-\frac12(2mr+s)(2p+1)} \, 
\big[\theta_{s,m}+\theta_{-s,m}\big](\tau, z)
\label{m1:eqn:2022-918e1}
\end{eqnarray}}
The LHS of this equation \eqref{m1:eqn:2022-918e1} becomes by 
Note \ref{m1:note:2022-917b} as follows
$$
\text{LHS of \eqref{m1:eqn:2022-918e1}} \, = \, 
q^{-\frac{m^2}{2(2m+1)}(2p+1)^2}
\theta_{2mp+m+\frac12, m+\frac12}^{(-)}(\tau,0) \, 
\Phi^{[m,0] \, \ast}(\tau, z_1, z_2, 0) 
$$
Also (I) is computed by using Notes \ref{m1:note:2022-917c}
and \ref{m1:note:2022-917d} as follows:
{\allowdisplaybreaks
\begin{eqnarray*}
& & \hspace{-9mm}
{\rm (I)} \, = \, - \, \frac{ \displaystyle 
\theta_{-\frac12, \, m+\frac12}^{(-)}
\Big(\tau, \,\ z+\frac{(2p+1)\tau}{2m+1}\Big)}{\displaystyle 
\vartheta_{11}\Big(\tau, \, \frac{z}{2}+\frac{(2p+1)\tau}{2}\Big)}
\, + \, 
\frac{\displaystyle 
\theta_{\frac12, \, m+\frac12}^{(-)}
\Big(\tau, \,\ z-\frac{(2p+1)\tau}{2m+1}\Big)}{\displaystyle 
\vartheta_{11}\Big(\tau, \, \frac{z}{2}-\frac{(2p+1)\tau}{2}\Big)}
\\[3mm]
&=&
- \,\ \frac{
q^{-\frac{1}{16(m+\frac12)}(2p+1)^2}
e^{-\frac{\pi i}{2}(2p+1)z}
\theta_{p, m+\frac12}^{(-)}(\tau,0)
}{- i \, (-1)^p 
q^{-\frac18(2p+1)^2} e^{-\frac{\pi i}{2}(2p+1)z}
\theta_{0,\frac12}^{(-)}(\tau,z)}
+ \frac{
q^{-\frac{1}{16(m+\frac12)}(2p+1)^2}
e^{\frac{\pi i}{2}(2p+1)z}
\theta_{-p, m+\frac12}^{(-)}(\tau,0)
}{i \, (-1)^p 
q^{-\frac18(2p+1)^2} e^{\frac{\pi i}{2}(2p+1)z}
\theta_{0,\frac12}^{(-)}(\tau,z)}
\\[2mm]
&=&
- \, i \, (-1)^p \, 
q^{\frac{m}{8(m+\frac12)}(2p+1)^2} \, 
\frac{1}{\theta_{0,\frac12}^{(-)}(\tau,z)} \cdot 
\big[\theta_{p, m+\frac12}^{(-)}+ \theta_{-p, m+\frac12}^{(-)}\big](\tau,0)
\end{eqnarray*}}
Then substituting these into \eqref{m1:eqn:2022-918e1} and rewriting 
the 2nd term in the RHS of \eqref{m1:eqn:2022-918e1} by using 
\eqref{m1:eqn:2022-918c}, the above formula \eqref{m1:eqn:2022-918e1}
becomes as follows:
{\allowdisplaybreaks
\begin{eqnarray}
& & \hspace{-10mm}
q^{-\frac{m^2}{2(2m+1)}(2p+1)^2}
\theta_{2mp+m+\frac12, m+\frac12}^{(-)}(\tau,0) \, 
\Phi^{[m,0] \, \ast}(\tau, z_1, z_2, 0) 
\nonumber
\\[3mm]
&=&
(-1)^p \,  
q^{\frac{m}{8(m+\frac12)}(2p+1)^2} \, 
\frac{\eta(\tau)^3}{\theta_{0,\frac12}^{(-)}(\tau,z)} \cdot 
\big[\theta_{p, m+\frac12}^{(-)}+ \theta_{-p, m+\frac12}^{(-)}\big](\tau,0)
\nonumber
\\[3mm]
& & \hspace{-5mm}
- \,\ \bigg[
\sum_{\substack{j, \, r \, \in \zzz \\[1mm] 0 \, \leq \, r \, \leq \, j}}
-
\sum_{\substack{j, \, r \, \in \zzz \\[1mm]j \, < \, r \, < \, 0}}
\bigg]
\sum_{\substack{s \, \in \zzz \\[1mm]0 \, \leq \, s \, \leq \, m}}
(-1)^j \, 
q^{(m+\frac12)(j+ \frac12+\frac{2pm}{2m+1})^2
\, - \, \frac{m^2}{4(m+\frac12)}(2p+1)^2}
\nonumber
\\[2mm]
& &
\times \,\ 
q^{-\frac{1}{4m} \, [2m(r+\frac12)-s+2mp]^2
\, + \, \frac{m}{4}(2p+1)^2} \, 
\big[\theta_{s,m}+\theta_{-s,m}\big](\tau, z) \hspace{20mm}
\nonumber
\\[2mm]
& & \hspace{-5mm}
+ \,\
\underbrace{\sum_{j \in \zzz} (-1)^j \, 
q^{(m+\frac12)(j+ \frac12+\frac{2pm}{2m+1})^2}}_{
\substack{|| \\[-1mm] {\displaystyle 
\theta_{2mp+m+\frac12, m+\frac12}^{(-)}(\tau,0)
}}} \,
q^{ - \, \frac{m^2}{4(m+\frac12)}(2p+1)^2}
\nonumber
\\[2mm]
& &
\times \,\ 
\sum_{\substack{s \, \in \zzz \\[1mm] 0 \, \leq \, s \, \leq \, m}}
q^{-\frac{1}{4m}s^2+\frac{s}{2}(2p+1)} \, 
\big[\theta_{s,m}+\theta_{-s,m}\big](\tau, z)
\nonumber
\\[3mm]
& & \hspace{-5mm}
- \,\ \bigg[
\sum_{\substack{j, \, r \, \in \zzz \\[1mm] 0 \, \leq \, r \, < \, j}}
-
\sum_{\substack{j, \, r \, \in \zzz \\[1mm] j \, \leq \, r \, < \, 0}}
\bigg]
\sum_{\substack{s \, \in \zzz \\[1mm] 0 \, < \, s \, < \, m}}
(-1)^j \, 
q^{(m+\frac12)(j+ \frac12+\frac{2pm}{2m+1})^2
\, - \, \frac{m^2}{2(2m+1)}(2p+1)^2}
\nonumber
\\[3mm]
& & 
\times \,\ 
q^{-\frac{1}{4m} \, [2m(r+\frac12)+s+2mp]^2
\, + \, \frac{m}{4}(2p+1)^2} \, 
\big[\theta_{s,m}+\theta_{-s,m}\big](\tau, z)
\label{m1:eqn:2022-918e2}
\end{eqnarray}}
Multiplying $q^{\frac{m^2}{2(2m+1)}(2p+1)^2}$ and replacing 
$\sum_{j,r \in \zzz}$ with $\sum_{j', r' \in \frac12 \zzz_{\rm odd}}$
by putting $j+\frac12=j'$ and $r+\frac12=r'$ in the above formula 
\eqref{m1:eqn:2022-918e2}, we obtain
{\allowdisplaybreaks
\begin{eqnarray*}
& & \hspace{-7mm}
\theta_{2mp+m+\frac12, m+\frac12}^{(-)}(\tau,0) \, 
\Phi^{[m,0] \, \ast}(\tau, z_1, z_2, 0) 
\\[2mm]
&=&
(-1)^p \, q^{\frac{m}{4}(2p+1)^2} \, 
\frac{\eta(\tau)^3}{\theta_{0,\frac12}^{(-)}(\tau,z)} \cdot 
\big[\theta_{p, m+\frac12}^{(-)}+ \theta_{-p, m+\frac12}^{(-)}\big](\tau,0)
\\[3mm]
& & \hspace{-5mm}
- \,\ \bigg[
\sum_{\substack{j', \, r' \, \in \, \frac12 \, \zzz_{\rm odd} \\[1mm]
0 \, \leq \, r' \, < \, j'}}
-
\sum_{\substack{j', \, r' \, \in \, \frac12 \, \zzz_{\rm odd} \\[1mm]
j' \, \leq \, r' \, < \, 0}}
\bigg]
\sum_{\substack{k \, \in \zzz \\[1mm] 0 \, < \, k \, < \, m}}
(-1)^{j'-\frac12} \, 
q^{(m+\frac12)(j'+\frac{2pm}{2m+1})^2}
\\[-3mm]
& & \hspace{50mm}
\times \,\ 
q^{-\frac{1}{4m} \, [2mr'+k+2mp]^2
\, + \, \frac{m}{4}(2p+1)^2} \, 
\big[\theta_{k,m}+\theta_{-k,m}\big](\tau, z) \hspace{20mm}
\\[3mm]
& & \hspace{-5mm}
- \,\ \bigg[
\sum_{\substack{j', \, r' \, \in \, \frac12 \, \zzz_{\rm odd} \\[1mm]
0 \, \leq \, r' \, \leq \, j'
}}
-
\sum_{\substack{j', \, r' \, \in \, \frac12 \, \zzz_{\rm odd} \\[1mm]
j' \, < \, r' \, < \, 0}}\bigg]
\sum_{\substack{k \, \in \zzz \\[1mm] 0 \, \leq \, k \, \leq \, m}}
(-1)^{j'-\frac12} \, 
q^{(m+\frac12)(j'+\frac{2pm}{2m+1})^2}
\\[-3mm]
& & \hspace{50mm}
\times \,\ 
q^{-\frac{1}{4m} \, [2mr'-k+2mp]^2
\, + \, \frac{m}{4}(2p+1)^2} \, 
\big[\theta_{k,m}+\theta_{-k,m}\big](\tau, z)
\\[3mm]
& & \hspace{-5mm}
+ \,\ 
\underbrace{\sum_{j \in \zzz} (-1)^j \, 
q^{(m+\tfrac12)(j+ \frac{m+\frac12 +2pm}{2m+1})^2}
}_{\substack{|| \\[-2mm] {\displaystyle 
\theta_{2mp+m+\frac12, m+\frac12}^{(-)}(\tau,0)
}}}
\sum_{\substack{j, \, r \, \in \zzz \\[1mm] 0 \, \leq \, k \, \leq \, m}} 
q^{-\frac{1}{4m}k^2+\frac{k}{2}(2p+1)} \, 
\big[\theta_{k,m}+\theta_{-k,m}\big](\tau, z)
\end{eqnarray*}}
proving Lemma \ref{m1:lemma:2022-918b}.
\end{proof}

\section{$\Phi^{[m,0] \, \ast}_{\rm add}(\tau, z_1,z_2,t) \, \sim \, $ 
the case $z_1-z_2=2a\tau+2b$}
\label{sec:Phi;ast:add:(m0):z1-z2=2atau}

\begin{lemma} 
\label{m1:lemma:2022-921b}
Let $m \in \frac12 \nnn$, $j, a \in \frac12 \zzz$ and $b \in \qqq$ 
such that $4mb \in \zzz$. Then the functions 
$P_{j,m}(\tau, z):= P^{(+)}_{j,m}(\tau, z)$ 
and $Q_{j,m}(\tau, z):= Q^{(+)}_{j,m}(\tau, z)$ defined by the formulas 
(4.1a) and (4.1b) in \cite{W2022c} satisfy the following:
\begin{enumerate}
\item[{\rm 1)}] \,\ $P_{j,m}(\tau, a\tau+b) \, + \, 
e^{4\pi ijb} \, e^{8\pi imab} \, P_{-j,m}(\tau, a\tau+b)
\,\ = \,\ 0$
\item[{\rm 2)}] \,\ $Q_{j,m}(\tau, a\tau+b)
\,\ + \,\ 
e^{4\pi ijb} \, e^{8\pi imab} Q_{-j,m}(\tau, a\tau+b)$
{\allowdisplaybreaks
\begin{eqnarray*}
&=&
e^{2\pi ijb} \, e^{8\pi imab}
\sum_{\substack{k , \in \zzz \\[1mm] 0 \, \leq \, k \, \leq \, 2a}} 
e^{4\pi imbk} q^{- \frac{1}{4m} (j+2m(2a-k)) (j-2mk)}
\\[1mm]
& &
+ \,\ e^{2\pi ijb}  
\sum_{\substack{k \, \in \zzz \\[1mm] 0 \, \leq \, k \, \leq \, 2a}} 
e^{4\pi imbk} q^{- \frac{1}{4m} (j-2m(2a-k)) (j+2mk)}
\end{eqnarray*}}
\end{enumerate}
\end{lemma}

\begin{proof} The claim 1) follows from (4.4) in \cite{W2022c} and 
the claim 2) follows from (4.6b) in \cite{W2022c},
since $e^{-8\pi imab}= e^{8\pi imab}$.
\end{proof}

\vspace{1mm}

\begin{lemma} 
\label{m1:lemma:2022-921a}
Let $m \in \nnn$, $a,b \in \frac12 \zzz$ and $j \in \zzz$. Then 
\begin{enumerate}
\item[{\rm 1)}] \,\ $R_{j,m}(\tau, \, a\tau+b)
+ R_{2m-j,m}(\tau, \, a\tau+b)
\,\ = \,\ 
2 \, e^{2\pi ijb} \, 
\sum\limits_{\substack{k \, \in \zzz \\[1mm] 1 \, \leq \, k \, \leq \, 2a}} 
\hspace{-3mm}
q^{- \frac{1}{4m} (j+2m(2a-k)) (j-2mk)}$
\item[{\rm 2)}] \,\ $R_{0,m}(\tau, \, a\tau+b) \,\ = \,\ 
\sum\limits_{\substack{k \, \in \zzz \\[1mm] 0 \, \leq \, k \, \leq \, 2a}}
q^{mk(2a-k)}$
\end{enumerate}
\end{lemma}

\begin{proof} The claim 1) is obtained immediately 
from Lemma 4.1 in \cite{W2022c}.

\medskip

\noindent
To prove the claim 2) we note, by letting $j=0$ in Lemma 
\ref{m1:lemma:2022-921b}, that
$$
P_{0,m}(\tau, \, a\tau+b) \,\ = \,\ 0
$$
and that 
{\allowdisplaybreaks
\begin{eqnarray*}
Q_{0,m}(\tau, \, a\tau+b) 
&=& \frac12
\sum_{\substack{k \, \in \zzz \\[1mm] 0 \, \leq \, k \, \leq \, 2a}} 
\Big\{
\underbrace{q^{-\frac{1}{4m} \cdot 2m (2a-k)(-2mk)}
\, + \, 
q^{\frac{1}{4m} \cdot 2m (2a-k) \, 2mk}}_{
\substack{|| \\[0mm] {\displaystyle 2 \, q^{mk(2a-k)}
}}}\Big\}
\\[-1mm]
&=&
\sum_{\substack{k \, \in \zzz \\[1mm] 0 \, \leq \, k \, \leq \, 2a}}
q^{mk(2a-k)}
\end{eqnarray*}}
Thus we have
$$
R_{0,m}(\tau, \, a\tau+b) \, = \, 
P_{0,m}(\tau, \, a\tau+b) + Q_{0,m}(\tau, \, a\tau+b) 
\,\ = 
\sum_{\substack{k \, \in \zzz \\[1mm] 0 \, \leq \, k \, \leq \, 2a}}
q^{mk(2a-k)}
$$
proving Lemma \ref{m1:lemma:2022-921a}.
\end{proof}

Using the above Lemmas \ref{m1:lemma:2022-921b} and
\ref{m1:lemma:2022-921a}, the Zwegers' additional function 
$$
\Phi^{[m,0] \, \ast}_{\rm add}(\tau, z_1, z_2,0) 
\, := \, 
\Phi^{[m,0]}_{1, \, {\rm add}}(\tau, z_1, z_2,0) 
+
\Phi^{[m,0]}_{2, \, {\rm add}}(\tau, z_1, z_2,0) 
$$ 
is obtained as follows:

\vspace{1mm}

\begin{prop} 
\label{m1:prop:2022-921a}
Let $m \in \nnn$ and $a, b \in \frac12 \zzz$. 
Then, for $z_1$ and $z_2$ satisfying $z_1-z_2=2a\tau+2b$, 
the correction function \, 
$\Phi^{[m,0] \, \ast}_{\rm add}(\tau, z_1, z_2,0) $ \, 
is given by the following formula:
{\allowdisplaybreaks
\begin{eqnarray}
& & \hspace{-15mm}
\Phi^{[m,0] \, \ast}_{\rm add}(\tau, z_1, z_2,0) 
\,\ = \,\ 
- \, 
\sum_{\substack{k \, \in \zzz \\[1mm] 0 \, \leq \, k \, \leq \, 2a}}
q^{mk(2a-k)} \, 
\theta_{0,m}(\tau, z_1+z_2)
\nonumber
\\[2mm]
& &
- \,\ \frac12 \hspace{-3mm}
\sum_{\substack{j \in \zzz \\[1mm] 0 \, <  \, j \, < \, 2m }} \hspace{-3mm} 
e^{2\pi ijb} \hspace{-2mm}
\sum_{\substack{k \, \in \zzz \\[1mm] 1 \, \leq \, k \, \leq \, 2a}} \hspace{-3mm}
q^{- \frac{1}{4m} (j+2m(2a-k)) (j-2mk)}
\big[\theta_{j,m}+\theta_{-j,m}\big](\tau, z_1+z_2)
\label{m1:eqn:2022-921a}
\end{eqnarray}}
\end{prop}

\begin{proof} By the formula for $\Phi^{[m,0]}_{i, \, {\rm add}}$ in 
\cite{W2022c}, we have
{\allowdisplaybreaks
\begin{eqnarray*}
& & \hspace{-8mm}
\Phi^{[m,0] \, \ast}_{\rm add}(\tau, z_1, z_2,0) 
\, = \, - \, 
\dfrac{1}{2} \hspace{-5mm}
\underbrace{\sum_{\substack{j \in \zzz \\[1mm]
0 \, \leq  \, j \, < \, 2m }}}_{\substack{|| \\[-0.5mm] 
{\displaystyle 
\sum_{j=0} \,\ +
\sum_{\substack{j \in \zzz \\[1mm] 0 \, <  \, j \, < \, 2m }}
}}}  \hspace{-7mm}
R_{j,m}(\tau, a\tau+b)
\, \big[\theta_{j,m}+\theta_{-j,m}\big](\tau, z_1+z_2)
\\[2mm]
&=& - \, \frac12 \hspace{-25mm} 
\underbrace{R_{0,m}(\tau, a\tau+b)}_{\substack{|| \\[1mm] 
{\displaystyle \hspace{27mm}
\sum_{\substack{k \, \in \zzz \\[1mm] 0 \, \leq \, k \, \leq \, 2a}}
q^{mk(2a-k)} \quad 
\text{by Lemma \ref{m1:lemma:2022-921a}}
}}} \hspace{-25mm}
\times \,\ 2 \, \theta_{0,m}(\tau, z_1+z_2)
\\[3mm]
& &
- \, \frac{1}{4} \hspace{-3mm}
\sum\limits_{\substack{j \in \zzz \\[1mm]
0 \, <  \, j \, < \, 2m }}  R_{j,m}(\tau, a\tau+b)
\, \big[\theta_{j,m}+\theta_{-j,m}\big](\tau, z_1+z_2)
\\[3mm]
& &
- \,\ \frac14 \hspace{-3mm}
\sum_{\substack{j \in \zzz \\[1mm]
0 \, <  \, j \, < \, 2m }}  R_{2m-j,m}(\tau, a\tau+b)
\, \big[
\underbrace{\theta_{2m-j,m}+\theta_{-(2m-j),m}}_{\substack{|| 
\\[0mm] {\displaystyle \theta_{j,m}+\theta_{-j,m}
}}}
\big](\tau, z_1+z_2)
\\[3mm]
&=&
- \, 
\sum_{\substack{k \, \in \zzz \\[1mm] 0 \, \leq \, k \, \leq \, 2a}}
q^{mk(2a-k)} \, 
\theta_{0,m}(\tau, z_1+z_2)
\\[2mm]
& &
- \, \frac14 \hspace{-3mm}
\sum_{\substack{j \in \zzz \\[1mm] 0 \, <  \, j \, < \, 2m }}  
\big\{
R_{2m-j,m}(\tau, a\tau+b)+ R_{j,m}(\tau, a\tau+b)
\big\}
\, \big[\theta_{j,m}+\theta_{-j,m}\big](\tau, z_1+z_2)
\end{eqnarray*}}
Then, using Lemma \ref{m1:lemma:2022-921a}, this is rewritten as follows:
{\allowdisplaybreaks
\begin{eqnarray*}
&=&
- \, 
\sum_{\substack{k \, \in \zzz \\[1mm] 0 \, \leq \, k \, \leq \, 2a}}
q^{mk(2a-k)} \, 
\theta_{0,m}(\tau, z_1+z_2)
\\[2mm]
& &
- \, \frac14 \hspace{-3mm}
\sum_{\substack{j \in \zzz \\[1mm] 0 \, <  \, j \, < \, 2m }}  
2 \, e^{2\pi ijb} 
\sum_{\substack{k \, \in \zzz \\[1mm] 1 \, \leq \, k \, \leq \, 2a}} 
\hspace{-3mm}
q^{- \frac{1}{4m} (j+2m(2a-k)) (j-2mk)}
\big[\theta_{j,m}+\theta_{-j,m}\big](\tau, z_1+z_2)
\\[3mm]
&=&
- \, 
\sum_{\substack{k \, \in \zzz \\[1mm] 0 \, \leq \, k \, \leq \, 2a}}
q^{mk(2a-k)} \, 
\theta_{0,m}(\tau, z_1+z_2)
\\[2mm]
& &
- \, \frac12 \hspace{-3mm}
\sum_{\substack{j \in \zzz \\[1mm] 0 \, <  \, j \, < \, 2m }} \hspace{-3mm} 
e^{2\pi ijb} \hspace{-2mm}
\sum_{\substack{k \, \in \zzz \\[1mm] 1 \, \leq \, k \, \leq \, 2a}} 
\hspace{-3mm}
q^{- \frac{1}{4m} (j+2m(2a-k)) (j-2mk)}
\big[\theta_{j,m}+\theta_{-j,m}\big](\tau, z_1+z_2)
\end{eqnarray*}}
proving Proposition \ref{m1:prop:2022-921a}.
\end{proof}

Using the above Proposition \ref{m1:prop:2022-921a}, the Zwegers's 
additional functions $\Phi^{[m,0] \, \ast}_{\rm add}(\tau, z_1, z_2,0)$
for $(z_1, z_2) = \frac{z}{2}+\frac{\tau}{2}-\frac12, \, 
\frac{z}{2}-\frac{\tau}{2}+\frac12)$ and 
$(z_1, z_2) = (\frac{z}{2}+\frac{\tau}{2}, \, 
\frac{z}{2}-\frac{\tau}{2})$ are obtained as follows:

\vspace{1mm}

\begin{lemma} 
\label{m1:lemma:2022-921c}
For $m \in \nnn$, the following formulas hold:
\begin{enumerate}
\item[{\rm 1)}] $\Phi^{[m,0] \, \ast}_{\rm add}\Big(\tau, \, 
\dfrac{z}{2}+\dfrac{\tau}{2}-\dfrac12, \, 
\dfrac{z}{2}-\dfrac{\tau}{2}+\dfrac12, \, 0\Big)
\, = \, 
- 
\sum\limits_{\substack{j \in \zzz \\[1mm] 0 \, \leq  \, j \, \leq \, 2m }} 
\hspace{-2mm}
(-1)^j \, 
q^{-\frac{1}{4m}j(j-2m)} \theta_{j,m}(\tau, z)$
\item[{\rm 2)}] $\Phi^{[m,0] \, \ast}_{\rm add}\Big(\tau, \, 
\dfrac{z}{2}+\dfrac{\tau}{2}, \, 
\dfrac{z}{2}-\dfrac{\tau}{2}, \, 0\Big)
\, = \, 
- 
\sum\limits_{\substack{j \in \zzz \\[1mm] 0 \, \leq  \, j \, \leq \, 2m }} 
\hspace{-2mm}
q^{-\frac{1}{4m}j(j-2m)} \theta_{j,m}(\tau, z)$
\end{enumerate}
\end{lemma}

\begin{proof} 1) \,\ In the case
$\left\{
\begin{array}{lcl}
z_1 &=& \frac{z}{2}+\frac{\tau}{2}-\frac12 \\[0mm]
z_2 &=&\frac{z}{2}-\frac{\tau}{2}+\frac12
\end{array}\right. $, letting $\left\{
\begin{array}{ccr}
2a &=& 1 \\[0mm]
2b &=& -1
\end{array}\right. $ in \eqref{m1:eqn:2022-921a}, we have
{\allowdisplaybreaks
\begin{eqnarray*}
& & \hspace{-10mm}
\Phi^{[m,0] \, \ast}_{\rm add}(\tau, z_1, z_2,0) 
= \,\ 
- \, \underbrace{
\sum_{\substack{k \, \in \zzz \\[1mm] 0 \, \leq \, k \, \leq \, 1}}
q^{mk(1-k)}}_{2} 
\, \theta_{0,m}(\tau, z_1+z_2)
\\[2mm]
& &
- \,\ \frac12 \hspace{-3mm}
\sum_{\substack{j \in \zzz \\[1mm] 0 \, <  \, j \, < \, 2m }} \,\ 
e^{-\pi ij} 
\sum_{\substack{k \, \in \zzz \\[1mm] 1 \, \leq \, k \, \leq \, 1}} 
\hspace{-3mm}
q^{- \frac{1}{4m} (j+2m(1-k)) (j-2mk)}
\big[\theta_{j,m}+\theta_{-j,m}\big](\tau, z)
\\[2mm]
&=&
- \, 2 \, \theta_{0,m}(\tau, z_1+z_2) 
\,\ - \,\ 
\frac12 \sum_{\substack{j \in \zzz \\[1mm] 0 \, <  \, j \, < \, 2m }}
(-1)^j \, 
q^{-\frac{1}{4m}j(j-2m)} 
\big[\theta_{j,m}+\theta_{-j,m}\big](\tau, z)
\\[2mm]
&=&
- \, 2 \theta_{0,m}(\tau, z) 
- 
\frac12 \hspace{-3mm}
\sum_{\substack{j \in \zzz \\[1mm] 0 \, <  \, j \, < \, 2m }} \hspace{-3mm}
(-1)^j 
q^{-\frac{1}{4m}j(j-2m)} \theta_{j,m}(\tau, z)
- 
\frac12 \hspace{-4mm}
\underbrace{\sum_{\substack{j \in \zzz \\[1mm] 0 \, <  \, j \, < \, 2m }}
\hspace{-3mm}
(-1)^j q^{-\frac{1}{4m}j(j-2m)} \theta_{-j,m}(\tau, z)}_{
\substack{\hspace{19mm}
|| \,\ \leftarrow \,\ 2m-j \, = \, j' \\[1mm] {\displaystyle 
\sum_{\substack{j \in \zzz \\[1mm] 0 \, <  \, j' \, < \, 2m }} \hspace{-3mm}
(-1)^{j'} 
q^{-\frac{1}{4m}j'(j'-2m)} \theta_{j',m}(\tau, z)
}}}
\\[0mm]
&=&
- \, 2 \, \theta_{0,m}(\tau, z) 
\,\ - \, 
\sum_{\substack{j \in \zzz \\[1mm] 0 \, <  \, j \, < \, 2m }} 
(-1)^j \, 
q^{-\frac{1}{4m}j(j-2m)} \theta_{j,m}(\tau, z)
\end{eqnarray*}}
proving 1).

\medskip

\noindent
2) \,\ In the case $\left\{
\begin{array}{lcl}
z_1 &=& \frac{z}{2}+\frac{\tau}{2} \\[0mm]
z_2 &=&\frac{z}{2}-\frac{\tau}{2}
\end{array}\right. $, letting $\left\{
\begin{array}{ccr}
2a &=& 1 \\[0mm]
2b &=& 0
\end{array}\right. $ in \eqref{m1:eqn:2022-921a}, we have 
{\allowdisplaybreaks
\begin{eqnarray*}
& & \hspace{-10mm}
\Phi^{[m,0] \, \ast}_{\rm add}(\tau, z_1, z_2,0) 
= \, 
- \, \underbrace{
\sum_{\substack{k \, \in \zzz \\[1mm] 0 \, \leq \, k \, \leq \, 1}}
q^{mk(1-k)}}_{2} 
\, \theta_{0,m}(\tau, z_1+z_2)
\\[2mm]
& &
- \,\ \frac12 \hspace{-3mm}
\sum_{\substack{j \in \zzz \\[1mm] 0 \, <  \, j \, < \, 2m }} 
\underbrace{
\sum_{\substack{k \, \in \zzz \\[1mm] 1 \, \leq \, k \, \leq \, 1}} 
\hspace{-3mm}
q^{- \frac{1}{4m} (j+2m(1-k)) (j-2mk)}}_{\substack{|| \\[0mm] 
{\displaystyle q^{-\frac{1}{4m}j(j-2m)}
}}}
\big[\theta_{j,m}+\theta_{-j,m}\big](\tau, z)
\\[2mm]
&=&
- \, 2 \, \theta_{0,m}(\tau, z_1+z_2) 
\,\ - \,\ 
\frac12 \sum_{\substack{j \in \zzz \\[1mm] 0 \, <  \, j \, < \, 2m }}
q^{-\frac{1}{4m}j(j-2m)} 
\big[\theta_{j,m}+\theta_{-j,m}\big](\tau, z)
\\[2mm]
&=&
- \, 2 \, \theta_{0,m}(\tau, z) 
- 
\frac12 \hspace{-3mm}
\sum_{\substack{j \in \zzz \\[1mm] 0 \, <  \, j \, < \, 2m }} \hspace{-3mm}
q^{-\frac{1}{4m}j(j-2m)} \theta_{j,m}(\tau, z)
- 
\frac12 \hspace{-3mm}
\underbrace{\sum_{\substack{j \in \zzz \\[1mm] 0 \, <  \, j \, < \, 2m }}
\hspace{-3mm}
q^{-\frac{1}{4m}j(j-2m)} \theta_{-j,m}(\tau, z)}_{
\substack{\hspace{19mm}
|| \,\ \leftarrow \,\ 2m-j \, = \, j' \\[1mm] {\displaystyle 
\sum_{\substack{j \in \zzz \\[1mm] 0 \, <  \, j' \, < \, 2m }}
\hspace{-3mm}
q^{-\frac{1}{4m}j'(j'-2m)} \theta_{j',m}(\tau, z)
}}}
\\[0mm]
&=&
- \, 2 \, \theta_{0,m}(\tau, z) 
\,\ - \, 
\sum_{\substack{j \in \zzz \\[1mm] 0 \, <  \, j \, < \, 2m }} 
q^{-\frac{1}{4m}j(j-2m)} \theta_{j,m}(\tau, z)
\end{eqnarray*}}
proving 2).
\end{proof}

\section{$\Phi^{[m,0] \, \ast}(\tau, z_1+p\tau, z_2-p\tau, 0)$}
\label{sec:Phi(+:0)ast:(z1+ptau:z2-ptau)}

\subsection{$\Phi^{[m,0] \, \ast}(\tau, z_1+p\tau, z_2-p\tau,t)$ 
$\sim$ the case $z_1-z_2=\tau-1$}
\label{subsec:Phi(+:0)ast(z1+ptau:z2-ptau):z1-z2=(1+2p)tau-1}

\medskip

\begin{lemma} 
\label{m1:lemma:2022-919a}
For $m \in \nnn$ and $p \in \zzz$, the following formula holds:
{\allowdisplaybreaks
\begin{eqnarray}
& & \hspace{-10mm}
\Phi^{[m,0] \, \ast}\Big(\tau, \,\ 
\frac{z}{2}+\dfrac{\tau}{2}-\dfrac12+p\tau, \,\ 
\frac{z}{2}-\dfrac{\tau}{2}+\dfrac12-p\tau, \,\ 0\Big)
\nonumber
\\[2mm]
&=&
q^{mp(p+1)} \, \bigg\{
\widetilde{\Phi}^{[m,0] \, \ast}
\Big(\tau, \,\ 
\frac{z}{2}+\frac{\tau}{2}-\frac12, \,\ 
\frac{z}{2}-\frac{\tau}{2}+\frac12, \,\ 0\Big)
\nonumber
\\[2mm]
& & \hspace{-5mm}
+ \sum_{\substack{r \, \in \zzz \\[1mm] -p \, \leq \, r \, \leq \, p}} \,
\sum_{\substack{k \, \in \zzz \\[1mm] 0 \, \leq \, k \, < \, m}} 
(-1)^k \, q^{-\frac{1}{4m}(2mr+k)(2m(r-1)+k)} 
\big[\theta_{k,m}+\theta_{-k,m}\big](\tau, z)
\nonumber
\\[3mm]
& & \hspace{-5mm}
- \,\ 2 \hspace{-2mm}
\sum_{\substack{r \, \in \zzz \\[1mm] 0 \, \leq \, r \, \leq \, p-1}} \hspace{-3mm}
q^{-mr(r+1)} \theta_{0,m}(\tau, z)
\, + \, 
(-1)^m \hspace{-3mm}
\sum_{\substack{r \, \in \zzz \\[1mm] -p \, \leq \, r \, \leq \, p}} \hspace{-3mm}
q^{-\frac{m}{4}(2r+1)(2r-1)} \, \theta_{m,m}(\tau, z)
\bigg\}
\label{m1:eqn:2022-919a}
\end{eqnarray}}
\end{lemma}

\begin{proof} Letting $\left\{
\begin{array}{lcc}
z_1 &=&\frac{z}{2}+\frac{\tau}{2}-\frac12 \\[1mm]
z_2 &=&\frac{z}{2}-\frac{\tau}{2}+\frac12
\end{array} \right. $ namely
$\left\{
\begin{array}{ccc}
z_1-z_2 &=& \tau-1 \\[0mm]
z_1+z_2 &=& z
\end{array}\right. $ in the formula \eqref{m1:eqn:2022-917b}
in Lemma \ref{m1:lemma:2022-917b}, we have 
{\allowdisplaybreaks
\begin{eqnarray}
& & \hspace{-5mm}
\Phi^{[m,0] \, \ast}(\tau, \, z_1+p\tau, \, z_2-p\tau, \, 0)
\nonumber
\\[2mm]
&=&
e^{2\pi imp(\tau-1)}
q^{mp^2} \bigg\{
\Phi^{[m,0] \, \ast}(\tau, z_1, z_2,0)
+ \sum_{\substack{k \, \in \zzz \\[1mm] 1 \, \leq \, k \, \leq \, 2pm}} 
e^{-\pi ik(\tau-1)} q^{-\frac{1}{4m}k^2} 
\big[\theta_{k,m}+\theta_{-k,m}\big](\tau, z)\bigg\}
\nonumber
\\[2mm]
&=&
q^{mp(p+1)} \bigg\{
\Phi^{[m,0] \, \ast}(\tau, z_1, z_2,0)
+ \hspace{-2mm}
\underbrace{
\sum_{\substack{k \, \in \zzz \\[1mm] 1 \, \leq \, k \, \leq \, 2pm}} 
\hspace{-4mm}
(-1)^k q^{-\frac{1}{4m}(k^2+2mk)}
\big[\theta_{k,m}+\theta_{-k,m}\big](\tau, z)}_{\rm (I)}
\bigg\}
\label{m1:eqn:2022-919b}
\end{eqnarray}}
We compute (I) by putting $k=2mr+k'$ \,\ 
$(0 \leq r <p, \,\ 1 \leq k' \leq 2m)$ \, as follows:
{\allowdisplaybreaks
\begin{eqnarray*}
& & \hspace{-10mm}
{\rm (I)} =
\sum_{\substack{r \, \in \zzz \\[1mm] 0 \leq r <p}}
\sum_{\substack{k' \, \in \zzz \\[1mm] 1 \, \leq \, k' \, \leq \, 2m}} 
(-1)^{2mr+k'} q^{-\frac{1}{4m}(2mr+k')(2mr+k'+2m)}
\big[\theta_{2mr+k',m}+\theta_{-(2mr+k'),m}\big](\tau, z)
\\[2mm]
&=&
\underbrace{
\sum_{\substack{r \, \in \zzz \\[1mm] 0 \, \leq \, r \, < \, p}} \, 
\sum_{\substack{k' \, \in \zzz \\[1mm] 1 \, \leq \, k' \, \leq \, 2m}}
(-1)^{k'} q^{-\frac{1}{4m}(2mr+k')(2m(r+1)+k')} \, 
\theta_{k',m}(\tau, z)}_{(A)}
\\[1mm]
& & \hspace{-3mm}
+ \,\ 
\underbrace{
\sum_{\substack{r \, \in \zzz \\[1mm] 0 \, \leq \, r \, < \, p}} \, 
\sum_{\substack{k' \, \in \zzz \\[1mm] 1 \, \leq \, k' \, \leq \, 2m}}
(-1)^{k'} q^{-\frac{1}{4m}(2mr+k')(2m(r+1)+k')} \, 
\theta_{-k',m}(\tau, z)}_{(B)}
\end{eqnarray*}}
where $(A)$ is computed by using
\begin{equation}
\sum_{\substack{k \, \in \zzz \\[1mm] 1 \, \leq \, k \, \leq \, 2m}} 
= 
\sum_{\substack{k \, \in \zzz \\[1mm] 0 \, \leq \, k \, < \, 2m}} 
- \,\ \sum_{k \, = \, 0} \,\ + \,\ \sum_{k \, = \, 2m}
\label{m1:eqn:2022-919c}
\end{equation}
as follows:
{\allowdisplaybreaks
\begin{eqnarray*}
& & \hspace{-7mm}
(A) \,\ = \, 
\sum_{\substack{r \, \in \zzz \\[1mm] 0 \, \leq \, r \, < \, p}} \, 
\sum_{\substack{k \, \in \zzz \\[1mm] 0 \, \leq \, k \, < \, 2m}}
(-1)^{k} q^{-\frac{1}{4m}(2mr+k)(2m(r+1)+k)} \, 
\theta_{k,m}(\tau, z)
\\[2mm]
& &
- \, \sum_{\substack{r \, \in \zzz \\[1mm] 0 \, \leq \, r \, < \, p}} \, 
\underbrace{q^{-\frac{1}{4m}2mr \cdot 2m(r+1)}}_{
\substack{|| \\[0mm] {\displaystyle q^{-mr(r+1)}
}}} \, 
\theta_{0,m}(\tau, z)
+ 
\underbrace{
\sum_{\substack{r \, \in \zzz \\[1mm] 0 \, \leq \, r \, < \, p}} 
\hspace{-2mm} 
q^{-\frac{1}{4m}(2mr+2m)(2m(r+1)+2m)}
}_{\substack{|| \\[0mm] {\displaystyle 
\sum_{\substack{r \, \in \zzz \\[1mm] 1 \, \leq \, r \, \leq \, p}}
q^{-mr(r+1)}
}}}
\underbrace{\theta_{2m,m}}_{\substack{|| \\[1mm] 
{\displaystyle \theta_{0,m}
}}}(\tau, z)
\\[2mm]
&=&
\sum_{\substack{r \, \in \zzz \\[1mm] 0 \, \leq \, r \, < \, p}} \, 
\sum_{\substack{k \, \in \zzz \\[1mm] 0 \, \leq \, k \, < \, 2m}}
(-1)^{k} q^{-\frac{1}{4m}(2mr+k)(2m(r+1)+k)} \, 
\theta_{k,m}(\tau, z)
\\[2mm]
& &
+ \,\ \bigg\{- \, 
\sum_{\substack{r \, \in \zzz \\[1mm] 0 \, \leq \, r \, < \, p}} 
\, + \, 
\sum_{\substack{r \, \in \zzz \\[1mm] 1 \, \leq \, r \, \leq \, p}} 
\bigg\}
q^{-mr(r+1)} \, \theta_{0,m}(\tau, z)
\\[2mm]
&=&
\sum_{\substack{r \, \in \zzz \\[1mm] 0 \, \leq \, r \, < \, p}} \, 
\sum_{\substack{k \, \in \zzz \\[1mm] 0 \, \leq \, k \, < \, 2m}}
(-1)^{k} q^{-\frac{1}{4m}(2mr+k)(2m(r+1)+k)} \, 
\theta_{k,m}(\tau, z)
+ 
\big\{q^{-mp(p+1)}- 1\big\} \, \theta_{0,m}(\tau, z)
\end{eqnarray*}}
And $(B)$ becomes by putting $k'=2m-k$ and $r+2=-r'$ as follows:
{\allowdisplaybreaks
\begin{eqnarray*}
(B) &=&
\sum_{\substack{r' \, \in \zzz \\[1mm] -p-1 \, \leq \, r' \, \leq \, -1}} \,\ 
\sum_{\substack{k' \, \in \zzz \\[1mm] 0 \, \leq \, k' \, < \, 2m}}
(-1)^{k'} \, q^{-\frac{1}{4m}(2mr'+k')(2m(r'+1)+k')} \, 
\theta_{k',m}(\tau, z)
\\[2mm]
& &
- \sum_{\substack{k' \, \in \zzz \\[1mm] 0 \, \leq \, k' \, < \, 2m}}
(-1)^{k'} \, q^{-\frac{1}{4m}(-2m+k')k'} \, 
\theta_{k',m}(\tau, z)
\end{eqnarray*}}
Then (I) becomes as follows:
{\allowdisplaybreaks
\begin{eqnarray}
& & \hspace{-10mm}
{\rm (I)} \,\ = \,\ (A)+(B) 
\nonumber
\\[2mm]
&=&
\sum_{\substack{k \, \in \zzz \\[1mm] 0 \, \leq \, k \, < \, 2m}}
(-1)^{k}
\underbrace{
\sum_{\substack{r \, \in \zzz \\[1mm] -p-1 \, \leq \, r \, < \, p}} \, 
q^{-\frac{1}{4m}(2mr+k)(2m(r+1)+k)}}_{\substack{\hspace{16.5mm}
|| \,\ \leftarrow \,\ r+1 \, = \, r' \\[1mm] {\displaystyle 
\sum_{\substack{r' \, \in \zzz \\[1mm] -p \, \leq \, r' \, \leq \, p}}
q^{-\frac{1}{4m}(2m(r'-1)+k)(2mr'+k)}
}}} \, \theta_{k,m}(\tau, z)
\nonumber
\\[2mm]
& & \hspace{-5mm}
\underbrace{- 
\sum_{\substack{k \, \in \zzz \\[1mm] 0 \, \leq \, k \, < \, 2m}}
(-1)^k q^{-\frac{1}{4m}(-2m+k)k} \, \theta_{k,m}(\tau, z)
\, - \, \theta_{0,m}(\tau, z)}_{\substack{|| \\[0mm] 
{\displaystyle 
- \sum_{\substack{k \, \in \zzz \\[1mm] 0 \, \leq \, k \, \leq \, 2m}}
(-1)^k q^{-\frac{1}{4m}(-2m+k)k} \, \theta_{k,m}(\tau, z)
}}}
\, + \, 
q^{-mp(p+1)}\, \theta_{0,m}(\tau, z)
\label{m1:eqn:2022-919d}
\end{eqnarray}}
Then substituting \eqref{m1:eqn:2022-919d} into 
\eqref{m1:eqn:2022-919b}, we have
{\allowdisplaybreaks
\begin{eqnarray}
& & \hspace{-10mm}
\Phi^{[m,0] \, \ast}(\tau, \, z_1+p\tau, \, z_2-p\tau, \, 0)
\,\ = \,\ 
q^{mp(p+1)} \, \Big\{
\Phi^{[m,0] \, \ast}(\tau, z_1, z_2,0) \,\ + \,\ {\rm (I)}\Big\}
\nonumber
\\[2mm]
&=&
q^{mp(p+1)} \, \bigg\{
\Phi^{[m,0] \, \ast}(\tau, z_1, z_2,0)
\nonumber
\\[2mm]
& &
+ \,\ 
\sum_{\substack{r \, \in \zzz \\[1mm] -p \, \leq \, r \, \leq \, p}} \, 
\sum_{\substack{k \, \in \zzz \\[1mm] 0 \, \leq \, k \, < \, 2m}}
(-1)^{k} q^{-\frac{1}{4m}(2mr+k)(2m(r-1)+k)} \, \theta_{k,m}(\tau, z)
\nonumber
\\[2mm]
& & \hspace{-7.2mm}
\underbrace{- 
\sum_{\substack{k \, \in \zzz \\[1mm] 0 \, \leq \, k \, \leq \, 2m}}
(-1)^k q^{-\frac{1}{4m}k(k-2m)} \theta_{k,m}(\tau, z)}_{\substack{|| 
\\[-1mm] {\displaystyle \hspace{20mm}
\Phi^{[m,0] \, \ast}_{\rm add}(\tau, z_1, z_2,0) \quad 
\text{by Lemma \ref{m1:lemma:2022-921c}}
}}} \hspace{-3mm}
+ \quad q^{-mp(p+1)}\, \theta_{0,m}(\tau, z)\bigg\}
\nonumber
\\[2mm]
&=&
q^{mp(p+1)} \, \bigg\{
\widetilde{\Phi}^{[m,0] \, \ast}(\tau, z_1, z_2,0)
+ \hspace{-3mm}
\underbrace{
\sum_{\substack{r \, \in \zzz \\[1mm] -p \, \leq \, r \, \leq \, p}} \, 
\sum_{\substack{k \, \in \zzz \\[1mm] 0 \, \leq \, k \, < \, 2m}}
\hspace{-3mm}
(-1)^{k} q^{-\frac{1}{4m}(2mr+k)(2m(r-1)+k)} \theta_{k,m}(\tau, z)
}_{\rm (II)}
\nonumber
\\[-5mm]
& &
+ \,\ q^{-mp(p+1)}\, \theta_{0,m}(\tau, z)\bigg\}
\label{m1:eqn:2022-919e}
\end{eqnarray}}
We go further to compute 
$$
{\rm (II)} \,\ = \,\ {\rm (II)}_A + {\rm (II)}_B
$$
where 
{\allowdisplaybreaks
\begin{eqnarray*}
{\rm (II)}_A &:=&
\sum_{\substack{r \, \in \zzz \\[1mm] -p \, \leq \, r \, \leq \, p}} \,\
\sum_{\substack{k \, \in \zzz \\[1mm] 0 \, \leq \, k \, \leq \, m}}
(-1)^{k} q^{-\frac{1}{4m}(2mr+k)(2m(r-1)+k)} \theta_{k,m}(\tau, z)
\\[2mm]
{\rm (II)}_B &:=&
\sum_{\substack{r \, \in \zzz \\[1mm] -p \, \leq \, r \, \leq \, p}} \,\
\sum_{\substack{k \, \in \zzz \\[1mm] m \, < \, k \, < \, 2m}}
(-1)^{k} q^{-\frac{1}{4m}(2mr+k)(2m(r-1)+k)} \theta_{k,m}(\tau, z)
\end{eqnarray*}}
First we compute ${\rm (II)}_A$ :
\begin{subequations}
{\allowdisplaybreaks
\begin{eqnarray}
{\rm (II)}_A &=&
\sum_{\substack{r \, \in \zzz \\[1mm] -p \, \leq \, r \, \leq \, p}} \,
\underbrace{
\sum_{\substack{k \, \in \zzz \\[1mm] 0 \, \leq \, k \, \leq \, m}}}_{
\substack{|| \\[-2mm] {\displaystyle 
\sum_{\substack{k \, \in \zzz \\[1mm] 0 \, \leq \, k \, < \, m}} 
+ \sum_{k=m}
}}}
(-1)^k \, q^{-\frac{1}{4m}(2mr+k)(2m(r-1)+k)} \theta_{k,m}(\tau, z)
\nonumber
\\[2mm]
&=&
\sum_{\substack{r \, \in \zzz \\[1mm] -p \, \leq \, r \, \leq \, p}} \,\
\sum_{\substack{k \, \in \zzz \\[1mm] 0 \, \leq \, k \, < \, m}} 
(-1)^k \, q^{-\frac{1}{4m}(2mr+k)(2m(r-1)+k)} \theta_{k,m}(\tau, z)
\nonumber
\\[3mm]
& &
+ \,\ (-1)^m 
\sum_{\substack{r \, \in \zzz \\[1mm] -p \, \leq \, r \, \leq \, p}}
q^{-\frac{m}{4}(2r+1)(2r-1)} \, \theta_{m,m}(\tau, z)
\label{m1:eqn:2022-919f1}
\end{eqnarray}}
Next, ${\rm (II)}_B$ is rewitten as follows by putting $k=2m-k'$:
{\allowdisplaybreaks
\begin{eqnarray}
{\rm (II)}_B &=&
\sum_{\substack{r \, \in \zzz \\[1mm] -p \, \leq \, r \, \leq \, p}} \,\
\sum_{\substack{k' \, \in \zzz \\[1mm] 0 \, < \, k' \, < \, m}}
(-1)^{k'} q^{-\frac{1}{4m}(2m(r+1)-k')(2mr-k')} \theta_{-k',m}(\tau, z)
\nonumber
\\[2mm]
& \hspace{-2mm}
\underset{\substack{\\[0.5mm] \uparrow \\[1mm] 
r = -r'}}{=} \hspace{-2mm} &
\sum_{\substack{r' \, \in \zzz \\[1mm] -p \, \leq \, r' \, \leq \, p}} 
\underbrace{
\sum_{\substack{k' \, \in \zzz \\[1mm] 0 \, < \, k' \, < \, m}}
}_{\substack{|| \\[-0.5mm] {\displaystyle \hspace{-5mm}
\sum_{\substack{k' \, \in \zzz \\[1mm] 0 \, \leq \, k' \, < \, m}} 
-
\sum_{k'=0}
}}}
(-1)^{k'} q^{-\frac{1}{4m}(2m(r'-1)+k')(2mr'+k')} \theta_{-k',m}(\tau, z)
\nonumber
\\[1mm]
&=&
\sum_{\substack{r' \, \in \zzz \\[1mm] -p \, \leq \, r' \, \leq \, p}} \,\ 
\sum_{\substack{k' \, \in \zzz \\[1mm] 0 \, \leq \, k' \, < \, m}}
(-1)^{k'} q^{-\frac{1}{4m}(2m(r'-1)+k')(2mr'+k')} \theta_{-k',m}(\tau, z)
\nonumber
\\[2.5mm]
& & \hspace{-35mm}
\underbrace{
- \sum_{\substack{r' \, \in \zzz \\[1mm] -p \, \leq \, r' \, \leq \, p}}
q^{-mr'(r'-1)} \, \theta_{0,m}(\tau, z)
}_{\substack{|| \\[0mm] {\displaystyle \hspace{38mm}
- \, 2 
\sum_{\substack{r \, \in \zzz \\[1mm] 0 \, \leq \, r \, \leq \, p-1}} 
\hspace{-2mm}
q^{-mr(r+1)} \, \theta_{0,m}(\tau, z)
\, - \, 
q^{-mp(p+1)} \theta_{0,m}(\tau, z)
}}} 
\label{m1:eqn:2022-919f2}
\end{eqnarray}}
\end{subequations}
Then by \eqref{m1:eqn:2022-919f1} and \eqref{m1:eqn:2022-919f2},
we have
{\allowdisplaybreaks
\begin{eqnarray}
& & \hspace{-12mm}
{\rm (II)} +  q^{-mp(p+1)} \theta_{0,m}(\tau, z)
\nonumber
\\[3mm]
&=&
\sum_{\substack{r \, \in \zzz \\[1mm] -p \, \leq \, r \, \leq \, p}} \,
\sum_{\substack{k \, \in \zzz \\[1mm] 0 \, \leq \, k \, < \, m}} 
(-1)^k \, q^{-\frac{1}{4m}(2mr+k)(2m(r-1)+k)} 
\big[\theta_{k,m}+\theta_{-k,m}\big](\tau, z)
\nonumber
\\[3mm]
& & \hspace{-5mm}
- \, 2 
\sum_{\substack{r \, \in \zzz \\[1mm] 0 \, \leq \, r \, \leq \, p-1}} 
\hspace{-2mm}
q^{-mr(r+1)} \theta_{0,m}(\tau, z)
\, + \, 
(-1)^m \hspace{-2mm}
\sum_{\substack{r \, \in \zzz \\[1mm] -p \, \leq \, r \, \leq \, p}} 
\hspace{-2mm}
q^{-\frac{m}{4}(2r+1)(2r-1)} \, \theta_{m,m}(\tau, z)
\label{m1:eqn:2022-919g}
\end{eqnarray}}
Substituting this equation \eqref{m1:eqn:2022-919g} into 
\eqref{m1:eqn:2022-919b}, we obtain 
{\allowdisplaybreaks
\begin{eqnarray*}
& & \hspace{-10mm}
\Phi^{[m,0] \, \ast}(\tau, \, z_1+p\tau, \, z_2-p\tau, \, 0)
\,\ = \,\ 
q^{mp(p+1)} \, \bigg\{
\widetilde{\Phi}^{[m,0] \, \ast}(\tau, z_1, z_2,0)
\\[2mm]
& & 
+ \sum_{\substack{r \, \in \zzz \\[1mm] -p \, \leq \, r \, \leq \, p}} \,
\sum_{\substack{k \, \in \zzz \\[1mm] 0 \, \leq \, k \, < \, m}} 
(-1)^k \, q^{-\frac{1}{4m}(2mr+k)(2m(r-1)+k)} 
\big[\theta_{k,m}+\theta_{-k,m}\big](\tau, z)
\\[3mm]
& &
- \,\ 2 \hspace{-2mm}
\sum_{\substack{r \, \in \zzz \\[1mm] 0 \, \leq \, r \, \leq \, p-1}} 
\hspace{-3mm}
q^{-mr(r+1)} \theta_{0,m}(\tau, z)
\, + \, 
(-1)^m \hspace{-3mm}
\sum_{\substack{r \, \in \zzz \\[1mm] -p \, \leq \, r \, \leq \, p}} 
\hspace{-3mm}
q^{-\frac{m}{4}(2r+1)(2r-1)} \, \theta_{m,m}(\tau, z)
\bigg\}
\end{eqnarray*}}
proving Lemma \ref{m1:lemma:2022-919a}.
\end{proof}

\subsection{$\Phi^{[m,0] \, \ast}(\tau, z_1+p\tau, z_2-p\tau,t)$ 
$\sim$ the case $z_1-z_2=\tau$}
\label{subsec:Phi(+:0)ast(z1+ptau:z2-ptau):z1-z2=(1+2p)tau}

\medskip

\begin{lemma} 
\label{m1:lemma:2022-920a}
For $m \in \nnn$ and $p \in \zzz$, the following formula holds:
{\allowdisplaybreaks
\begin{eqnarray}
& & \hspace{-15mm}
\Phi^{[m,0] \, \ast}\Big(\tau, \,\ 
\frac{z}{2}+\dfrac{\tau}{2}+p\tau, \,\ 
\frac{z}{2}-\dfrac{\tau}{2}-p\tau, \,\ 0\Big)
\nonumber
\\[2mm] 
&= &
q^{mp(p+1)} \, \bigg\{
\widetilde{\Phi}^{[m,0] \, \ast}
\Big(\tau, \,\ 
\frac{z}{2}+\frac{\tau}{2}, \,\ 
\frac{z}{2}-\frac{\tau}{2}, \,\ 0\Big)
\nonumber
\\[2mm]
& & 
+ \sum_{\substack{r \, \in \zzz \\[1mm] -p \, \leq \, r \, \leq \, p}} \,
\sum_{\substack{k \, \in \zzz \\[1mm] 0 \, \leq \, k \, < \, m}} 
q^{-\frac{1}{4m}(2mr+k)(2m(r-1)+k)} 
\big[\theta_{k,m}+\theta_{-k,m}\big](\tau, z)
\nonumber
\\[3mm]
& & 
- \,\ 2 \hspace{-2mm}
\sum_{\substack{r \, \in \zzz \\[1mm] 0 \, \leq \, r \, \leq \, p-1}} 
\hspace{-3mm}
q^{-mr(r+1)} \theta_{0,m}(\tau, z)
\,\ + \hspace{-1mm}
\sum_{\substack{r \, \in \zzz \\[1mm] -p \, \leq \, r \, \leq \, p}} 
\hspace{-3mm}
q^{-\frac{m}{4}(2r+1)(2r-1)} \, \theta_{m,m}(\tau, z)
\bigg\}
\label{m1:eqn:2022-920a}
\end{eqnarray}}
\end{lemma}

\begin{proof} 
Letting $\left\{
\begin{array}{lcc}
z_1 &=&\frac{z}{2}+\frac{\tau}{2} \\[1mm]
z_2 &=&\frac{z}{2}-\frac{\tau}{2}
\end{array} \right. $ namely
$\left\{
\begin{array}{ccc}
z_1-z_2 &=& \tau \\[0mm]
z_1+z_2 &=& z
\end{array}\right. $ in the formula \eqref{m1:eqn:2022-917b}
in Lemma \ref{m1:lemma:2022-917b}, we have 
{\allowdisplaybreaks
\begin{eqnarray}
& & \hspace{-5mm}
\Phi^{[m,0] \, \ast}(\tau, \, z_1+p\tau, \, z_2-p\tau, \, 0)
\nonumber
\\[2mm]
&=&
e^{2\pi imp\tau}
q^{mp^2} \bigg\{
\Phi^{[m,0] \, \ast}(\tau, z_1, z_2,0)
+ 
\sum_{\substack{k \, \in \zzz \\[1mm] 1 \, \leq \, k \, \leq \, 2pm}} 
\hspace{-3mm}
e^{-\pi ik\tau}
q^{-\frac{1}{4m}k^2} 
\big[\theta_{k,m}+\theta_{-k,m}\big](\tau, z)\bigg\}
\nonumber
\\[2mm]
&=&
q^{mp(p+1)} \bigg\{
\Phi^{[m,0] \, \ast}(\tau, z_1, z_2,0)
+ \hspace{-2mm}
\underbrace{
\sum_{\substack{k \, \in \zzz \\[1mm] 1 \, \leq \, k \, \leq \, 2pm}} 
\hspace{-4mm}
q^{-\frac{1}{4m}(k^2+2mk)}
\big[\theta_{k,m}+\theta_{-k,m}\big](\tau, z)}_{\rm (I)}
\bigg\}
\label{m1:eqn:2022-920b}
\end{eqnarray}}
We compute (I) by putting $k=2mr+k'$ \,\ 
$(0 \leq r <p, \,\ 1 \leq k' \leq 2m)$ \, as follows:
{\allowdisplaybreaks
\begin{eqnarray*}
& & \hspace{-10mm}
{\rm (I)} := \, 
\sum_{\substack{r \, \in \zzz \\[1mm] 0 \, \leq \, r \, < \, p}}
\sum_{\substack{k' \, \in \zzz \\[1mm] 1 \, \leq \, k' \, \leq \, 2m}} 
q^{-\frac{1}{4m}(2mr+k')(2mr+k'+2m)}
\big[\theta_{2mr+k',m}+\theta_{-(2mr+k'),m}\big](\tau, z)
\\[2mm]
&=&
\underbrace{
\sum_{\substack{r \, \in \zzz \\[1mm] 0 \, \leq \, r \, < \, p}} \, 
\sum_{\substack{k' \, \in \zzz \\[1mm] 1 \, \leq \, k' \, \leq \, 2m}}
q^{-\frac{1}{4m}(2mr+k')(2m(r+1)+k')} \, 
\theta_{k',m}(\tau, z)}_{(A)}
\\[1mm]
& & \hspace{-3mm}
+ \,\ 
\underbrace{
\sum_{\substack{r \, \in \zzz \\[1mm] 0 \, \leq \, r \, < \, p}} \, 
\sum_{\substack{k' \, \in \zzz \\[1mm] 1 \, \leq \, k' \, \leq \, 2m}}
q^{-\frac{1}{4m}(2mr+k')(2m(r+1)+k')} \, 
\theta_{-k',m}(\tau, z)}_{(B)}
\end{eqnarray*}}
where $(A)$ is computed by using \eqref{m1:eqn:2022-919c} as follows:
{\allowdisplaybreaks
\begin{eqnarray*}
& & \hspace{-7mm}
(A) \,\ = \, 
\sum_{\substack{r \, \in \zzz \\[1mm] 0 \, \leq \, r \, < \, p}} \, 
\sum_{\substack{k \, \in \zzz \\[1mm] 0 \, \leq \, k \, < \, 2m}}
q^{-\frac{1}{4m}(2mr+k)(2m(r+1)+k)} \, 
\theta_{k,m}(\tau, z)
\\[2mm]
& &
- \, \sum_{\substack{r \, \in \zzz \\[1mm] 0 \, \leq \, r \, < \, p}} \, 
\underbrace{q^{-\frac{1}{4m}2mr \cdot 2m(r+1)}}_{
\substack{|| \\[0mm] {\displaystyle q^{-mr(r+1)}
}}} \, 
\theta_{0,m}(\tau, z)
\, + 
\underbrace{
\sum_{\substack{r \, \in \zzz \\[1mm] 0 \, \leq \, r \, < \, p}} 
\hspace{-2mm} 
q^{-\frac{1}{4m}(2mr+2m)(2m(r+1)+2m)}}_{\substack{|| \\[0mm] 
{\displaystyle 
\sum_{\substack{r \, \in \zzz \\[1mm] 1 \, \leq \, r \, \leq \, p}}
q^{-mr(r+1)}
}}}
\, 
\underbrace{\theta_{2m,m}}_{\substack{|| \\[0mm] {\displaystyle \theta_{0,m}
}}}(\tau, z)
\\[2mm]
&=&
\sum_{\substack{r \, \in \zzz \\[1mm] 0 \, \leq \, r \, < \, p}} \, 
\sum_{\substack{k \, \in \zzz \\[1mm] 0 \, \leq \, k \, < \, 2m}}
q^{-\frac{1}{4m}(2mr+k)(2m(r+1)+k)} \, 
\theta_{k,m}(\tau, z)
\\[2mm]
& &
+ \,\ \bigg\{- 
\sum_{\substack{r \, \in \zzz \\[1mm] 0 \, \leq \, r \, < \, p}} 
\, + \, 
\sum_{\substack{r \, \in \zzz \\[1mm] 1 \, \leq \, r \, \leq \, p}}
\bigg\}
q^{-mr(r+1)} \, \theta_{0,m}(\tau, z)
\\[2mm]
&=&
\sum_{\substack{r \, \in \zzz \\[1mm] 0 \, \leq \, r \, < \, p}} \, 
\sum_{\substack{k \, \in \zzz \\[1mm] 0 \, \leq \, k \, < \, 2m}}
q^{-\frac{1}{4m}(2mr+k)(2m(r+1)+k)} \, 
\theta_{k,m}(\tau, z)
+ 
\big\{q^{-mp(p+1)}- 1\big\} \, \theta_{0,m}(\tau, z)
\end{eqnarray*}}
And $(B)$ becomes by putting $k'=2m-k$ and $r+2=-r'$ as follows:
{\allowdisplaybreaks
\begin{eqnarray*}
(B) &=&
\sum_{\substack{r' \, \in \zzz \\[1mm] -p-1 \, \leq \, r' \, \leq \, -1}} \,\ 
\sum_{\substack{k' \, \in \zzz \\[1mm] 0 \, \leq \, k' \, < \, 2m}}
q^{-\frac{1}{4m}(2mr'+k')(2m(r'+1)+k')} \, 
\theta_{k',m}(\tau, z)
\\[2mm]
& &
- \sum_{\substack{k' \, \in \zzz \\[1mm] 0 \, \leq \, k' \, < \, 2m}}
q^{-\frac{1}{4m}(-2m+k')k'} \, \theta_{k',m}(\tau, z)
\end{eqnarray*}}
Then (I) becomes as follows:
{\allowdisplaybreaks
\begin{eqnarray}
& & \hspace{-10mm}
{\rm (I)} \,\ = \,\ (A)+(B) 
\nonumber
\\[2mm]
&=&
\sum_{\substack{k \, \in \zzz \\[1mm] 0 \, \leq \, k \, < \, 2m}}
\underbrace{
\sum_{\substack{r \, \in \zzz \\[1mm] -p-1 \, \leq \, r \, < \, p}} \, 
q^{-\frac{1}{4m}(2mr+k)(2m(r+1)+k)}}_{\substack{\hspace{16.5mm}
|| \,\ \leftarrow \,\ r+1 \, = \, r' \\[1mm] {\displaystyle 
\sum_{\substack{r' \, \in \zzz \\[1mm] -p \, \leq \, r' \, \leq \, p}}
q^{-\frac{1}{4m}(2m(r'-1)+k)(2mr'+k)}
}}} \, \theta_{k,m}(\tau, z)
\nonumber
\\[2mm]
& &
\underbrace{- 
\sum_{\substack{k \, \in \zzz \\[1mm] 0 \, \leq \, k \, < \, 2m}}
q^{-\frac{1}{4m}(-2m+k)k} \, \theta_{k,m}(\tau, z)
\, - \, \theta_{0,m}(\tau, z)}_{\substack{|| \\[0mm] 
{\displaystyle 
- \sum_{\substack{k \, \in \zzz \\[1mm] 0 \, \leq \, k \, \leq \, 2m}}
q^{-\frac{1}{4m}(-2m+k)k} \, \theta_{k,m}(\tau, z)
}}}
\, + \, 
q^{-mp(p+1)}\, \theta_{0,m}(\tau, z)
\label{m1:eqn:2022-920c}
\end{eqnarray}}
Then substituting \eqref{m1:eqn:2022-920c} into \eqref{m1:eqn:2022-920b},
we have
{\allowdisplaybreaks
\begin{eqnarray}
& & \hspace{-10mm}
\Phi^{[m,0] \, \ast}(\tau, \, z_1+p\tau, \, z_2-p\tau, \, 0)
\,\ = \,\ 
q^{mp(p+1)} \, \Big\{
\Phi^{[m,0] \, \ast}(\tau, z_1, z_2,0) \,\ + \,\ {\rm (I)}\Big\}
\nonumber
\\[2mm]
&=&
q^{mp(p+1)} \, \bigg\{
\Phi^{[m,0] \, \ast}(\tau, z_1, z_2,0)
\\[2mm]
& &
+ \,\ 
\sum_{\substack{r \, \in \zzz \\[1mm] -p \, \leq \, r \, \leq \, p}} \, 
\sum_{\substack{k \, \in \zzz \\[1mm] 0 \, \leq \, k \, < \, 2m}}
q^{-\frac{1}{4m}(2mr+k)(2m(r-1)+k)} \, \theta_{k,m}(\tau, z)
\nonumber
\\[2mm]
& & \hspace{-12mm}
\underbrace{- 
\sum_{\substack{k \, \in \zzz \\[1mm] 0 \, \leq \, k \, \leq \, 2m}}
q^{-\frac{1}{4m}k(k-2m)} \theta_{k,m}(\tau, z)}_{\substack{|| 
\\[-0.5mm] {\displaystyle \hspace{20mm}
\Phi^{[m,0] \, \ast}_{\rm add}(\tau, z_1, z_2,0) \quad 
\text{by Lemma \ref{m1:lemma:2022-921c}}
}}} \hspace{-3mm}
+ \quad q^{-mp(p+1)}\, \theta_{0,m}(\tau, z)\bigg\}
\nonumber
\\[2mm]
&=&
q^{mp(p+1)} \, \bigg\{
\widetilde{\Phi}^{[m,0] \, \ast}(\tau, z_1, z_2,0)
+ \hspace{-3mm}
\underbrace{
\sum_{\substack{r \, \in \zzz \\[1mm] -p \, \leq \, r \, \leq \, p}} \, 
\sum_{\substack{k \, \in \zzz \\[1mm] 0 \, \leq \, k \, < \, 2m}}
q^{-\frac{1}{4m}(2mr+k)(2m(r-1)+k)} \theta_{k,m}(\tau, z)}_{\rm (II)}
\nonumber
\\[-5mm]
& &
+ \,\ q^{-mp(p+1)}\, \theta_{0,m}(\tau, z)\bigg\}
\label{m1:eqn:2022-920d}
\end{eqnarray}}
We go further to compute 
$$
{\rm (II)} \,\ = \,\ {\rm (II)}_A+{\rm (II)}_B
$$
where 
\begin{subequations}
{\allowdisplaybreaks
\begin{eqnarray*}
{\rm (II)}_A &:=&
\sum_{\substack{r \, \in \zzz \\[1mm] -p \, \leq \, r \, \leq \, p}} \,\
\sum_{\substack{k \, \in \zzz \\[1mm] 0 \, \leq \, k \, \leq \, m}}
q^{-\frac{1}{4m}(2mr+k)(2m(r-1)+k)} \theta_{k,m}(\tau, z)
\\[2mm]
{\rm (II)}_B &:=&
\sum_{\substack{r \, \in \zzz \\[1mm] -p \, \leq \, r \, \leq \, p}} \,\ 
\sum_{\substack{k \, \in \zzz \\[1mm] m \, < \, k \, < \, 2m}} 
q^{-\frac{1}{4m}(2mr+k)(2m(r-1)+k)} \theta_{k,m}(\tau, z)
\end{eqnarray*}}
First we compute ${\rm (II)}_A$ : 
{\allowdisplaybreaks
\begin{eqnarray}
{\rm (II)}_A &=&
\sum_{\substack{r \, \in \zzz \\[1mm] -p \, \leq \, r \, \leq \, p}} \,
\underbrace{
\sum_{\substack{k \, \in \zzz \\[1mm] 0 \, \leq \, k \, \leq \, m}}}_{
\substack{|| \\[-0.5mm] {\displaystyle 
\sum_{\substack{k \, \in \zzz \\[1mm] 0 \, \leq \, k \, < \, m}} + \sum_{k=m}
}}}
q^{-\frac{1}{4m}(2mr+k)(2m(r-1)+k)} \theta_{k,m}(\tau, z)
\nonumber
\\[2mm]
&=&
\sum_{\substack{r \, \in \zzz \\[1mm] -p \, \leq \, r \, \leq \, p}} \,
\sum_{\substack{k \, \in \zzz \\[1mm] 0 \, \leq \, k \, < \, m}} 
q^{-\frac{1}{4m}(2mr+k)(2m(r-1)+k)} \theta_{k,m}(\tau, z)
\nonumber
\\[3mm]
& &
+ \,\ 
\sum_{\substack{r \, \in \zzz \\[1mm] -p \, \leq \, r \, \leq \, p}}
q^{-\frac{m}{4}(2r+1)(2r-1)} \, \theta_{m,m}(\tau, z)
\label{m1:eqn:2022-920e1}
\end{eqnarray}}
Next, ${\rm (II)}_B$ is rewritten as follows by putting $k=2m-k'$ :
{\allowdisplaybreaks
\begin{eqnarray}
{\rm (II)}_B &=&
\sum_{\substack{r \, \in \zzz \\[1mm] -p \, \leq \, r \, \leq \, p}} \,
\sum_{\substack{k' \, \in \zzz \\[1mm] 0 \, < \, k' \, < \, m}}
q^{-\frac{1}{4m}(2m(r+1)-k')(2mr-k')} \theta_{-k',m}(\tau, z)
\nonumber
\\[2mm]
& \hspace{-2mm}
\underset{\substack{\\[0.5mm] \uparrow \\[1mm] 
r = -r'}}{=} \hspace{-2mm} &
\sum_{\substack{r' \, \in \zzz \\[1mm] -p \, \leq \, r' \, \leq \, p}} 
\underbrace{
\sum_{\substack{k' \, \in \zzz \\[1mm] 0 \, < \, k' \, < \, m}}}_{
\substack{|| \\[-0.5mm] {\displaystyle \hspace{-5mm}
\sum_{\substack{k' \, \in \zzz \\[1mm] 0 \, \leq \, k' \, < \, m}} -\sum_{k'=0}
}}}
q^{-\frac{1}{4m}(2m(r'-1)+k')(2mr'+k')} \theta_{-k',m}(\tau, z)
\nonumber
\\[1mm]
&=&
\sum_{\substack{r' \, \in \zzz \\[1mm] -p \, \leq \, r' \, \leq \, p}} \, 
\sum_{\substack{k' \, \in \zzz \\[1mm] 0 \, \leq \, k' \, < \, m}}
q^{-\frac{1}{4m}(2m(r'-1)+k')(2mr'+k')} \theta_{-k',m}(\tau, z)
\nonumber
\\[2.5mm]
& & \hspace{-35mm}
\underbrace{
- \sum_{\substack{r' \, \in \zzz \\[1mm] -p \, \leq \, r' \, \leq \, p}}
q^{-mr'(r'-1)} \, \theta_{0,m}(\tau, z)
}_{\substack{|| \\[0mm] {\displaystyle \hspace{38mm}
- \, 2 
\sum_{\substack{r \, \in \zzz \\[1mm] 0 \, \leq \, r \, \leq \, p-1}} 
\hspace{-2mm}
q^{-mr(r+1)} \, \theta_{0,m}(\tau, z)
\, - \, 
q^{-mp(p+1)} \theta_{0,m}(\tau, z)
}}} 
\label{m1:eqn:2022-920e2}
\end{eqnarray}}
\end{subequations}
Then by \eqref{m1:eqn:2022-920e1} and \eqref{m1:eqn:2022-920e2}, 
we have 
{\allowdisplaybreaks
\begin{eqnarray}
& & \hspace{-12mm}
{\rm (II)} +  q^{-mp(p+1)} \theta_{0,m}(\tau, z)
\nonumber
\\[3mm]
&=&
\sum_{\substack{r \, \in \zzz \\[1mm] -p \, \leq \, r \, \leq \, p}} \,
\sum_{\substack{k \, \in \zzz \\[1mm] 0 \, \leq \, k \, < \, m}} 
q^{-\frac{1}{4m}(2mr+k)(2m(r-1)+k)} 
\big[\theta_{k,m}+\theta_{-k,m}\big](\tau, z)
\nonumber
\\[3mm]
& &
- \, 2 
\sum_{\substack{r \, \in \zzz \\[1mm] 0 \, \leq \, r \, \leq \, p-1}}
q^{-mr(r+1)} \theta_{0,m}(\tau, z)
\, + \, 
\sum_{\substack{r \, \in \zzz \\[1mm] -p \, \leq \, r \, \leq \, p}}
q^{-\frac{m}{4}(2r+1)(2r-1)} \, \theta_{m,m}(\tau, z)
\label{m1:eqn:2022-920f}
\end{eqnarray}}
Substituting this equation \eqref{m1:eqn:2022-920f} into 
\eqref{m1:eqn:2022-920b}, we obtain 
{\allowdisplaybreaks
\begin{eqnarray*}
& & \hspace{-10mm}
\Phi^{[m,0] \, \ast}(\tau, \, z_1+p\tau, \, z_2-p\tau, \, 0)
\,\ = \,\ 
q^{mp(p+1)} \, \bigg\{
\widetilde{\Phi}^{[m,0] \, \ast}(\tau, z_1, z_2,0)
\nonumber
\\[2mm]
& & 
+ \sum_{\substack{r \, \in \zzz \\[1mm] -p \, \leq \, r \, \leq \, p}} \,
\sum_{\substack{k \, \in \zzz \\[1mm] 0 \, \leq \, k \, < \, m}} 
(-1)^k \, q^{-\frac{1}{4m}(2mr+k)(2m(r-1)+k)} 
\big[\theta_{k,m}+\theta_{-k,m}\big](\tau, z)
\nonumber
\\[3mm]
& &
- \,\ 2 \hspace{-2mm}
\sum_{\substack{r \, \in \zzz \\[1mm] 0 \, \leq \, r \, \leq \, p-1}} 
\hspace{-3mm}
q^{-mr(r+1)} \theta_{0,m}(\tau, z)
\, + 
\sum_{\substack{r \, \in \zzz \\[1mm] -p \, \leq \, r \, \leq \, p}} 
\hspace{-3mm}
q^{-\frac{m}{4}(2r+1)(2r-1)} \, \theta_{m,m}(\tau, z)
\bigg\}
\label{eqn:2022-807a4}
\end{eqnarray*}}
proving Lemma \ref{m1:lemma:2022-920a}.
\end{proof}

\section{Modified function $\widetilde{\Phi}^{[m,0] \, \ast}$
with specialization}
\label{sec:Phi:tilde:ast}

\begin{prop} 
\label{m1:prop:2022-921b}
For $m \in \nnn$ and $p \in \zzz_{\geq 0}$, the following formulas hold:
\begin{subequations}
\begin{enumerate}
\item[{\rm 1)}] \,\ $(-1)^p \, q^{-\frac{m}{4}} \, 
\theta_{p-m-\frac12, m+\frac12}^{(-)}(\tau,0) \,\ 
\widetilde{\Phi}^{[m,0] \, \ast}\Big(\tau, \,\ 
\dfrac{z}{2}+\dfrac{\tau}{2}-\dfrac12, \,\ 
\dfrac{z}{2}-\dfrac{\tau}{2}+\dfrac12, \,\ 0\Big)$
{\allowdisplaybreaks
\begin{eqnarray}
&=&
\frac{\eta(\tau)^3}{\theta_{0,\frac12}(\tau,z)} \cdot 
\big[\theta_{p, m+\frac12}+ \theta_{-p, m+\frac12}\big](\tau,0)
\nonumber
\\[2mm]
& & \hspace{-5mm}
- \,\ 
\bigg[\sum_{\substack{j, \, r \, \in \, \frac12 \, \zzz_{\rm odd} \\[1mm]
0 \, \leq \, r \, < \, j
}}
-\sum_{\substack{j, \, r \, \in \, \frac12 \, \zzz_{\rm odd} \\[1mm]
j \, \leq \, r \, < \, 0}}\bigg]
\sum_{\substack{k \, \in \zzz \\[1mm] 0 \, < \, k \, < \, m}}
(-1)^{j-\frac12+k} \, 
q^{(m+\frac12)(j+\frac{2pm}{2m+1})^2}
\nonumber
\\[-4mm]
& & \hspace{60mm}
\times \,\ 
q^{-\frac{1}{4m} \, [2mr+k+2mp]^2}
\big[\theta_{k,m}+\theta_{-k,m}\big](\tau, z)
\nonumber
\\[3mm]
& & \hspace{-5mm}
- \,\ \bigg[
\sum_{\substack{j, \, r \, \in \, \frac12 \, \zzz_{\rm odd} \\[1mm]
0 \, \leq \, r \, \leq \, j
}}
-\sum_{\substack{j, \, r \, \in \, \frac12 \, \zzz_{\rm odd} \\[1mm]
j \, < \, r \, < \, 0}}\bigg]
\sum_{\substack{k \, \in \zzz \\[1mm] 0 \, \leq \, k \, \leq \, m}}
(-1)^{j-\frac12+k} \, 
q^{(m+\frac12)(j+\frac{2pm}{2m+1})^2}
\nonumber
\\[-4mm]
& & \hspace{60mm}
\times \,\ 
q^{-\frac{1}{4m} \, [2mr-k+2mp]^2}
\big[\theta_{k,m}+\theta_{-k,m}\big](\tau, z)
\nonumber
\\[3mm]
& & \hspace{-6mm}
- \, (-1)^p \, 
\theta_{p-m-\frac12, m+\frac12}^{(-)}(\tau,0) \hspace{-3mm}
\sum_{\substack{r \, \in \zzz \\[1mm] -p \, < \, r \, \leq \, p}} \,\ 
\sum_{\substack{k \, \in \zzz \\[1mm] 0 \, \leq \, k \, < \, m}} 
\hspace{-2mm}
(-1)^k q^{- \, \frac{1}{4m} \, (m(2r-1)+k)^2} 
\big[\theta_{k,m}+\theta_{-k,m}\big](\tau, z)
\nonumber
\\[3mm]
& & \hspace{-5mm}
+ \,\ 2 \, (-1)^p \, 
\theta_{p-m-\frac12, m+\frac12}^{(-)}(\tau,0)
\sum_{\substack{r \, \in \zzz \\[1mm] 0 \, \leq \, r \, \leq \, p-1}}
q^{-m(r+\frac12)^2} \, \theta_{0,m}(\tau, z)
\nonumber
\\[2mm]
& & \hspace{-5mm}
- \,\ (-1)^{m+p} \, \theta_{p-m-\frac12, m+\frac12}^{(-)}(\tau,0)
\sum_{\substack{r \, \in \zzz \\[1mm] -p \, < \, r \, < \, p}}
q^{-mr^2} \, \theta_{m,m}(\tau,z)
\label{m1:eqn:2022-921b}
\end{eqnarray}}
\item[{\rm 2)}] \,\ $(-1)^p \, 
q^{-\frac{m}{4}} \, 
\theta_{-p+m+\frac12, m+\frac12}^{(-)}(\tau,0) \,\ 
\widetilde{\Phi}^{[m,0] \, \ast}\Big(\tau, \,\ 
\dfrac{z}{2}+\dfrac{\tau}{2}, \,\ 
\dfrac{z}{2}-\dfrac{\tau}{2}, \,\ 0\Big)$
{\allowdisplaybreaks
\begin{eqnarray}
&=&
(-1)^p \, 
\frac{\eta(\tau)^3}{\theta_{0,\frac12}^{(-)}(\tau,z)} \cdot 
\big[\theta_{p, m+\frac12}^{(-)}+ \theta_{-p, m+\frac12}^{(-)}\big](\tau,0)
\nonumber
\\[2mm]
& & \hspace{-5mm}
- \,\ 
\bigg[\sum_{\substack{j, \, r \, \in \, \frac12 \, \zzz_{\rm odd} \\[1mm]
0 \, \leq \, r \, < \, j
}}
-\sum_{\substack{j, \, r \, \in \, \frac12 \, \zzz_{\rm odd} \\[1mm]
j \, \leq \, r \, < \, 0}}\bigg]
\sum_{\substack{k \, \in \zzz \\[1mm] 0 \, < \, k \, < \, m}}
(-1)^{j-\frac12} \, 
q^{(m+\frac12)(j+\frac{2pm}{2m+1})^2}
\nonumber
\\[-4mm]
& & \hspace{60mm}
\times \,\ 
q^{-\frac{1}{4m} \, [2mr+k+2mp]^2}
\big[\theta_{k,m}+\theta_{-k,m}\big](\tau, z)
\nonumber
\\[3mm]
& & \hspace{-5mm}
- \,\ \bigg[
\sum_{\substack{j, \, r \, \in \, \frac12 \, \zzz_{\rm odd} \\[1mm]
0 \, \leq \, r \, \leqq \, j
}}
-\sum_{\substack{j, \, r \, \in \, \frac12 \, \zzz_{\rm odd} \\[1mm]
j \, < \, r \, < \, 0}}\bigg]
\sum_{\substack{k \, \in \zzz \\[1mm] 0 \, \leq \, k \, \leq \, m}}
(-1)^{j-\frac12} \, 
q^{(m+\frac12)(j+\frac{2pm}{2m+1})^2}
\nonumber
\\[-4mm]
& & \hspace{60mm}
\times \,\ 
q^{-\frac{1}{4m} \, [2mr-k+2mp]^2}
\big[\theta_{k,m}+\theta_{-k,m}\big](\tau, z)
\nonumber
\\[3mm]
& & \hspace{-5mm}
- \,\ 
\theta_{2mp+m+\frac12, m+\frac12}^{(-)}(\tau,0) \hspace{-3mm}
\sum_{\substack{r \, \in \zzz \\[1mm] -p \, < \, r \, \leq \, p}} \,\ 
\sum_{\substack{k \, \in \zzz \\[1mm] 0 \, \leq \, k \, < \, m}} 
\hspace{-2mm}
q^{- \, \frac{1}{4m} \, (m(2r-1)+k)^2} 
\big[\theta_{k,m}+\theta_{-k,m}\big](\tau, z)
\nonumber
\\[3mm]
& & \hspace{-5mm}
+ \,\ 2 \, 
\theta_{2mp+m+\frac12, m+\frac12}^{(-)}(\tau,0)
\sum_{\substack{r \, \in \zzz \\[1mm] 0 \, \leq \, r \, \leq \, p-1}}
q^{-m(r+\frac12)^2} \, \theta_{0,m}(\tau, z)
\nonumber
\\[2mm]
& & \hspace{-5mm}
- \,\ \theta_{2mp+m+\frac12, m+\frac12}^{(-)}(\tau,0)
\sum_{\substack{r \, \in \zzz \\[1mm] -p \, < \, r \, < \, p}}
q^{-mr^2} \, \theta_{m,m}(\tau,z)
\label{m1:eqn:2022-921c}
\end{eqnarray}}
\end{enumerate}
\end{subequations}
\end{prop}

\begin{proof} 1) In the case $\left\{
\begin{array}{lcc}
z_1 &=& \frac{z}{2}+\frac{\tau}{2}-\frac12 \\[1mm]
z_2 &=& \frac{z}{2}-\frac{\tau}{2}+\frac12
\end{array}\right.$, substituting \eqref{m1:eqn:2022-919a} into 
\eqref{m1:eqn:2022-918a}, we have
{\allowdisplaybreaks
\begin{eqnarray*}
& & \hspace{-7mm}
\theta_{2mp+m+\frac12, m+\frac12}^{(-)}(\tau,0) \,\ 
\times \,\ \Bigg\{
q^{mp(p+1)} \, 
\widetilde{\Phi}^{[m,0] \, \ast}
\Big(\tau, \,\ 
\frac{z}{2}+\frac{\tau}{2}-\frac12, \,\ 
\frac{z}{2}-\frac{\tau}{2}+\frac12, \,\ 0\Big)
\\[2mm]
& & 
+ \,\ 
q^{mp(p+1)} 
\sum_{\substack{r \, \in \zzz \\[1mm] -p \, \leq \, r \, \leq \, p}} \, 
\sum_{\substack{k \, \in \zzz \\[1mm] 0 \, \leq \, k \, < \, m}} 
\hspace{-2mm}
(-1)^k q^{-\frac{1}{4m}(2mr+k)(2m(r-1)+k)} 
\big[\theta_{k,m}+\theta_{-k,m}\big](\tau, z)
\\[2mm]
& &
- \,\ 2 \, q^{mp(p+1)} 
\sum_{\substack{r \, \in \zzz \\[1mm] 0 \, \leq \, r \, \leq \, p-1}}
q^{-mr(r+1)} \, \theta_{0,m}(\tau, z)
\\[2mm]
& &
+ \,\ (-1)^m \, q^{mp(p+1)}
\sum_{\substack{r \, \in \zzz \\[1mm] -p \, \leq \, r \, \leq \, p}}
q^{-\frac{m}{4}(2r+1)(2r-1)} \, \theta_{m,m}(\tau,z)
\Bigg\}
\\[2mm]
&=&
q^{\frac{m}{4}(2p+1)^2} \, 
\frac{\eta(\tau)^3}{\theta_{0,\frac12}(\tau,z)} \cdot 
\big[\theta_{p, m+\frac12}+ \theta_{-p, m+\frac12}\big](\tau,0)
\\[3mm]
& & \hspace{-5mm}
- \,\ q^{\frac{m}{4}(2p+1)^2} \, 
\bigg[\sum_{\substack{j, \, r \, \in \, \frac12 \, \zzz_{\rm odd} \\[1mm]
0 \, \leq \, r \, < \, j
}}
-\sum_{\substack{j, \, r \, \in \, \frac12 \, \zzz_{\rm odd} \\[1mm]
j \, \leq \, r \, < \, 0}}\bigg]
\sum_{\substack{k \, \in \zzz \\[1mm] 0 \, < \, k \, < \, m}}
(-1)^{j-\frac12+k} \, 
q^{(m+\frac12)(j+\frac{2pm}{2m+1})^2}
\\[-1mm]
& & \hspace{60mm}
\times \,\ 
q^{-\frac{1}{4m} \, [2mr+k+2mp]^2}
\big[\theta_{k,m}+\theta_{-k,m}\big](\tau, z)
\\[3mm]
& & \hspace{-5mm}
- \,\ q^{\frac{m}{4}(2p+1)^2} \, \bigg[
\sum_{\substack{j, \, r \, \in \, \frac12 \, \zzz_{\rm odd} \\[1mm]
0 \, \leq \, r \, \leq \, j
}}
-\sum_{\substack{j, \, r \, \in \, \frac12 \, \zzz_{\rm odd} \\[1mm]
j \, < \, r \, < \, 0}}\bigg]
\sum_{\substack{k \, \in \zzz \\[1mm] 0 \, \leq \, k \, \leq \, m}}
(-1)^{j-\frac12+k} \, 
q^{(m+\frac12)(j+\frac{2pm}{2m+1})^2}
\\[-1mm]
& & \hspace{60mm}
\times \,\ 
q^{-\frac{1}{4m} \, [2mr-k+2mp]^2}
\big[\theta_{k,m}+\theta_{-k,m}\big](\tau, z)
\\[3mm]
& & \hspace{-5mm}
+ \,\ 
\theta_{2mp+m+\frac12, m+\frac12}^{(-)}(\tau,0)
\sum_{\substack{k \, \in \zzz \\[1mm] 0 \, \leq \, k \, \leq \, m}} 
(-1)^k
q^{-\frac{1}{4m}k^2+\frac{k}{2}(2p+1)} \, 
\big[\theta_{k,m}+\theta_{-k,m}\big](\tau, z)
\end{eqnarray*}}
Multiplying $q^{-\frac{m}{4}(2p+1)^2} \, = \, q^{-mp(p+1)-\frac{m}{4}}$ \, 
to both sides, we have
{\allowdisplaybreaks
\begin{eqnarray*}
& & \hspace{5mm}
q^{-\frac{m}{4}} \, 
\theta_{2mp+m+\frac12, m+\frac12}^{(-)}(\tau,0) \, 
\widetilde{\Phi}^{[m,0] \, \ast}
\Big(\tau, \,\ 
\frac{z}{2}+\frac{\tau}{2}-\frac12, \,\ 
\frac{z}{2}-\frac{\tau}{2}+\frac12, \,\ 0\Big)
\\[2mm]
& & 
+ \,\ 
q^{-\frac{m}{4}} \, 
\theta_{2mp+m+\frac12, m+\frac12}^{(-)}(\tau,0) \, 
\\[2mm]
& & \hspace{20mm}
\times 
\sum_{\substack{r \, \in \zzz \\[1mm] -p \, \leq \, r \, \leq \, p}} \, 
\sum_{\substack{k \, \in \zzz \\[1mm] 0 \, \leq \, k \, < \, m}} 
\hspace{-2mm}
(-1)^{k} q^{-\frac{1}{4m}(2mr+k)(2m(r-1)+k)} 
\big[\theta_{k,m}+\theta_{-k,m}\big](\tau, z)
\\[3mm]
& &
- \,\ 2 \, 
q^{-\frac{m}{4}} \, 
\theta_{2mp+m+\frac12, m+\frac12}^{(-)}(\tau,0)
\sum_{\substack{r \, \in \zzz \\[1mm] 0 \, \leq \, r \, \leq \, p-1}}
q^{-mr(r+1)} \, \theta_{0,m}(\tau, z)
\\[2mm]
& &
+ \,\ (-1)^m \, 
q^{-\frac{m}{4}} \, 
\theta_{2mp+m+\frac12, m+\frac12}^{(-)}(\tau,0)
\sum_{\substack{r \, \in \zzz \\[1mm] -p \, \leq \, r \, \leq \, p}}
q^{-\frac{m}{4}(2r+1)(2r-1)} \, \theta_{m,m}(\tau,z)
\\[2mm]
&=&
\frac{\eta(\tau)^3}{\theta_{0,\frac12}(\tau,z)} \cdot 
\big[\theta_{p, m+\frac12}+ \theta_{-p, m+\frac12}\big](\tau,0)
\\[3mm]
& & \hspace{-5mm}
- \,\ 
\bigg[\sum_{\substack{j, \, r \, \in \, \frac12 \zzz_{\rm odd} \\[1mm]
0 \, \leq \, r \, < \, j
}}
-\sum_{\substack{j, \, r \, \in \, \frac12 \, \zzz_{\rm odd} \\[1mm]
j \, \leq \, r \, < \, 0}}\bigg]
\sum_{\substack{k \, \in \zzz \\[1mm] 0 \, < \, k \, < \, m}}
(-1)^{j-\frac12+k} \, 
q^{(m+\frac12)(j+\frac{2pm}{2m+1})^2}
\\[-4mm]
& & \hspace{60mm}
\times \,\ 
q^{-\frac{1}{4m} \, [2mr+k+2mp]^2}
\big[\theta_{k,m}+\theta_{-k,m}\big](\tau, z)
\\[3mm]
& & \hspace{-5mm}
- \,\ \bigg[
\sum_{\substack{j, \, r \, \in \, \frac12 \, \zzz_{\rm odd} \\[1mm]
0 \, \leq \, r \, \leq \, j
}}
-\sum_{\substack{j, \, r \, \in \, \frac12 \, \zzz_{\rm odd} \\[1mm]
j \, < \, r \, < \, 0}}\bigg]
\sum_{\substack{k \, \in \zzz \\[1mm] 0 \, \leq \, k \, \leq \, m}}
(-1)^{j-\frac12+k} \, 
q^{(m+\frac12)(j+\frac{2pm}{2m+1})^2}
\\[-4mm]
& & \hspace{60mm}
\times \,\ 
q^{-\frac{1}{4m} \, [2mr-k+2mp]^2}
\big[\theta_{k,m}+\theta_{-k,m}\big](\tau, z)
\\[3mm]
& & \hspace{-5mm}
+ \,\ 
\theta_{2mp+m+\frac12, m+\frac12}^{(-)}(\tau,0)
\sum_{\substack{k \, \in \zzz \\[1mm] 0 \, \leq \, k \, \leq \, m}} 
(-1)^k
q^{-\frac{1}{4m}[k-m(2p+1)]^2}
\big[\theta_{k,m}+\theta_{-k,m}\big](\tau, z)
\end{eqnarray*}}
namely
\begin{subequations}
{\allowdisplaybreaks
\begin{eqnarray}
& & \hspace{-7mm}
q^{-\frac{m}{4}} \, 
\theta_{2mp+m+\frac12, m+\frac12}^{(-)}(\tau,0) \, 
\widetilde{\Phi}^{[m,0] \, \ast}
\Big(\tau, \,\ 
\frac{z}{2}+\frac{\tau}{2}-\frac12, \,\ 
\frac{z}{2}-\frac{\tau}{2}+\frac12, \,\ 0\Big)
\nonumber
\\[2mm]
&=&
\frac{\eta(\tau)^3}{\theta_{0,\frac12}(\tau,z)} \cdot 
\big[\theta_{p, m+\frac12}+ \theta_{-p, m+\frac12}\big](\tau,0)
\nonumber
\\[3mm]
& & \hspace{-5mm}
- \,\ 
\bigg[\sum_{\substack{j, \, r \, \in \, \frac12 \, \zzz_{\rm odd} \\[1mm]
0 \, \leq \, r \, < \, j
}}
-\sum_{\substack{j, \, r \, \in \, \frac12 \, \zzz_{\rm odd} \\[1mm]
j \, \leq \, r \, < \, 0}}\bigg]
\sum_{\substack{k \, \in \zzz \\[1mm] 0 \, < \, k \, < \, m}}
(-1)^{j-\frac12+k} \, 
q^{(m+\frac12)(j+\frac{2pm}{2m+1})^2}
\nonumber
\\[-4mm]
& & \hspace{60mm}
\times \,\ 
q^{-\frac{1}{4m} \, [2mr+k+2mp]^2}
\big[\theta_{k,m}+\theta_{-k,m}\big](\tau, z)
\nonumber
\\[3mm]
& & \hspace{-5mm}
- \,\ \bigg[
\sum_{\substack{j, \, r \, \in \, \frac12 \, \zzz_{\rm odd} \\[1mm]
0 \, \leq \, r \, \leq \, j
}}
-\sum_{\substack{j, \, r \, \in \, \frac12 \, \zzz_{\rm odd} \\[1mm]
j \, < \, r \, < \, 0}}\bigg]
\sum_{\substack{k \, \in \zzz \\[1mm] 0 \, \leq \, k \, \leq \, m}}
(-1)^{j-\frac12+k} \, 
q^{(m+\frac12)(j+\frac{2pm}{2m+1})^2}
\nonumber
\\[-4mm]
& & \hspace{60mm}
\times \,\ 
q^{-\frac{1}{4m} \, [2mr-k+2mp]^2}
\big[\theta_{k,m}+\theta_{-k,m}\big](\tau, z)
\nonumber
\\[3mm]
& & \hspace{-5mm}
+ \,\ 
\underbrace{\theta_{2mp+m+\frac12, m+\frac12}^{(-)}(\tau,0)
\sum_{\substack{k \, \in \zzz \\[1mm] 0 \, \leq \, k \, \leq \, m}} 
(-1)^k
q^{-\frac{1}{4m}[k-m(2p+1)]^2} \, 
\big[\theta_{k,m}+\theta_{-k,m}\big](\tau, z)}_{\rm (I)_A}
\nonumber
\\[2mm]
& & \hspace{-5mm}
\underbrace{- \, 
q^{-\frac{m}{4}}  
\theta_{2mp+m+\frac12, m+\frac12}^{(-)}(\tau,0) \hspace{-2mm} 
\sum_{\substack{r \, \in \zzz \\[1mm] -p \, \leq \, r \, \leq \, p}} \, 
\sum_{\substack{k \, \in \zzz \\[1mm] 0 \, \leq \, k \, < \, m}} 
\hspace{-2mm}
(-1)^k q^{-\frac{1}{4m}(2mr+k)(2m(r-1)+k)} 
\big[\theta_{k,m}+\theta_{-k,m}\big](\tau, z)}_{\rm (I)_B}
\nonumber
\\[3mm]
& & \hspace{-5mm}
+ \,\ 2 \, 
q^{-\frac{m}{4}} \, 
\theta_{2mp+m+\frac12, m+\frac12}^{(-)}(\tau,0)
\sum_{\substack{r \, \in \zzz \\[1mm] 0 \, \leq \, r \, \leq \, p-1}}
q^{-mr(r+1)} \, \theta_{0,m}(\tau, z)
\nonumber
\\[2mm]
& & \hspace{-5mm}
- \,\ (-1)^m \, q^{-\frac{m}{4}} \, 
\theta_{2mp+m+\frac12, m+\frac12}^{(-)}(\tau,0)
\sum_{\substack{r \, \in \zzz \\[1mm] -p \, \leq \, r \, \leq \, p}}
q^{-mr^2+\frac{m}{4}} \, \theta_{m,m}(\tau,z)
\label{m1:eqn:2022-921d1}
\end{eqnarray}}
Since
$$
- \, \frac{1}{4m} \, (2mr+k)(2m(r-1)+k) \,\ - \,\ \frac{m}{4}
\,\ = \,\ 
- \, \frac{1}{4m} \big(m(2r-1)+k\big)^2
$$
${\rm (I)}_B$  becomes as follows:
{\allowdisplaybreaks
\begin{eqnarray*}
& & \hspace{-8mm}
{\rm (I)}_B = 
- \, 
\theta_{2mp+m+\frac12, m+\frac12}^{(-)}(\tau,0) \hspace{-5mm} 
\underbrace{
\sum_{\substack{r \, \in \zzz \\[1mm] -p \, \leq \, r \, \leq \, p}}}_{
\substack{|| \\[-0.5mm] {\displaystyle \hspace{-1mm}
\sum_{\substack{r \, \in \zzz \\[1mm] -p \, < \, r \, \leq \, p}} + \sum_{r=-p}
}}} \hspace{-3mm}
\sum_{\substack{k \, \in \zzz \\[1mm] 0 \, \leq \, k \, < \, m}} \hspace{-2mm}
(-1)^k q^{- \, \frac{1}{4m} \, (m(2r-1)+k)^2} 
\big[\theta_{k,m}+\theta_{-k,m}\big](\tau, z)
\\[2mm]
&=&
- \, 
\theta_{2mp+m+\frac12, m+\frac12}^{(-)}(\tau,0) 
\sum_{\substack{r \, \in \zzz \\[1mm] -p \, < \, r \, \leq \, p}}
\sum_{\substack{k \, \in \zzz \\[1mm] 0 \, \leq \, k \, < \, m}} \hspace{-2mm}
(-1)^k q^{- \, \frac{1}{4m} \, \big(m(2r-1)+k\big)^2} 
\big[\theta_{k,m}+\theta_{-k,m}\big](\tau, z)
\\[2mm]
& &
\underbrace{- \, \theta_{2mp+m+\frac12, m+\frac12}^{(-)}(\tau,0) 
\sum_{\substack{k \, \in \zzz \\[1mm] 0 \, \leq \, k \, < \, m}} 
\hspace{-2mm}
(-1)^k q^{- \, \frac{1}{4m} \, \big(-m(2p+1)+k\big)^2} 
\big[\theta_{k,m}+\theta_{-k,m}\big](\tau, z)}_{
\substack{\uparrow \\[1mm] {\displaystyle - \, {\rm (I)}_A
}}}
\end{eqnarray*}}
so
{\allowdisplaybreaks
\begin{eqnarray}
{\rm (I)}_A+{\rm (I)}_B &=& 
- \, 
\theta_{2mp+m+\frac12, m+\frac12}^{(-)}(\tau,0) \hspace{-3mm}
\sum_{\substack{r \, \in \zzz \\[1mm] -p \, < \, r \, \leq \, p}}
\sum_{\substack{k \, \in \zzz \\[1mm] 0 \, \leq \, k \, < \, m}} 
\hspace{-2mm}
(-1)^k q^{- \, \frac{1}{4m} \, (m(2r-1)+k)^2} 
\nonumber
\\[1mm]
& & \hspace{10mm}
\times \,\ 
\big[\theta_{k,m}+\theta_{-k,m}\big](\tau, z)
\label{m1:eqn:2022-921d2}
\end{eqnarray}}
\end{subequations}
Substituting \eqref{m1:eqn:2022-921d2} into \eqref{m1:eqn:2022-921d1}
and noticing that
\begin{equation}
\theta_{2mp+m+\frac12, m+\frac12}^{(-)}(\tau,0) 
\, = \, 
(-1)^p \, \theta_{-p+m+\frac12, m+\frac12}^{(-)}(\tau,0)
\, = \, 
(-1)^p \, \theta_{p-m-\frac12, m+\frac12}^{(-)}(\tau, 0)
\label{m1:eqn:2022-921e}
\end{equation}
we obtain \eqref{m1:eqn:2022-921b}, proving the claim 1).

\vspace{2.5mm}

\noindent
2) In the case $\left\{
\begin{array}{lcc}
z_1 &=& \frac{z}{2}+\frac{\tau}{2} \\[1mm]
z_2 &=& \frac{z}{2}-\frac{\tau}{2}
\end{array}\right.$, substituting \eqref{m1:eqn:2022-920a} into 
\eqref{m1:eqn:2022-918d}, we have
{\allowdisplaybreaks
\begin{eqnarray*}
& & \hspace{-7mm}
\theta_{2mp+m+\frac12, m+\frac12}^{(-)}(\tau,0) \,\ 
\times \,\ \Bigg\{
q^{mp(p+1)} \, 
\widetilde{\Phi}^{[m,0] \, \ast}
\Big(\tau, \,\ 
\frac{z}{2}+\frac{\tau}{2}, \,\ 
\frac{z}{2}-\frac{\tau}{2}, \,\ 0\Big)
\\[2mm]
& & 
+ \,\ 
q^{mp(p+1)} 
\sum_{\substack{r \, \in \zzz \\[1mm] -p \, \leq \, r \, \leq \, p}} \, 
\sum_{\substack{k \, \in \zzz \\[1mm] 0 \, \leq \, k \, < \, m}} 
\hspace{-2mm}
q^{-\frac{1}{4m}(2mr+k)(2m(r-1)+k)} 
\big[\theta_{k,m}+\theta_{-k,m}\big](\tau, z)
\\[2mm]
& &
- \,\ 2 \, q^{mp(p+1)} 
\sum_{\substack{r \, \in \zzz \\[1mm] 0 \, \leq \, r \, \leq \, p-1}}
q^{-mr(r+1)} \, \theta_{0,m}(\tau, z)
\\[2mm]
& &
+ \,\ q^{mp(p+1)}
\sum_{\substack{r \, \in \zzz \\[1mm] -p \, \leq \, r \, \leq \, p}}
q^{-\frac{m}{4}(2r+1)(2r-1)} \, \theta_{m,m}(\tau,z)
\Bigg\}
\\[2mm]
&=&
(-1)^p \, 
q^{\frac{m}{4}(2p+1)^2} \, 
\frac{\eta(\tau)^3}{\theta_{0,\frac12}^{(-)}(\tau,z)} \cdot 
\big[\theta_{p, m+\frac12}^{(-)}+ \theta_{-p, m+\frac12}^{(-)}\big](\tau,0)
\\[3mm]
& & \hspace{-5mm}
- \,\ q^{\frac{m}{4}(2p+1)^2} \, 
\bigg[\sum_{\substack{j, \, r \, \in \, \frac12 \, \zzz_{\rm odd} \\[1mm]
0 \, \leq \, r \, < \, j}}
-
\sum_{\substack{j, \, r \, \in \, \frac12 \, \zzz_{\rm odd} \\[1mm]
j \, \leq \, r \, < \, 0}}\bigg]
\sum_{\substack{k \, \in \zzz \\[1mm] 0 \, < \, k \, < \, m}}
(-1)^{j-\frac12} \, 
q^{(m+\frac12)(j+\frac{2pm}{2m+1})^2}
\\[-1mm]
& & \hspace{60mm}
\times \,\ 
q^{-\frac{1}{4m} \, [2mr+k+2mp]^2}
\big[\theta_{k,m}+\theta_{-k,m}\big](\tau, z)
\\[3mm]
& & \hspace{-5mm}
- \,\ q^{\frac{m}{4}(2p+1)^2} \, \bigg[
\sum_{\substack{j, \, r \, \in \, \frac12 \, \zzz_{\rm odd} \\[1mm]
0 \, \leq \, r \, \leq \, j
}}
-\sum_{\substack{j, \, r \, \in \, \frac12 \, \zzz_{\rm odd} \\[1mm]
j \, < \, r \, < \, 0}}\bigg]
\sum_{\substack{k \, \in \zzz \\[1mm] 0 \, \leq \, k \, \leq \, m}}
(-1)^{j-\frac12} \, 
q^{(m+\frac12)(j+\frac{2pm}{2m+1})^2}
\\[-1mm]
& & \hspace{60mm}
\times \,\ 
q^{-\frac{1}{4m} \, [2mr-k+2mp]^2}
\big[\theta_{k,m}+\theta_{-k,m}\big](\tau, z)
\\[3mm]
& & \hspace{-5mm}
+ \,\ 
\theta_{2mp+m+\frac12, m+\frac12}^{(-)}(\tau,0)
\sum_{\substack{k \, \in \zzz \\[1mm] 0 \, \leq \, k \, \leq \, m}} 
q^{-\frac{1}{4m}k^2+\frac{k}{2}(2p+1)} \, 
\big[\theta_{k,m}+\theta_{-k,m}\big](\tau, z)
\end{eqnarray*}}
Multiplying \, $q^{-\frac{m}{4}(2p+1)^2} \, = \, q^{-mp(p+1)-\frac{m}{4}}$ \, 
to both sides, we have 
{\allowdisplaybreaks
\begin{eqnarray*}
& &
q^{-\frac{m}{4}} \, 
\theta_{2mp+m+\frac12, m+\frac12}^{(-)}(\tau,0) \, 
\widetilde{\Phi}^{[m,0] \, \ast}
\Big(\tau, \,\ 
\frac{z}{2}+\frac{\tau}{2}-\frac12, \,\ 
\frac{z}{2}-\frac{\tau}{2}+\frac12, \,\ 0\Big)
\\[2mm]
& & 
+ \,\ 
q^{-\frac{m}{4}} \, 
\theta_{2mp+m+\frac12, m+\frac12}^{(-)}(\tau,0) \, 
\\[1mm]
& & \hspace{20mm}
\times 
\sum_{\substack{r \, \in \zzz \\[1mm] -p \, \leq \, r \, \leq \, p}} \, 
\sum_{\substack{k \, \in \zzz \\[1mm] 0 \, \leq \, k \, < \, m}} \hspace{-2mm}
q^{-\frac{1}{4m}(2mr+k)(2m(r-1)+k)} 
\big[\theta_{k,m}+\theta_{-k,m}\big](\tau, z)
\\[3mm]
& &
- \,\ 2 \, 
q^{-\frac{m}{4}} \, 
\theta_{2mp+m+\frac12, m+\frac12}^{(-)}(\tau,0)
\sum_{\substack{r \, \in \zzz \\[1mm] 0 \, \leq \, r \, \leq \, p-1}}
q^{-mr(r+1)} \, \theta_{0,m}(\tau, z)
\\[2mm]
& &
+ \,\ 
q^{-\frac{m}{4}} \, 
\theta_{2mp+m+\frac12, m+\frac12}^{(-)}(\tau,0)
\sum_{\substack{r \, \in \zzz \\[1mm] -p \, \leq \, r \, \leq \, p}}
q^{-\frac{m}{4}(2r+1)(2r-1)} \, \theta_{m,m}(\tau,z)
\\[2mm]
&=&
(-1)^p \,\ 
\frac{\eta(\tau)^3}{\theta_{0,\frac12}^{(-)}(\tau,z)} \cdot 
\big[\theta_{p, m+\frac12}^{(-)}+ \theta_{-p, m+\frac12}^{(-)}\big](\tau,0)
\\[3mm]
& & \hspace{-5mm}
- \,\ 
\bigg[\sum_{\substack{j, \, r \, \in \, \frac12 \, \zzz_{\rm odd} \\[1mm]
0 \, \leq \, r \, < \, j}}
-\sum_{\substack{j, \, r \, \in \, \frac12 \, \zzz_{\rm odd} \\[1mm]
j \, \leq \, r \, < \, 0}}\bigg]
\sum_{\substack{k \, \in \zzz \\[1mm] 0 \, < \, k \, < \, m}}
(-1)^{j-\frac12} \, 
q^{(m+\frac12)(j+\frac{2pm}{2m+1})^2}
\\[-4mm]
& & \hspace{60mm}
\times \,\ 
q^{-\frac{1}{4m} \, [2mr+k+2mp]^2}
\big[\theta_{k,m}+\theta_{-k,m}\big](\tau, z)
\\[3mm]
& & \hspace{-5mm}
- \,\ \bigg[
\sum_{\substack{j, \, r \, \in \, \frac12 \, \zzz_{\rm odd} \\[1mm]
0 \, \leq \, r \, \leq \, j}}
-\sum_{\substack{j, \, r \, \in \, \frac12 \, \zzz_{\rm odd} \\[1mm]
j \, < \, r \, < \, 0}}\bigg]
\sum_{\substack{k \, \in \zzz \\[1mm] 0 \, \leq \, k \, \leq \, m}}
(-1)^{j-\frac12} \, 
q^{(m+\frac12)(j+\frac{2pm}{2m+1})^2}
\\[-4mm]
& & \hspace{60mm}
\times \,\ 
q^{-\frac{1}{4m} \, [2mr-k+2mp]^2}
\big[\theta_{k,m}+\theta_{-k,m}\big](\tau, z)
\\[3mm]
& & \hspace{-5mm}
+ \,\ 
\theta_{2mp+m+\frac12, m+\frac12}^{(-)}(\tau,0)
\sum_{\substack{k \, \in \zzz \\[1mm] 0 \, \leq \, k \, \leq \, m}} 
q^{-\frac{1}{4m}[k-m(2p+1)]^2} \, 
\big[\theta_{k,m}+\theta_{-k,m}\big](\tau, z)
\end{eqnarray*}}
namely
\begin{subequations}
{\allowdisplaybreaks
\begin{eqnarray}
& & \hspace{-7mm}
q^{-\frac{m}{4}} \, 
\theta_{2mp+m+\frac12, m+\frac12}^{(-)}(\tau,0) \, 
\widetilde{\Phi}^{[m,0] \, \ast}
\Big(\tau, \,\ 
\frac{z}{2}+\frac{\tau}{2}-\frac12, \,\ 
\frac{z}{2}-\frac{\tau}{2}+\frac12, \,\ 0\Big)
\nonumber
\\[2mm]
&=&
(-1)^p \,\ 
\frac{\eta(\tau)^3}{\theta_{0,\frac12}^{(-)}(\tau,z)} \cdot 
\big[\theta_{p, m+\frac12}^{(-)}+ \theta_{-p, m+\frac12}^{(-)}\big](\tau,0)
\nonumber
\\[3mm]
& & \hspace{-5mm}
- \,\ 
\bigg[\sum_{\substack{j, \, r \, \in \, \frac12 \, \zzz_{\rm odd} \\[1mm]
0 \, \leq \, r \, < \, j}}
-
\sum_{\substack{j, \, r \, \in \, \frac12 \, \zzz_{\rm odd} \\[1mm]
j \, \leq \, r \, < \, 0}}\bigg]
\sum_{\substack{k \, \in \zzz \\[1mm] 0 \, < \, k \, < \, m}}
(-1)^{j-\frac12} \, 
q^{(m+\frac12)(j+\frac{2pm}{2m+1})^2}
\nonumber
\\[-1mm]
& & \hspace{60mm}
\times \,\ 
q^{-\frac{1}{4m} \, [2mr+k+2mp]^2}
\big[\theta_{k,m}+\theta_{-k,m}\big](\tau, z)
\nonumber
\\[3mm]
& & \hspace{-5mm}
- \,\ \bigg[
\sum_{\substack{j, \, r \, \in \, \frac12 \, \zzz_{\rm odd} \\[1mm]
0 \, \leq \, r \, \leq \, j}}
-
\sum_{\substack{j, \, r \, \in \, \frac12 \, \zzz_{\rm odd} \\[1mm]
j \, < \, r \, < \, 0}}\bigg]
\sum_{\substack{k \, \in \zzz \\[1mm] 0 \, \leq \, k \, \leq \, m}}
(-1)^{j-\frac12} \, 
q^{(m+\frac12)(j+\frac{2pm}{2m+1})^2}
\nonumber
\\[-1mm]
& & \hspace{60mm}
\times \,\ 
q^{-\frac{1}{4m} \, [2mr-k+2mp]^2}
\big[\theta_{k,m}+\theta_{-k,m}\big](\tau, z)
\nonumber
\\[3mm]
& & \hspace{-5mm}
+ \,\ 
\underbrace{\theta_{2mp+m+\frac12, m+\frac12}^{(-)}(\tau,0)
\sum_{\substack{k \, \in \zzz \\[1mm] 0 \, \leq \, k \, \leq \, m}} 
q^{-\frac{1}{4m}[k-m(2p+1)]^2} \, 
\big[\theta_{k,m}+\theta_{-k,m}\big](\tau, z)}_{\rm (I)_A}
\nonumber
\\[2mm]
& & \hspace{-5mm}
\underbrace{- \, 
q^{-\frac{m}{4}}  
\theta_{2mp+m+\frac12, m+\frac12}^{(-)}(\tau,0) \hspace{-2mm} 
\sum_{\substack{r \, \in \zzz \\[1mm] -p \, \leq \, r \, \leq \, p}} \, 
\sum_{\substack{k \, \in \zzz \\[1mm] 0 \, \leq \, k \, < \, m}} 
\hspace{-2mm}
q^{-\frac{1}{4m}(2mr+k)(2m(r-1)+k)} 
\big[\theta_{k,m}+\theta_{-k,m}\big](\tau, z)}_{\rm (I)_B}
\nonumber
\\[3mm]
& & \hspace{-5mm}
+ \,\ 2 \, 
q^{-\frac{m}{4}} \, 
\theta_{2mp+m+\frac12, m+\frac12}^{(-)}(\tau,0)
\sum_{\substack{r \, \in \zzz \\[1mm] 0 \, \leq \, r \, \leq \, p-1}}
q^{-mr(r+1)} \, \theta_{0,m}(\tau, z)
\nonumber
\\[2mm]
& & \hspace{-5mm}
- \,\ q^{-\frac{m}{4}} \, 
\theta_{2mp+m+\frac12, m+\frac12}^{(-)}(\tau,0)
\sum_{\substack{r \, \in \zzz \\[1mm] -p \, \leq \, r \, \leq \, p}}
q^{-mr^2+\frac{m}{4}} \, \theta_{m,m}(\tau,z)
\label{m1:eqn:2022-922a}
\end{eqnarray}}
where ${\rm (I)}_B$ becomes as follows:
{\allowdisplaybreaks
\begin{eqnarray*}
& & \hspace{-8mm}
{\rm (I)}_B = 
- \, 
\theta_{2mp+m+\frac12, m+\frac12}^{(-)}(\tau,0) \hspace{-5mm} 
\underbrace{
\sum_{\substack{r \, \in \zzz \\[1mm] -p \, \leq \, r \, \leq \, p}}
}_{\substack{|| \\[-1.5mm] {\displaystyle \hspace{-1mm}
\sum_{\substack{r \, \in \zzz \\[1mm] -p \, < \, r \, \leq \, p}} 
+ \sum_{r=-p}
}}} \hspace{-3mm}
\sum_{\substack{k \, \in \zzz \\[1mm] 0 \, \leq \, k \, < \, m}} 
\hspace{-2mm}
q^{- \, \frac{1}{4m} \, \big(m(2r-1)+k\big)^2} 
\big[\theta_{k,m}+\theta_{-k,m}\big](\tau, z)
\\[2mm]
&=&
- \, 
\theta_{2mp+m+\frac12, m+\frac12}^{(-)}(\tau,0) 
\sum_{\substack{r \, \in \zzz \\[1mm] -p \, < \, r \, \leq \, p}} \,\ 
\sum_{\substack{k \, \in \zzz \\[1mm] 0 \, \leq \, k \, < \, m}} 
\hspace{-2mm}
q^{- \, \frac{1}{4m} \, \big(m(2r-1)+k\big)^2} 
\big[\theta_{k,m}+\theta_{-k,m}\big](\tau, z)
\\[2mm]
& &
\underbrace{- \, \theta_{2mp+m+\frac12, m+\frac12}^{(-)}(\tau,0) 
\sum_{\substack{k \, \in \zzz \\[1mm] 0 \, \leq \, k \, < \, m}} 
\hspace{-2mm}
q^{- \, \frac{1}{4m} \, \big(-m(2p+1)+k\big)^2} 
\big[\theta_{k,m}+\theta_{-k,m}\big](\tau, z)}_{\substack{|| \\[0mm] 
{\displaystyle \hspace{5mm}
- \, {\rm (I)}_A
}}}
\end{eqnarray*}}
so
{\allowdisplaybreaks
\begin{eqnarray}
{\rm (I)}_A+{\rm (I)}_B &=& 
- \, 
\theta_{2mp+m+\frac12, m+\frac12}^{(-)}(\tau,0) \hspace{-2mm}
\sum_{\substack{r \, \in \zzz \\[1mm] -p \, < \, r \, \leq \, p}} \,\
\sum_{\substack{k \, \in \zzz \\[1mm] 0 \, \leq \, k \, < \, m}} 
\hspace{-1mm}
q^{- \, \frac{1}{4m} \, \big(m(2r-1)+k\big)^2} 
\nonumber
\\[1mm]
& & \hspace{10mm}
\times \,\ 
\big[\theta_{k,m}+\theta_{-k,m}\big](\tau, z)
\label{m1:eqn:2022-922b}
\end{eqnarray}}
\end{subequations}
Substituting \eqref{m1:eqn:2022-922b} into \eqref{m1:eqn:2022-922a}
and using \eqref{m1:eqn:2022-921e}, we obtain 
\eqref{m1:eqn:2022-921c}, proving the claim 2).
\end{proof}

\section{Modular transformation of $\widetilde{\Phi}^{[m,0] \, \ast}$}
\label{sec:Phi:tilde:ast:modular}

\subsection{$\widetilde{\psi}^{(i)[m] \, \ast}(\tau,z)$}
\label{subsec:psi:tilde:ast}

For $m \in \nnn$ and $i \in \{1,2\}$, we consider functions 
$\widetilde{\psi}^{(i)[m] \, \ast}(\tau,z)$ defined by 
{\allowdisplaybreaks
\begin{eqnarray*}
\widetilde{\psi}^{(1)[m] \, \ast}(\tau,z) 
&:=& 
\widetilde{\Phi}^{[m,0] \, \ast}
\Big(\tau, \,\ \frac{z}{2}+\frac{\tau}{2}-\frac12, \,\ 
\frac{z}{2}-\frac{\tau}{2}+\frac12, \,\ 0\Big)
\\[2mm]
&=&
\widetilde{\Phi}^{[m,0] \, \ast}
\Big(\tau, \,\ \frac{z}{2}+\frac{\tau}{2}+\frac12, \,\ 
\frac{z}{2}-\frac{\tau}{2}-\frac12, \,\ 0\Big)
\\[3mm]
\widetilde{\psi}^{(2)[m] \, \ast}(\tau,z) 
&:=& 
\widetilde{\Phi}^{[m,0] \, \ast}
\Big(\tau, \,\ \frac{z}{2}+\frac{\tau}{2}, \,\ 
\frac{z}{2}-\frac{\tau}{2}, \,\ 0\Big)
\end{eqnarray*}}
The function $\widetilde{\psi}^{(1)[m] \, \ast}(\tau,z)$ satisfies 
the follwing modular transformation properties.

\vspace{1mm}

\begin{lemma} 
\label{m1:lemma:2022-923a}
Let $m \in \nnn$, then
\begin{enumerate}
\item[{\rm 1)}] \,\ $\widetilde{\psi}^{(1)[m] \, \ast} 
\Big(-\dfrac{1}{\tau}, \dfrac{z}{\tau}\Big) 
\,\ = \,\ 
(-1)^m \, \tau \, e^{\frac{\pi im}{2\tau}z^2} \, 
e^{-\frac{\pi im}{2\tau}} \, q^{-\frac{m}{4}} \, 
\widetilde{\psi}^{(1)[m] \, \ast}(\tau, z)$
\item[{\rm 2)}] \,\ $\widetilde{\psi}^{(1)[m] \, \ast}(\tau+1,z) 
\,\ = \,\
\widetilde{\psi}^{(2)[m] \, \ast}(\tau,z)$
\end{enumerate}
\end{lemma}

\begin{proof} These are obtained easily from Lemma 2.1 in \cite{W2022c}
as follows.

\medskip

\noindent
1) \, $\widetilde{\psi}^{(1)[m] \, \ast}
\Big(-\dfrac{1}{\tau}, \dfrac{z}{\tau}\Big) 
=
\widetilde{\Phi}^{[m, 0] \, \ast}\Big(-\dfrac{1}{\tau},
\dfrac{z}{2\tau}-\dfrac{1}{2\tau}+\dfrac12, 
\dfrac{z}{2\tau}+\dfrac{1}{2\tau}-\dfrac12, 0\Big)$
{\allowdisplaybreaks
\begin{eqnarray*}
&=& 
\widetilde{\Phi}^{[m, 0] \, \ast}\Big(-\dfrac{1}{\tau}, \, 
\frac{\frac{z}{2}-\frac12+\frac{\tau}{2}}{\tau}, \, 
\frac{\frac{z}{2}+\frac12-\frac{\tau}{2}}{\tau}, \, 0\Big)
\\[1mm]
&=& 
\tau \, e^{\frac{2\pi im}{\tau}
(\frac{z}{2}-\frac12+\frac{\tau}{2})
(\frac{z}{2}+\frac12-\frac{\tau}{2})}
\widetilde{\Phi}^{[m, 0] \, \ast}\Big(\tau, 
\frac{z}{2}-\frac12+\frac{\tau}{2},
\frac{z}{2}+\frac12-\frac{\tau}{2}, 0 \Big)
\\[1mm]
&=&
(-1)^m \, \tau \, e^{\frac{\pi im}{2\tau}z^2}
e^{-\frac{\pi im}{2\tau}}
q^{-\frac{m}{4}}
\widetilde{\psi}^{(1)[m] \, \ast}(\tau,z) \, ,  
\hspace{10mm} \text{proving 1)}.
\end{eqnarray*}}

\noindent
2) \, $\widetilde{\psi}^{(1)[m] \, \ast}(\tau+1,z) 
\, = \, 
\widetilde{\Phi}^{[m,0] \, \ast}\Big(\tau+1, \, 
\dfrac{z}{2}+\dfrac{\tau+1}{2}-\dfrac12, \, 
\dfrac{z}{2}-\dfrac{\tau+1}{2}+\dfrac12, \, 0\Big)$
{\allowdisplaybreaks
\begin{eqnarray*}
&=&
\widetilde{\Phi}^{[m,0] \, \ast}\Big(\tau+1, \, 
\frac{z}{2}+\frac{\tau}{2}, \, 
\frac{z}{2}-\frac{\tau}{2}, \, 0\Big)
\\[1mm]
&=&
\widetilde{\Phi}^{[m,0] \, \ast}\Big(\tau, \, 
\frac{z}{2}+\frac{\tau}{2}, \, 
\frac{z}{2}-\frac{\tau}{2}, \, 0\Big)
\, = \, \widetilde{\psi}^{(2)[m] \, \ast}(\tau,z) \, ,  
\hspace{10mm} \text{proving 2)}.
\end{eqnarray*}}
\end{proof}

\subsection{$\Xi^{(i)[m,p] \, \ast}(\tau,z)$ and 
$\Upsilon^{(i)[m,p] \, \ast}(\tau,z)$}
\label{subsec:Xi:ast}

For $m \in \nnn$ and $p \in \zzz$ such that 
$0 \leq p \leq 2m$ and $i \in \{1,2,3\}$, we define functions 
$\Xi^{(i)[m,p] \, \ast}(\tau,z)$ 
and $\Upsilon^{(i)[m,p] \, \ast}(\tau,z)$ as follows:
\begin{subequations}
{\allowdisplaybreaks
\begin{eqnarray}
\Xi^{(1)[m,p] \, \ast}(\tau,z) &:=&
(-1)^p \, q^{-\frac{m}{4}} \, 
\theta_{p-m-\frac12,m+\frac12}^{(-)}(\tau,0) \cdot 
\widetilde{\psi}^{(1)[m] \, \ast}(\tau,z)
\label{m1:eqn:2022-914a1}
\\[1mm]
\Xi^{(2)[m,p] \, \ast}(\tau,z) &:=&
q^{-\frac{m}{4}} \, 
\theta_{p-m-\frac12,m+\frac12}^{(-)}(\tau,0) \cdot 
\widetilde{\psi}^{(2)[m] \, \ast}(\tau,z)
\label{m1:eqn:2022-914a2}
\end{eqnarray}}
\end{subequations}
and
\begin{subequations}
{\allowdisplaybreaks
\begin{eqnarray}
\Upsilon^{(1)[m,p] \, \ast}(\tau,z) &:=&
\eta(\tau)^3 \, \cdot 
\frac{\theta_{p,m+\frac12}(\tau,z)
\, + \, 
\theta_{-p,m+\frac12}(\tau,z)}{
\theta_{0, \frac12}(\tau,z)}
\label{m1:eqn:2022-914b1}
\\[1mm]
\Upsilon^{(2)[m,p] \, \ast}(\tau,z) &:=&
\eta(\tau)^3 \, \cdot 
\frac{\theta_{p,m+\frac12}^{(-)}(\tau,z)
\, + \, 
\theta_{-p,m+\frac12}^{(-)}(\tau,z)
}{\theta_{0, \frac12}^{(-)}(\tau,z)}
\label{m1:eqn:2022-914b2}
\end{eqnarray}}
\end{subequations}

To compute modular transformation of these functions,
we use the following formulas which are obtained easily from 
Lemmas 1.3 and 1.4 in \cite{W2022c}.

\medskip

\begin{note}
\label{m1:note:2022-923a}
Let $m \in \zzz_{\geq 0}$ and $p \in \zzz$. Then
\begin{enumerate}
\item[{\rm 1)}] 
\begin{enumerate}
\item[{\rm (i)}] 
$\theta_{p-m-\frac12,m+\frac12}^{(-)}
\Big(-\dfrac{1}{\tau},0\Big) 
= \, 
- \, i \,  (-1)^{m+p} \, 
\dfrac{(-i\tau)^{\frac12}}{\sqrt{2m+1}}
\sum\limits_{p'=0}^{2m}
(-1)^{p'}e^{-\frac{2\pi i}{2m+1}pp'}
\theta_{p'-m-\frac12,m+\frac12}^{(-)}(\tau,0)$
\item[{\rm (ii)}]  $\theta_{p-m-\frac12,m+\frac12}^{(-)}
(\tau+1, \, z) \,\ = \,\ 
(-1)^p \, e^{\frac{\pi i}{2}(m+\frac12)} \,
e^{\frac{\pi i}{2m+1} p^2} \, 
\theta_{p-m-\frac12,m+\frac12}^{(-)}(\tau, z) $
\end{enumerate}
\item[{\rm 2)}] 
\begin{enumerate}
\item[{\rm (i)}] $\big[\theta_{p,m+\frac12}+ 
\theta_{-p,m+\frac12}\big]\Big(-\dfrac{1}{\tau}, \dfrac{z}{\tau}\Big)$
{\allowdisplaybreaks
\begin{eqnarray*}
&=& 
\frac{(-i\tau)^{\frac12}}{\sqrt{2m+1}} \, 
e^{\frac{\pi i}{2\tau}(m+\frac12)z^2}
\sum_{p' \in \zzz/(2m+1)\zzz}
e^{\frac{2\pi i}{2m+1}pp'} \, \big[
\theta_{p',m+\frac12}+ \theta_{-p',m+\frac12}\big](\tau,z)
\\[1mm]
&=& 
\frac{(-i\tau)^{\frac12}}{\sqrt{2m+1}} \, 
e^{\frac{\pi i}{2\tau}(m+\frac12)z^2}
\sum_{p' \in \zzz/(2m+1)\zzz}
e^{-\frac{2\pi i}{2m+1}pp'} \, \big[
\theta_{p',m+\frac12}+ \theta_{-p',m+\frac12}\big](\tau,z)
\end{eqnarray*}}
\item[{\rm (ii)}] $\big[\theta_{p,m+\frac12}+ 
\theta_{-p,m+\frac12}\big](\tau+1,z) 
\, = \, 
e^{\frac{\pi i}{2m+1}p^2} \, 
\big[\theta_{p,m+\frac12}^{(-)}+ \theta_{-p,m+\frac12}^{(-)}\big](\tau,z) $
\end{enumerate}
\end{enumerate}
\end{note}

\medskip

Then modular transformation properties of $\Xi^{(1)[m,p] \, \ast}$ 
and $\Upsilon^{(1)[m,p] \, \ast}$ are given by the following formulas:

\vspace{1mm}

\begin{lemma} 
\label{m1:lemma:2022-923b}
Let $m \in \nnn$ and $p \in \zzz$. Then 
\begin{enumerate}
\item[{\rm 1)}] \,\ $\Xi^{(1)[m,p] \, \ast}
\Big(-\dfrac{1}{\tau}, \dfrac{z}{\tau}\Big)
\, = \, 
\dfrac{(-i\tau)^{\frac32}}{\sqrt{2m+1}} \, 
e^{\frac{\pi im}{2\tau}z^2} \, 
\sum\limits_{p'=0}^{2m}
e^{-\frac{2\pi i}{2m+1}pp'} \, \Xi^{(1)[m,p'] \, \ast}(\tau,z)$
\item[{\rm 2)}] \,\ $\Xi^{(1)[m,p] \, \ast}(\tau+1,z) \, = \, 
e^{\frac{\pi i}{4}} \, e^{\frac{\pi i}{2m+1} p^2} \, 
\Xi^{(2)[m,p] \, \ast}(\tau,z)$
\end{enumerate}
\end{lemma}

\begin{proof} 1) \,\ $\Xi^{(1)[m,p] \, \ast}
\Big(-\dfrac{1}{\tau}, \dfrac{z}{\tau}\Big)$
{\allowdisplaybreaks
\begin{eqnarray*}
&=&
(-1)^p e^{-2\pi i \frac{m}{4}(-\frac{1}{\tau})} \, 
\theta_{p-m-\frac12,m+\frac12}^{(-)}\Big(-\dfrac{1}{\tau}, 0\Big) \, 
\widetilde{\psi}^{(1)[m] \, \ast}\Big(-\dfrac{1}{\tau}, \dfrac{z}{\tau}\Big)
\\[2mm]
&=&
(-1)^p \, e^{\frac{\pi im}{2\tau}} 
\, \times \,\bigg\{
- \, i \,  (-1)^{m+p} \, 
\frac{(-i\tau)^{\frac12}}{\sqrt{2m+1}}
\sum_{p'=0}^{2m}
(-1)^{p'}e^{-\frac{2\pi i}{2m+1}pp'}
\theta_{p'-m-\frac12,m+\frac12}^{(-)}(\tau,0)\bigg\}
\\[2mm]
& &
\times \,\ 
(-1)^m \, \tau \, e^{\frac{\pi im}{2\tau}z^2} \, 
e^{-\frac{\pi im}{2\tau}} \, q^{-\frac{m}{4}} \, 
\widetilde{\psi}^{(1)[m] \, \ast}(\tau, z)
\\[2mm]
&=&
\frac{(-i\tau)^{\frac32}}{\sqrt{2m+1}} \, 
e^{\frac{\pi im}{2\tau}z^2} \, 
\sum_{p'=0}^{2m}
e^{-\frac{2\pi i}{2m+1}pp'} \, 
\underbrace{(-1)^{p'} \, q^{-\frac{m}{4}} \, 
\theta_{p'-m-\frac12,m+\frac12}^{(-)}(\tau,0) \, 
\widetilde{\psi}^{(1)[m] \, \ast}(\tau, z)}_{\substack{|| \\[0mm] 
{\displaystyle \Xi^{(1)[m,p'] \, \ast}(\tau,z)
}}}
\end{eqnarray*}}

\vspace{-8mm}

proving 1).

\vspace{4mm}

\noindent
2) \,\ $\Xi^{(1)[m,p] \, \ast}(\tau+1,z) \, = \, 
(-1)^p \, 
e^{-\frac{\pi im}{2}(\tau+1)} \, 
\theta_{p-m-\frac12, m+\frac12}^{(-)}(\tau+1,0) \, 
\widetilde{\psi}^{(1)[m] \, \ast}(\tau+1, z)$
{\allowdisplaybreaks
\begin{eqnarray*}
&=&
(-1)^p \, q^{-\frac{m}{4}} \, e^{-\frac{\pi im}{2}} \, 
\times \, \Big\{
(-1)^p \, e^{\frac{\pi i}{2}(m+\frac12)} \,
e^{\frac{\pi i}{2m+1} p^2} \, 
\theta_{p-m-\frac12,m+\frac12}^{(-)}(\tau, z) \Big\}
\, \widetilde{\psi}^{(2)[m] \, \ast}(\tau, z)
\\[2mm]
&=&
e^{\frac{\pi i}{4}} \,
e^{\frac{\pi i}{2m+1} p^2} \, 
\underbrace{q^{-\frac{m}{4}} \, 
\theta_{p-m-\frac12,m+\frac12}^{(-)}(\tau, z) \, 
\widetilde{\psi}^{(2)[m] \, \ast}(\tau, z)}_{\substack{|| \\[0mm] 
{\displaystyle \Xi^{(2)[m,p] \, \ast}(\tau,z)
}}} \, , \hspace{10mm} \text{proving 2).}
\end{eqnarray*}}

\vspace{-9mm}

\end{proof}

\vspace{1mm}

\begin{lemma}
\label{m1:lemma:2022-923c}
Let $m \in \nnn$ and $p \in \zzz$. Then 
\begin{enumerate}
\item[{\rm 1)}]$\Upsilon^{(1)[m,p] \, \ast}
\Big(-\dfrac{1}{\tau}, \dfrac{z}{\tau}\Big)
\,\ = \,\ 
\dfrac{(-i\tau)^{\frac32}}{\sqrt{2m+1}} \, 
e^{\frac{\pi im}{2\tau}z^2}
\sum\limits_{p'=0}^{2m}
e^{-\frac{2\pi i}{2m+1}pp'} \, \Upsilon^{(1)[m,p'] \, \ast}(\tau,z)$

\item[{\rm 2)}] \,\ $\Upsilon^{(1)[m,p] \, \ast}(\tau+1,z)
\,\ = \,\ 
e^{\frac{\pi i}{2m+1}p^2+\frac{\pi i}{4}} \, 
\Upsilon^{(2)[m,p] \, \ast}(\tau, z)$
\end{enumerate}
\end{lemma}

\begin{proof} 1) \,\ $\Upsilon^{(1)[m,p] \, \ast}
\Big(-\dfrac{1}{\tau}, \dfrac{z}{\tau}\Big)
=
\eta (-\frac{1}{\tau})^3 \cdot 
\dfrac{
\theta_{p,m+\frac12}(-\frac{1}{\tau}, \frac{z}{\tau}) 
\, + \, 
\theta_{-p,m+\frac12}(-\frac{1}{\tau}, \frac{z}{\tau})
}{
\theta_{0, \frac12}(-\frac{1}{\tau}, \, \frac{z}{\tau})}$
{\allowdisplaybreaks
\begin{eqnarray*}
&=&
(-i\tau)^{\frac32}\eta(\tau)^3 \, 
\frac{(-i\tau)^{\frac12}}{\sqrt{2m+1}} \, 
e^{\frac{\pi i}{2\tau}(m+\frac12)z^2} \hspace{-5mm}
\sum_{p' \in \zzz(2m+1)\zzz} \hspace{-5mm}
e^{-\frac{2\pi i}{2m+1}pp'} \, 
\frac{\big[
\theta_{p',m+\frac12}+ \theta_{-p',m+\frac12}\big](\tau,z)}
{(-i\tau)^{\frac12} \, e^{\frac{\pi iz^2}{4\tau}} \, 
\theta_{0, \frac12}(\tau, z)}
\\[0mm]
&=&
\frac{(-i\tau)^{\frac32}}{\sqrt{2m+1}} \, e^{\frac{\pi im}{2\tau}z^2}
\sum_{p' \in \zzz(2m+1)\zzz}
e^{-\frac{2\pi i}{2m+1}pp'} \, \bigg\{
\underbrace{\eta(\tau)^3 \, 
\frac{\big[\theta_{p',m+\frac12}+ \theta_{-p',m+\frac12}\big](\tau,z)}
{\theta_{0, \frac12}(\tau, z)}}_{\substack{|| \\[0mm] 
{\displaystyle \Upsilon^{(1)[m,p'] \, \ast}(\tau,z)
}}}
\bigg\}
\end{eqnarray*}}

\vspace{-9mm}

proving 1).

\vspace{4mm}
\noindent
2)  \,\ $\Upsilon^{(1)[m,p] \, \ast}(\tau+1, z) = \, 
\eta(\tau+1)^3 \, 
\dfrac{
[\theta_{p, m+\frac12}+\theta_{-p, m+\frac12}](\tau+1,z)
}{
\theta_{0, \frac12}(\tau+1,z)} $
{\allowdisplaybreaks
\begin{eqnarray*}
&=&
[e^{\frac{\pi i}{12}} \eta(\tau)]^3 \, \frac{
e^{\frac{\pi i}{2m+1}p^2} \, 
[\theta_{p, m+\frac12}^{(-)}+\theta_{-p, m+\frac12}^{(-)}](\tau,z)
}{
\theta_{0, \frac12}^{(-)}(\tau,z)} 
\\[1mm]
&=& 
e^{\frac{\pi i}{2m+1}p^2+\frac{\pi i}{4}} \,\ 
\underbrace{\eta(\tau)^3 \, 
\frac{[\theta_{p, m+\frac12}^{(-)}
-
\theta_{-p, m+\frac12}^{(-)}](\tau,z)}{
\theta_{0, \frac12}^{(-)}(\tau,z)}}_{\substack{|| \\[0mm] 
{\displaystyle \Upsilon^{(2)[m,p] \, \ast}(\tau, z)
}}} \, , \hspace{10mm} \text{proving 2).}
\end{eqnarray*}}

\vspace{-10mm}

\end{proof}

\section{$G^{(i)[m,p] \, \ast}(\tau,z)$ and $g^{(i)[m,p] \, \ast}_k(\tau)$}
\label{sec:g(i):ast}

\subsection{$G^{(i)[m,p] \, \ast}(\tau,z)$}
\label{subsec:G:ast}


For $m \in \nnn$ and $p \in \zzz$ such that 
$0 \leq p \leq 2m$ and $i \in \{1,2\}$, we put 
\begin{subequations}
\begin{equation}
G^{(i)[m,p] \, \ast}(\tau,z) \,\ := \,\ 
\Xi^{(i)[m,p] \, \ast}(\tau,z) -\Upsilon^{(i)[m,p] \, \ast}(\tau,z)
\label{m1:eqn:2022-923a1}
\end{equation}
Then, by Proposition \ref{m1:prop:2022-921b}, $G^{(i)[m,p] \, \ast}(\tau,z)$ 
can be written in the following form:
{\allowdisplaybreaks
\begin{eqnarray}
& & \hspace{-20mm}
G^{(i)[m,p] \, \ast}(\tau,z) \,\ = \,\  
\sum_{\substack{k \, \in \zzz \\[1mm]
0 \, \leq \, k \, \leq \, m}}g^{(i)[m,p] \, \ast}_k(\tau)
\big[\theta_{k,m}+\theta_{-k,m}\big](\tau,z)
\nonumber
\\[1mm]
& & \hspace{-8mm}
= \hspace{-2mm}
\sum_{\substack{k \, \in \zzz \\[1mm] 0 \, < \, k \, < \, m}}
\hspace{-2mm}
g^{(i)[m,p] \, \ast}_k(\tau)
\big[\theta_{k,m}+\theta_{-k,m}\big](\tau,z)
\, + \, 
2 \, g^{(i)[m,p] \, \ast}_0(\tau) \theta_{0,m}(\tau,z)
\nonumber
\\[-2mm]
& & \hspace{30mm}
+ \,\ 
2 \, g^{(i)[m,p] \, \ast}_m(\tau) \theta_{m,m}(\tau,z)
\label{m1:eqn:2022-923a2}
\end{eqnarray}}
\end{subequations}
The modular transformation properties of $G^{(1)[m,p] \, \ast}(\tau,z)$
are obtained immediately from \eqref{m1:eqn:2022-923a1} and 
Lemma \ref{m1:lemma:2022-923b} and Lemma \ref{m1:lemma:2022-923c} 
as follows:

\medskip

\begin{prop} 
\label{m1:prop:2022-923a}
Let $m \in \nnn$ and $p \in \zzz_{\geq 0}$ such 
that $0 \leq p \leq 2m$. Then 
\begin{subequations}
\begin{enumerate}
\item[{\rm 1)}] \,\ $G^{(1)[m,p] \, \ast}
\Big(-\dfrac{1}{\tau}, \dfrac{z}{\tau}\Big)
= 
\dfrac{(-i\tau)^{\frac32}}{\sqrt{2m+1}} 
e^{\frac{\pi im}{2\tau}z^2} 
\sum\limits_{p'=0}^{2m} e^{-\frac{2\pi i}{2m+1}pp'} 
G^{(1)[m,p'] \, \ast}(\tau,z)$

\vspace{-10mm}

\begin{equation}
\label{m1:eqn:2022-923b1}
\end{equation}
\item[{\rm 2)}] \,\ $G^{(1)[m,p] \, \ast}(\tau+1, z) 
\,\ = \,\ 
e^{\frac{\pi i}{2m+1}p^2+\frac{\pi i}{4}} \, 
G^{(2)[m, p] \, \ast}(\tau,z)$

\vspace{-10mm}

\begin{equation}
\label{m1:eqn:2022-1004a}
\end{equation}
\end{enumerate}
\end{subequations}
\end{prop}

The formula \eqref{m1:eqn:2022-923b1} is rewritten as follows:

\medskip

\begin{prop} 
\label{m1:prop:2022-923b}
Let $m \in \nnn$ and $\ell \in \zzz_{\geq 0}$ such 
that $0 \leq \ell \leq 2m$. Then 
\begin{equation}
\sum_{p=0}^{2m} \, 
e^{\frac{2\pi i}{2m+1}p\ell} \, 
G^{(1)[m,p] \, \ast}\Big(-\frac{1}{\tau}, \frac{z}{\tau}\Big)
\, = \,
(-i\tau)^{\frac32} \sqrt{2m+1} \, 
e^{\frac{\pi im}{2\tau}z^2} \, G^{(1)[m,\ell] \, \ast}(\tau,z)
\label{m1:eqn:2022-923b2}
\end{equation}
\end{prop}

\begin{proof} \,\ Applying \, 
$\sum_{p=0}^{2m} e^{\frac{2\pi i}{2m+1}p\ell}$ \, 
to \eqref{m1:eqn:2022-923b1}, we have
{\allowdisplaybreaks
\begin{eqnarray*}
& & \hspace{-10mm}
\sum_{p=0}^{2m} e^{\frac{2\pi i}{2m+1}p\ell} \, 
G^{(1)[m,p] \, \ast}\Big(-\dfrac{1}{\tau}, \dfrac{z}{\tau}\Big)
\\[2mm]
&=&
\sum_{p=0}^{2m} e^{\frac{2\pi i}{2m+1}p\ell} \, 
\frac{(-i\tau)^{\frac32}}{\sqrt{2m+1}} 
e^{\frac{\pi im}{2\tau}z^2} 
\sum_{p'=0}^{2m} e^{-\frac{2\pi i}{2m+1}pp'} \, 
G^{(1)[m,p'] \, \ast}(\tau,z)
\\[2mm]
&=&
\frac{(-i\tau)^{\frac32}}{\sqrt{2m+1}} \, 
e^{\frac{\pi im}{2\tau}z^2} 
\sum_{p'=0}^{2m} 
\underbrace{
\sum_{p=0}^{2m} e^{\frac{2\pi i}{2m+1}p\ell} e^{-\frac{2\pi i}{2m+1}pp'}
}_{\substack{|| \\[0mm] {\displaystyle (2m+1) \, \delta_{p', \ell}
}}} \, 
G^{(1)[m,p'] \, \ast}(\tau,z)
\\[2mm]
&=&
(-i\tau)^{\frac32} \sqrt{2m+1} \, 
e^{\frac{\pi im}{2\tau}z^2} \, G^{(1)[m,\ell] \, \ast}(\tau,z)
\, , \hspace{10mm} 
\text{proving Proposition \ref{m1:prop:2022-923b}.}
\end{eqnarray*}}

\vspace{-12mm}

\end{proof}

\subsection{Modular transformation of $g^{(i)[m,p] \, \ast}_k(\tau)$}
\label{subsec:g:ast:transf}

To compute modular transformation of $g^{(1)[m,p] \, \ast}_k(\tau)$, 
we use the following formulas which are obtained easily from 
Lemmas 1,3 and 1.4 in \cite{W2022c}.

\medskip

\begin{note} 
\label{m1:note:2022-923b}
Let $m \in \nnn$ and $k \in \zzz$. Then 
\begin{enumerate}
\item[{\rm 1)}]
\begin{enumerate}
\item[{\rm (i)}] $\big[\theta_{k,m}+\theta_{-k,m}\big] 
\Big(-\dfrac{1}{\tau}, \dfrac{z}{\tau}\Big)
\, = \, 
\dfrac{(-i\tau)^{\frac12}}{\sqrt{2m}} \, e^{\frac{\pi im}{2\tau}z^2}
\sum\limits_{j \in \zzz/2m\zzz} \cos \dfrac{\pi jk}{m} \, 
\big[\theta_{j,m}+ \theta_{-j,m}\big](\tau,z)$
{\allowdisplaybreaks
\begin{eqnarray*}
& & \hspace{-13mm}
= \,\ (-i\tau)^{\frac12} \, \sqrt{\frac{2}{m}} \, 
e^{\frac{\pi im}{2\tau}z^2} 
\\[1mm]
& & \hspace{-8mm}
\times \,\ \Bigg\{
\sum_{\substack{j \, \in \zzz \\[1mm] 0 \, < \, j \, < \, m}}
\cos \frac{\pi jk}{m} 
\big[\theta_{j,m}+ \theta_{-j,m}\big](\tau,z)
\, + \, \theta_{0,m}(\tau,z) 
\, + \, (-1)^k \, \theta_{m,m}(\tau,z)\Bigg\}
\end{eqnarray*}}
\item[{\rm (ii)}] $\theta_{0,m}
\Big(-\dfrac{1}{\tau}, \dfrac{z}{\tau}\Big) 
= \, 
\dfrac{(-i\tau)^{\frac12}}{\sqrt{2m}} \, e^{\frac{\pi im}{2\tau}z^2} 
\Bigg\{
\sum\limits_{\substack{j \, \in \zzz \\[1mm] 0 \, < \, j \, < \, m}} 
\hspace{-3mm}
\big[\theta_{j,m}+ \theta_{-j,m}\big](\tau,z)
\, + \, \theta_{0,m}(\tau,z) 
\, + \, \theta_{m,m}(\tau,z)\Bigg\}$
\item[{\rm (iii)}] $\theta_{m,m}
\Big(-\dfrac{1}{\tau}, \dfrac{z}{\tau}\Big) \,\ =$
$$
\frac{(-i\tau)^{\frac12}}{\sqrt{2m}} \, e^{\frac{\pi im}{2\tau}z^2} 
\Bigg\{
\sum_{\substack{j \, \in \zzz \\[1mm] 0 \, < \, j \, < \, m}} \hspace{-3mm}
(-1)^j \, 
\big[\theta_{j,m}+ \theta_{-j,m}\big](\tau,z)
\, + \, \theta_{0,m}(\tau,z) 
\, + \, (-1)^m \, \theta_{m,m}(\tau,z)\Bigg\}
$$
\end{enumerate}
\item[{\rm 2)}]
\begin{enumerate}
\item[{\rm (i)}] \,\ $\big[\theta_{k,m}+\theta_{-k,m}\big](\tau+1,z) 
\,\ = \,\ 
e^{\frac{\pi i}{2m}k^2} \, \big[\theta_{k,m}+\theta_{-k,m}\big](\tau,z)$
\item[{\rm (ii)}] \,\ $\theta_{0,m}(\tau+1,z) 
\,\ = \,\ \theta_{0,m}(\tau,z)$
\item[{\rm (iii)}] \,\ $\theta_{m,m}(\tau+1,z) \,\ = \,\ 
e^{\frac{\pi im}{2}} \, \theta_{m,m}(\tau,z)$
\end{enumerate}
\end{enumerate}
\end{note}

\medskip

The relation between modular transformation properties of 
$G^{(i)[m,p] \, \ast}(\tau,z)$ and those of 
$g^{(i)[m,p] \, \ast}_k(\tau)$, for $i \in \{1,2\}$, 
is obtained by using the above formulas as follows :

\medskip


\begin{lemma} 
\label{m1:lemma:2022-924a}
Let $m \in \nnn$ and $p \in \zzz$ and $i \in \{1,2\}$ Then 
\begin{subequations}
\begin{enumerate}
\item[{\rm 1)}] \,\ $G^{(i)[m,p] \, \ast}
\Big(-\dfrac{1}{\tau}, \dfrac{z}{\tau}\Big)
\,\ = \,\ 
(-i\tau)^{\frac12} \, \sqrt{\dfrac{2}{m}} \, 
e^{\frac{\pi im}{2\tau}z^2} \, \Bigg( $
{\allowdisplaybreaks
\begin{eqnarray}
& & \hspace{-8mm}
\sum_{\substack{k \in \zzz \\[1mm] 0<k<m}} \hspace{-2mm}
g^{(i)[m,p] \, \ast}_k\Big(-\frac{1}{\tau}\Big)
\bigg\{ \hspace{-2mm}
\sum_{\substack{j \in \zzz \\[1mm] 0<j<m}} \hspace{-2mm}
\cos \frac{\pi jk}{m} \, 
\big[\theta_{j,m}+ \theta_{-j,m}\big](\tau,z)
\, + \, \theta_{0,m}(\tau,z) 
\, + \, (-1)^k \, \theta_{m,m}(\tau,z)\bigg\}
\nonumber
\\[2mm]
& & \hspace{-5mm}
+ \,\ 
g^{(i)[m,p] \, \ast}_0\Big(-\frac{1}{\tau}\Big) \bigg\{
\sum_{\substack{j \in \zzz \\[1mm] 0<j<m}}
\big[\theta_{j,m}+ \theta_{-j,m}\big](\tau,z)
\, + \, \theta_{0,m}(\tau,z) 
\, + \, \theta_{m,m}(\tau,z)\bigg\}
\nonumber
\\[2mm]
& & \hspace{-5mm}
+ \,\ 
g^{(i)[m,p] \, \ast}_m\Big(-\frac{1}{\tau}\Big) \bigg\{ \hspace{-2mm}
\sum_{\substack{j \in \zzz \\[1mm] 0<j<m}} \hspace{-2mm}
(-1)^j \, 
\big[\theta_{j,m}+ \theta_{-j,m}\big](\tau,z)
\, + \, \theta_{0,m}(\tau,z) 
\, + \, (-1)^m \, \theta_{m,m}(\tau,z)\bigg\}\Bigg)
\nonumber
\\[-2mm]
& &
\label{m1:eqn:2022-924a1}
\end{eqnarray}}
\item[{\rm 2)}] \,\ $G^{(i)[m,p] \, \ast}(\tau+1,z)
\,\ = 
\sum\limits_{\substack{k \in \zzz \\[1mm] 0<k<m}}
e^{\frac{\pi i}{2m}k^2}
g^{(i)[m,p] \, \ast}_k(\tau+1)
\big[\theta_{k,m}+\theta_{-k,m}\big](\tau,z)$
\begin{equation}
+ \,\ 
2 \, g^{(i)[m,p] \, \ast}_0(\tau+1) \, \theta_{0,m}(\tau,z)
\,\ + \,\ 
2 \, e^{\frac{\pi im}{2}} \, 
g^{(i)[m,p] \, \ast}_m(\tau+1) \, \theta_{m,m}(\tau,z)
\label{m1:eqn:2022-924a2}
\end{equation}
\end{enumerate}
\end{subequations}
\end{lemma}

\begin{proof} By \eqref{m1:eqn:2022-923a2} and Note 
\ref{m1:note:2022-923b}, we have
{\allowdisplaybreaks
\begin{eqnarray*}
& & \hspace{-7mm}
G^{(i)[m,p] \, \ast}\Big(-\frac{1}{\tau}, \frac{z}{\tau}\Big)
\,\ = \,\ 
\sum_{\substack{k \in \zzz \\[1mm] 0<k<m}}
g^{(i)[m,p] \, \ast}_k\Big(-\frac{1}{\tau}\Big)
\big[\theta_{k,m}+\theta_{-k,m}\big]\Big(-\frac{1}{\tau}, \frac{z}{\tau}\Big)
\\[3mm]
& &
+ \,\ 
2 \, g^{(i)[m,p] \, \ast}_0\Big(-\frac{1}{\tau}\Big) \, 
\theta_{0,m}\Big(-\frac{1}{\tau}, \frac{z}{\tau}\Big)
\, + \, 
2 \, g^{(i)[m,p] \, \ast}_m\Big(-\frac{1}{\tau}\Big) \, 
\theta_{m,m}\Big(-\frac{1}{\tau}, \frac{z}{\tau}\Big)
\\[3mm]
&=&
(-i\tau)^{\frac12} \, \sqrt{\frac{2}{m}} \, 
e^{\frac{\pi im}{2\tau}z^2} \, \Bigg(
\\[2mm]
& & \hspace{-10mm}
\sum_{\substack{k \in \zzz \\[1mm] 0<k<m}}
g^{(i)[m,p] \, \ast}_k\Big(-\frac{1}{\tau}\Big)
\bigg\{
\sum_{\substack{j \in \zzz \\[1mm] 0<j<m}}
\cos \frac{\pi jk}{m} \, 
\big[\theta_{j,m}+ \theta_{-j,m}\big](\tau,z)
\, + \, \theta_{0,m}(\tau,z) 
\, + \, (-1)^k \, \theta_{m,m}(\tau,z)\bigg\}
\\[2mm]
& & \hspace{-5mm}
+ \,\ 
g^{(i)[m,p] \, \ast}_0\Big(-\frac{1}{\tau}\Big) \bigg\{
\sum_{\substack{j \in \zzz \\[1mm] 0<j<m}}
\big[\theta_{j,m}+ \theta_{-j,m}\big](\tau,z)
\, + \, \theta_{0,m}(\tau,z) 
\, + \, \theta_{m,m}(\tau,z)\bigg\}
\\[2mm]
& & \hspace{-5mm}
+ \,\ 
g^{(i)[m,p] \, \ast}_m\Big(-\frac{1}{\tau}\Big) \bigg\{
\sum_{\substack{j \in \zzz \\[1mm] 0<j<m}}
(-1)^j \, 
\big[\theta_{j,m}+ \theta_{-j,m}\big](\tau,z)
\, + \, \theta_{0,m}(\tau,z) 
\, + \, (-1)^m \, \theta_{m,m}(\tau,z)\bigg\}\Bigg)
\end{eqnarray*}}
and
{\allowdisplaybreaks
\begin{eqnarray*}
& & \hspace{-17mm}
G^{(i)[m,p] \, \ast}(\tau+1,z)
\,\ = \,\ 
\sum_{\substack{k \in \zzz \\[1mm] 0<k<m}}
g^{(i)[m,p] \, \ast}_k(\tau+1)
\underbrace{\big[\theta_{k,m}+\theta_{-k,m}\big](\tau+1,z)}_{
\substack{|| \\[-0.5mm] {\displaystyle 
e^{\frac{\pi i}{2m}k^2}\big[\theta_{k,m}+\theta_{-k,m}\big](\tau,z)
}}}
\\[2mm]
& &
+ \,\ 
2 \, g^{(i)[m,p] \, \ast}_0(\tau+1) \, 
\underbrace{\theta_{0,m}(\tau+1,z)}_{\substack{|| \\[-0.5mm] 
{\displaystyle \hspace{3mm}
\theta_{0,m}(\tau,z)
}}}
\, + \, 
2 \, g^{(i)[m,p] \, \ast}_m(\tau+1) \, 
\underbrace{\theta_{m,m}(\tau+1,z)}_{\substack{|| \\[-1.5mm] 
{\displaystyle e^{\frac{\pi im}{2}}\theta_{m,m}(\tau,z)
}}}
\\[3mm]
&=&
\sum_{\substack{k \in \zzz \\[1mm] 0<k<m}}
e^{\frac{\pi i}{2m}k^2}
g^{(i)[m,p] \, \ast}_k(\tau+1)
\big[\theta_{k,m}+\theta_{-k,m}\big](\tau,z)
\\[2mm]
& &
+ \,\ 
2 \, g^{(i)[m,p] \, \ast}_0(\tau+1) \, \theta_{0,m}(\tau,z)
\,\ + \,\ 
2 \, e^{\frac{\pi im}{2}} \, 
g^{(i)[m,p] \, \ast}_m(\tau+1) \, \theta_{m,m}(\tau,z)
\end{eqnarray*}}
proving Lemma \ref{m1:lemma:2022-924a}.
\end{proof}

\vspace{1mm}

\begin{prop} 
\label{m1:prop:2022-924a}
Let $m \in \nnn$ and $p, k \in \zzz$ such that 
$0 \, \leq \, p \, \leq \, 2m$ and
$0 \, \leq \, k \, \leq \, m$. Then 
\begin{enumerate}
\item[{\rm 1)}] \,\ $S$-transformation of $g^{(1)[m,p] \, \ast}_k(\tau)$ 
are as follows:
\begin{enumerate}
\item[{\rm (i)}] \,\ $
\underset{\substack{\\[1mm] (0<j<m)}}{g^{(1)[m,p] \, \ast}_j
\Big(-\dfrac{1}{\tau}\Big)} 
\,\ = \,\
\dfrac{-i\tau}{\sqrt{m(m+\frac12)}}
\sum\limits_{p'=0}^{2m} \, 
\sum\limits_{\substack{k \in \zzz \\[1mm] 0<k<m}}
e^{\frac{2\pi i}{2m+1}pp'} \, 
\cos \dfrac{\pi jk}{m} \, g^{(1)[m,p'] \, \ast}_k(\tau)$
$$
+ \,\ 
\frac{-i\tau}{\sqrt{m(m+\frac12)}}
\sum_{p'=0}^{2m} \, e^{\frac{2\pi i}{2m+1}pp'} \, \bigg\{
g^{(1)[m,p] \, \ast}_0(\tau)
\, + \, 
(-1)^j \, g^{(1)[m,p'] \, \ast}_m(\tau)\bigg\}
$$
\item[{\rm (ii)}] \,\ $
g^{(1)[m,p] \, \ast}_0\Big(-\dfrac{1}{\tau}\Big)
\,\ = \,\ 
\dfrac{-i\tau}{2\sqrt{m(m+\frac12)}}
\sum\limits_{p'=0}^{2m} \, 
\sum\limits_{\substack{k \in \zzz \\[1mm] 0<k<m}}
e^{\frac{2\pi i}{2m+1}pp'} \, g^{(1)[m,p'] \, \ast}_k(\tau)$
$$
+ \,\ 
\frac{-i\tau}{2\sqrt{m(m+\frac12)}}
\sum_{p'=0}^{2m} \, e^{\frac{2\pi i}{2m+1}pp'} \, \bigg\{
g^{(1)[m,p'] \, \ast}_0(\tau)
\, + \, 
g^{(1)[m,p'] \, \ast}_m(\tau)\bigg\}
$$
\item[{\rm (iii)}] \,\ $
g^{(1)[m,p] \, \ast}_m\Big(-\dfrac{1}{\tau}\Big)
\,\ = \,\ 
\dfrac{-i\tau}{2\sqrt{m(m+\frac12)}}
\sum\limits_{p'=0}^{2m} \, 
\sum\limits_{\substack{k \in \zzz \\[1mm] 0<k<m}}
(-1)^k \, e^{\frac{2\pi i}{2m+1}pp'} \, 
g^{(1)[m,p'] \, \ast}_k(\tau)$
$$
+ \,\ 
\frac{-i\tau}{2\sqrt{m(m+\frac12)}}
\sum_{p'=0}^{2m} \, e^{\frac{2\pi i}{2m+1}pp'} \, \bigg\{
g^{(1)[m,p'] \, \ast}_0(\tau)
\, + \, 
(-1)^m \, g^{(1)[m,p'] \, \ast}_m(\tau)\bigg\}
$$
\end{enumerate}
\item[{\rm 2)}] \,\ $T$-transformation of $g^{(1)[m,p] \, \ast}_k(\tau)$ 
are as follows:
\begin{enumerate}
\item[{\rm (i)}] \,\ $\underset{\substack{\\[-1mm] (0 < j <m)
}}{g^{(1)[m,p] \, \ast}_j(\tau+1)}
\,\ = \,\  
e^{\frac{\pi i}{2m+1}p^2+\frac{\pi i}{4}-\frac{\pi i}{2m}j^2} \,
g^{(2)[m,p] \, \ast}_j(\tau)$
\item[{\rm (ii)}] \,\ $g^{(1)[m,p] \, \ast}_0(\tau+1) 
\,\ = \,\ 
e^{\frac{\pi i}{2m+1}p^2+\frac{\pi i}{4}} \,
g^{(2)[m,p] \, \ast}_0(\tau)$
\item[{\rm (iii)}] \,\ $g^{(1)[m,p] \, \ast}_m(\tau+1) 
\,\ = \,\ 
e^{\frac{\pi i}{2m+1}p^2+\frac{\pi i}{4}-\frac{\pi im}{2}} \,
g^{(2)[m,p] \, \ast}_m(\tau)$
\end{enumerate}
\end{enumerate}
\end{prop}

\begin{proof} Substituting \eqref{m1:eqn:2022-923a2}
and \eqref{m1:eqn:2022-924a1} into \eqref{m1:eqn:2022-923b2}, 
we have
{\allowdisplaybreaks
\begin{eqnarray*}
& & \hspace{-8mm}
\sum_{p=0}^{2m} \, e^{\frac{2\pi i}{2m+1}p\ell}
(-i\tau)^{\frac12} \, \sqrt{\dfrac{2}{m}} \, 
e^{\frac{\pi im}{2\tau}z^2} \, \Bigg( 
\\[2mm]
& & \hspace{-8mm}
\sum_{\substack{k \in \zzz \\[1mm] 0<k<m}}
g^{(1)[m,p] \, \ast}_k\Big(-\frac{1}{\tau}\Big)
\bigg\{
\sum_{\substack{j \in \zzz \\[1mm] 0<j<m}}
\cos \frac{\pi jk}{m} \, 
\big[\theta_{j,m}+ \theta_{-j,m}\big](\tau,z)
\, + \, \theta_{0,m}(\tau,z) 
\, + \, (-1)^k \, \theta_{m,m}(\tau,z)\bigg\}
\\[2mm]
& & \hspace{-5mm}
+ \,\ 
g^{(1)[m,p] \, \ast}_0\Big(-\frac{1}{\tau}\Big) \bigg\{
\sum_{\substack{j \in \zzz \\[1mm] 0<j<m}}
\big[\theta_{j,m}+ \theta_{-j,m}\big](\tau,z)
\, + \, \theta_{0,m}(\tau,z) 
\, + \, \theta_{m,m}(\tau,z)\bigg\}
\\[2mm]
& & \hspace{-5mm}
+ \,\ 
g^{(1)[m,p] \, \ast}_m\Big(-\frac{1}{\tau}\Big) \bigg\{
\sum_{\substack{j \in \zzz \\[1mm] 0<j<m}}
(-1)^j \, 
\big[\theta_{j,m}+ \theta_{-j,m}\big](\tau,z)
\, + \, \theta_{0,m}(\tau,z) 
\, + \, (-1)^m \, \theta_{m,m}(\tau,z)\bigg\}\Bigg)
\\[2mm]
&=&
(-i\tau)^{\frac32} \, \sqrt{2m+1} \,\ 
e^{\frac{\pi im}{2\tau}z^2} \, \Bigg\{
\sum_{\substack{j \in \zzz \\[1mm] 0<j<m}}
g^{(1)[m,\ell] \, \ast}_j(\tau)
\big[\theta_{j,m}+\theta_{-j,m}\big](\tau,z)
\\[0mm]
& &
+ \,\ 
2 \, g^{(1)[m,\ell] \, \ast}_0(\tau) \, \theta_{0,m}(\tau,z)
\,\ + \,\ 
2 \, g^{(1)[m,\ell] \, \ast}_m(\tau) \, \theta_{m,m}(\tau,z)\Bigg\}
\end{eqnarray*}}
namely 
{\allowdisplaybreaks
\begin{eqnarray*}
& & \hspace{-1mm}
(-i\tau)^{-1}\sqrt{\frac{2}{m(2m+1)}}
\sum_{p=0}^{2m} \, e^{\frac{2\pi i}{2m+1}p\ell}
\sum_{\substack{k \in \zzz \\[1mm] 0<k<m}}
g^{(1)[m,p] \, \ast}_k\Big(-\frac{1}{\tau}\Big)
\\[2mm]
& & 
\times \,\ \bigg\{
\sum_{\substack{j \in \zzz \\[1mm] 0<j<m}}
\cos \frac{\pi jk}{m} \, 
\big[\theta_{j,m}+ \theta_{-j,m}\big](\tau,z)
\, + \, \theta_{0,m}(\tau,z) 
\, + \, (-1)^k \, \theta_{m,m}(\tau,z)\bigg\}
\\[2mm]
& & \hspace{-5mm}
+ \,\ 
(-i\tau)^{-1}\sqrt{\frac{2}{m(2m+1)}}
\sum_{p=0}^{2m} \, e^{\frac{2\pi i}{2m+1}p\ell}
g^{(1)[m,p] \, \ast}_0\Big(-\frac{1}{\tau}\Big) 
\\[2mm]
& &
\times \,\ \bigg\{ 
\sum_{\substack{j \in \zzz \\[1mm] 0<j<m}}
\big[\theta_{j,m}+ \theta_{-j,m}\big](\tau,z)
\, + \, \theta_{0,m}(\tau,z) 
\, + \, \theta_{m,m}(\tau,z)\bigg\}
\\[2mm]
& & \hspace{-5mm}
+ \,\ 
(-i\tau)^{-1}\sqrt{\frac{2}{m(2m+1)}}
\sum_{p=0}^{2m} \, e^{\frac{2\pi i}{2m+1}p\ell}
g^{(1)[m,p] \, \ast}_m\Big(-\frac{1}{\tau}\Big) 
\\[2mm]
& & 
\times \,\ \bigg\{ 
\sum_{\substack{j \in \zzz \\[1mm] 0<j<m}}
(-1)^j \, 
\big[\theta_{j,m}+ \theta_{-j,m}\big](\tau,z)
\, + \, \theta_{0,m}(\tau,z) 
\, + \, (-1)^m \, \theta_{m,m}(\tau,z)\bigg\}
\\[3mm]
&=& \hspace{-2mm}
\sum_{\substack{j \in \zzz \\[1mm] 0<j<m}} \hspace{-1mm}
g^{(1)[m,\ell] \, \ast}_j(\tau)
\big[\theta_{j,m}+\theta_{-j,m}\big](\tau,z)
+ 
2 \, g^{(1)[m,\ell] \, \ast}_0(\tau) \, \theta_{0,m}(\tau,z)
+ 
2 \, g^{(1)[m,\ell] \, \ast}_m(\tau) \, \theta_{m,m}(\tau,z)
\end{eqnarray*}}
Comparing the coefficients of \, 
$\big[\theta_{j,m}+\theta_{-j,m}\big](\tau,z)$, \, 
$\theta_{0,m}(\tau,z)$ and $\theta_{m,m}(\tau,z)$
in this equation, we have 
{\allowdisplaybreaks
\begin{eqnarray*}
& & \hspace{-7mm}
\underset{\substack{\\[1mm] (0<j<m)}}{g^{(1)[m,\ell] \, \ast}_j(\tau)} 
\,\ = \,\
\frac{(-i\tau)^{-1}}{\sqrt{m(m+\frac12)}}
\sum_{p=0}^{2m} \, \sum_{\substack{k \in \zzz \\[1mm] 0<k<m}}
e^{\frac{2\pi i}{2m+1}p\ell} \, 
\cos \frac{\pi jk}{m} \, 
g^{(1)[m,p] \, \ast}_k\Big(-\frac{1}{\tau}\Big)
\\[2mm]
& & \hspace{15mm}
+ \,\ 
\frac{(-i\tau)^{-1}}{\sqrt{m(m+\frac12)}}
\sum_{p=0}^{2m} \, e^{\frac{2\pi i}{2m+1}p\ell} \, \bigg\{
g^{(1)[m,p] \, \ast}_0\Big(-\frac{1}{\tau}\Big)
\, + \, 
(-1)^j \, g^{(1)[m,p] \, \ast}_m\Big(-\frac{1}{\tau}\Big)\bigg\}
\\[3mm]
& & \hspace{-7mm}
2 \, g^{(1)[m,\ell] \, \ast}_0(\tau)
\,\ = \,\ 
\frac{(-i\tau)^{-1}}{\sqrt{m(m+\frac12)}}
\sum_{p=0}^{2m} \, \sum_{\substack{k \in \zzz \\[1mm] 0<k<m}}
e^{\frac{2\pi i}{2m+1}p\ell} \, 
g^{(1)[m,p] \, \ast}_k\Big(-\frac{1}{\tau}\Big)
\\[2mm]
& & \hspace{15mm}
+ \,\ 
\frac{(-i\tau)^{-1}}{\sqrt{m(m+\frac12)}}
\sum_{p=0}^{2m} \, e^{\frac{2\pi i}{2m+1}p\ell} \, \bigg\{
g^{(1)[m,p] \, \ast}_0\Big(-\frac{1}{\tau}\Big)
\, + \, 
g^{(1)[m,p] \, \ast}_m\Big(-\frac{1}{\tau}\Big)\bigg\}
\\[3mm]
& & \hspace{-7mm}
2 \, g^{(1)[m,\ell] \, \ast}_m(\tau)
\,\ = \,\ 
\frac{(-i\tau)^{-1}}{\sqrt{m(m+\frac12)}}
\sum_{p=0}^{2m} \, \sum_{\substack{k \in \zzz \\[1mm] 0<k<m}}
(-1)^k \, e^{\frac{2\pi i}{2m+1}p\ell} \, 
g^{(1)[m,p] \, \ast}_k\Big(-\frac{1}{\tau}\Big)
\\[2mm]
& & \hspace{15mm}
+ \,\ 
\frac{(-i\tau)^{-1}}{\sqrt{m(m+\frac12)}}
\sum_{p=0}^{2m} \, e^{\frac{2\pi i}{2m+1}p\ell} \, \bigg\{
g^{(1)[m,p] \, \ast}_0\Big(-\frac{1}{\tau}\Big)
\, + \, (-1)^m \, 
g^{(1)[m,p] \, \ast}_m\Big(-\frac{1}{\tau}\Big)\bigg\}
\end{eqnarray*}}
Then, replacing $\tau$ with $-\frac{1}{\tau}$ in these 
equations, we obtain the formulas in the claim 1).

\vspace{2.5mm}

\noindent
2) \,\ Substituting \eqref{m1:eqn:2022-923a2} and 
\eqref{m1:eqn:2022-924a2} into \eqref{m1:eqn:2022-1004a}, we have
{\allowdisplaybreaks
\begin{eqnarray*}
& &
\sum_{\substack{j \in \zzz \\[1mm] 0<j<m}}
e^{\frac{\pi i}{2m}j^2}
g^{(1)[m,p] \, \ast}_j(\tau+1)
\big[\theta_{j,m}+\theta_{-j,m}\big](\tau,z)
\\[2mm]
& &
+ \,\ 
2 \, g^{(1)[m,p] \, \ast}_0(\tau+1) \, \theta_{0,m}(\tau,z)
\,\ + \,\ 
2 \, e^{\frac{\pi im}{2}} \, 
g^{(1)[m,p] \, \ast}_m(\tau+1) \, \theta_{m,m}(\tau,z)
\\[3mm]
&=&
e^{\frac{\pi i}{2m+1}p^2+\frac{\pi i}{4}} \, \bigg\{
\sum_{\substack{j \in \zzz \\[1mm] 0<j<m}}
g^{(2)[m,p] \, \ast}_j(\tau)
\big[\theta_{j,m}+\theta_{-j,m}\big](\tau,z)
\\[2mm]
& &
+ \,\ 
2 \, g^{(2)[m,p] \, \ast}_0(\tau) \, \theta_{0,m}(\tau,z)
\,\ + \,\ 
2 \, g^{(2)[m,p] \, \ast}_m(\tau) \, \theta_{m,m}(\tau,z)\bigg\}
\end{eqnarray*}}
Comparing the coefficients of \,
$\big[\theta_{j,m}+\theta_{-j,m}\big](\tau,z)$, \, 
$\theta_{0,m}(\tau,z)$ and $\theta_{m,m}(\tau,z)$ 
in this equation, we have 
$$\left\{
\begin{array}{rcr}
e^{\frac{\pi i}{2m}j^2} 
\underset{\substack{\\[-1mm] (0 < j <m)
}}{g^{(1)[m,p] \, \ast}_j(\tau+1)}
&=& 
e^{\frac{\pi i}{2m+1}p^2+\frac{\pi i}{4}} \,
g^{(2)[m,p] \, \ast}_j(\tau)
\\[5mm]
g^{(1)[m,p] \, \ast}_0(\tau+1) &=&
e^{\frac{\pi i}{2m+1}p^2+\frac{\pi i}{4}} \,
g^{(2)[m,p] \, \ast}_0(\tau)
\\[2.5mm]
e^{\frac{\pi im}{2}} \, 
g^{(1)[m,p] \, \ast}_m(\tau+1) &=&
e^{\frac{\pi i}{2m+1}p^2+\frac{\pi i}{4}} \,
g^{(2)[m,p] \, \ast}_m(\tau)
\end{array}\right.
$$
namely 
$$\left\{
\begin{array}{rcr}
\underset{\substack{\\[-1mm] (0 < j <m)
}}{g^{(1)[m,p] \, \ast}_j(\tau+1)}
&=& 
e^{\frac{\pi i}{2m+1}p^2+\frac{\pi i}{4}-\frac{\pi i}{2m}j^2} \,
g^{(2)[m,p] \, \ast}_j(\tau)
\\[5mm]
g^{(1)[m,p] \, \ast}_0(\tau+1) &=&
e^{\frac{\pi i}{2m+1}p^2+\frac{\pi i}{4}} \,
g^{(2)[m,p] \, \ast}_0(\tau)
\\[2.5mm]
g^{(1)[m,p] \, \ast}_m(\tau+1) &=&
e^{\frac{\pi i}{2m+1}p^2+\frac{\pi i}{4}-\frac{\pi im}{2}} \,
g^{(2)[m,p] \, \ast}_m(\tau)
\end{array}\right.
$$
proving the claim 2).
\end{proof}

\subsection{Explicit formula for $g^{(i)[m,p] \, \ast}_k(\tau)$}
\label{subsec:g:ast}

The explicit formulas for $G^{(i)[m,p] \, \ast}(\tau,z)$ \,\ 
$(i \in \{1,2\})$ are obtained from Proposition 
\ref{m1:prop:2022-921b} and the formulas \eqref{m1:eqn:2022-914a1},
\eqref{m1:eqn:2022-914a2}, \eqref{m1:eqn:2022-914b1}, 
\eqref{m1:eqn:2022-914b2} and \eqref{m1:eqn:2022-923a1} as follows:
{\allowdisplaybreaks
\begin{eqnarray*}
& & \hspace{-10mm} 
G^{(1)[m,p] \, \ast}(\tau,z) \,\ = \,\ 
\\[3mm]
& & \hspace{-5mm}
- \,\ 
\bigg[\sum_{\substack{j, \, r \, \in \, \frac12 \, \zzz_{\rm odd} \\[1mm]
0 \, \leq \, r \, < \, j}}
-
\sum_{\substack{j, \, r \, \in \, \frac12 \, \zzz_{\rm odd} \\[1mm]
j \, \leq \, r \, < \, 0}}\bigg]
\sum_{\substack{k \, \in \zzz \\[1mm] 0 \, < \, k \, < \, m}}
(-1)^{j-\frac12+k} \, 
q^{(m+\frac12)(j+\frac{2pm}{2m+1})^2}
\nonumber
\\[-4mm]
& & \hspace{60mm}
\times \,\ 
q^{-\frac{1}{4m} \, [2mr+k+2mp]^2}
\big[\theta_{k,m}+\theta_{-k,m}\big](\tau, z)
\nonumber
\\[3mm]
& & \hspace{-5mm}
- \,\ \bigg[
\sum_{\substack{j, \, r \, \in \, \frac12 \, \zzz_{\rm odd} \\[1mm]
0 \, \leq \, r \, \leq \, j}}
-
\sum_{\substack{j, \, r \, \in \, \frac12 \, \zzz_{\rm odd} \\[1mm]
j \, < \, r \, < \, 0}}\bigg]
\sum_{\substack{k \, \in \zzz \\[1mm] 0 \, \leq \, k \, \leq \, m}}
(-1)^{j-\frac12+k} \, 
q^{(m+\frac12)(j+\frac{2pm}{2m+1})^2}
\nonumber
\\[-4mm]
& & \hspace{60mm}
\times \,\ 
q^{-\frac{1}{4m} \, [2mr-k+2mp]^2}
\big[\theta_{k,m}+\theta_{-k,m}\big](\tau, z)
\nonumber
\\[3mm]
& & \hspace{-5mm}
- \,\ 
\theta_{2mp+m+\frac12, m+\frac12}^{(-)}(\tau,0) \hspace{-3mm}
\sum_{\substack{r \, \in \zzz \\[1mm] -p \, < \, r \, \leq \, p}} \,\ 
\sum_{\substack{k \, \in \zzz \\[1mm] 0 \, \leq \, k \, < \, m}} 
\hspace{-2mm}
(-1)^k q^{- \, \frac{1}{4m} \, \big(m(2r-1)+k\big)^2} 
\big[\theta_{k,m}+\theta_{-k,m}\big](\tau, z)
\nonumber
\\[3mm]
& & \hspace{-5mm}
+ \,\ 2 \, 
\theta_{2mp+m+\frac12, m+\frac12}^{(-)}(\tau,0)
\sum_{\substack{r \, \in \zzz \\[1mm] 0 \, \leq \, r \, \leq \, p-1}}
q^{-m(r+\frac12)^2} \, \theta_{0,m}(\tau, z)
\nonumber
\\[2mm]
& & \hspace{-5mm}
- \,\ (-1)^m \, \theta_{2mp+m+\frac12, m+\frac12}^{(-)}(\tau,0)
\sum_{\substack{r \, \in \zzz \\[1mm] -p \, \leq \, r \, \leq \, p}}
q^{-mr^2} \, \theta_{m,m}(\tau,z)
\\[3mm]
& & \hspace{-10mm} 
G^{(2)[m,p] \, \ast}(\tau,z) \,\ = \,\ 
\\[3mm]
& & \hspace{-5mm}
- \,\ (-1)^p \, 
\bigg[\sum_{\substack{j, \, r \, \in \, \frac12 \, \zzz_{\rm odd} \\[1mm]
0 \, \leq \, r \, < \, j}}
-
\sum_{\substack{j, \, r \, \in \, \frac12 \, \zzz_{\rm odd} \\[1mm]
j \, \leq \, r \, < \, 0}}\bigg]
\sum_{\substack{k \, \in \zzz \\[1mm] 0 \, < \, k \, < \, m}}
(-1)^{j-\frac12} \, 
q^{(m+\frac12)(j+\frac{2pm}{2m+1})^2}
\nonumber
\\[-1mm]
& & \hspace{60mm}
\times \,\ 
q^{-\frac{1}{4m} \, [2mr+k+2mp]^2}
\big[\theta_{k,m}+\theta_{-k,m}\big](\tau, z)
\nonumber
\\[3mm]
& & \hspace{-5mm}
- \,\ (-1)^p \, \bigg[
\sum_{\substack{j, \, r \, \in \, \frac12 \, \zzz_{\rm odd} \\[1mm]
0 \, \leq \, r \, \leq \, j}}
-
\sum_{\substack{j, \, r \, \in \, \frac12 \, \zzz_{\rm odd} \\[1mm]
j \, < \, r \, < \, 0}}\bigg]
\sum_{\substack{k \, \in \zzz \\[1mm] 0 \, \leq \, k \, \leq \, m}}
(-1)^{j-\frac12} \, 
q^{(m+\frac12)(j+\frac{2pm}{2m+1})^2}
\nonumber
\\[-1mm]
& & \hspace{60mm}
\times \,\ 
q^{-\frac{1}{4m} \, [2mr-k+2mp]^2}
\big[\theta_{k,m}+\theta_{-k,m}\big](\tau, z)
\nonumber
\\[3mm]
& & \hspace{-5mm}
- \,\ (-1)^p \, 
\theta_{2mp+m+\frac12, m+\frac12}^{(-)}(\tau,0) \hspace{-3mm}
\sum_{\substack{r \, \in \zzz \\[1mm] -p \, < \, r \, \leq \, p}} \,\ 
\sum_{\substack{k \, \in \zzz \\[1mm] 0 \, \leq \, k \, < \, m}} 
\hspace{-2mm}
q^{- \, \frac{1}{4m} \, \big(m(2r-1)+k\big)^2} 
\big[\theta_{k,m}+\theta_{-k,m}\big](\tau, z)
\nonumber
\\[3mm]
& & \hspace{-5mm}
+ \,\ 2 \, (-1)^p \, 
\theta_{2mp+m+\frac12, m+\frac12}^{(-)}(\tau,0)
\sum_{\substack{r \, \in \zzz \\[1mm] 0 \, \leq \, r \, \leq \, p-1}}
q^{-m(r+\frac12)^2} \, \theta_{0,m}(\tau, z)
\nonumber
\\[2mm]
& & \hspace{-5mm}
- \,\ (-1)^p \, \theta_{2mp+m+\frac12, m+\frac12}^{(-)}(\tau,0)
\sum_{\substack{r \, \in \zzz \\[1mm] -p \, \leq \, r \, \leq \, p}}
q^{-mr^2} \, \theta_{m,m}(\tau,z)
\end{eqnarray*}}
Then the explicit formulas for $g^{(i)[m,p] \, \ast}_k(\tau)$
follow immediately from \eqref{m1:eqn:2022-923a2} and 
the above formulas as follows:

\medskip

\begin{prop} 
\label{m1:prop:2022-924b}
Let $m \in \nnn$ and $p,k \in \zzz$ such that 
$0 \, \leq \, p \, \leq \, 2m$ and $0 \, \leq \, k \, \leq \, m$.
Then $g^{(i)[m,p] \, \ast}_k(\tau)$ \,\ $(i \in \{1,2\})$ are 
as follows:
\begin{enumerate}
\item[{\rm 1)}]  
\begin{enumerate}
\item[{\rm (i)}] \,\ $
\underset{\substack{\\[1mm] (0<k<m)}}{g^{(1)[m,p] \, \ast}_k}(\tau)$
{\allowdisplaybreaks
\begin{eqnarray*}
&=&
- \,\ 
\bigg[\sum_{\substack{j, \, r \, \in \, \frac12 \, \zzz_{\rm odd} \\[1mm]
0 \, \leq \, r \, < \, j}}
-
\sum_{\substack{j, \, r \, \in \, \frac12 \, \zzz_{\rm odd} \\[1mm]
j \, \leq \, r \, < \, 0}}\bigg]
(-1)^{j-\frac12+k} \, 
q^{(m+\frac12)(j+\frac{2pm}{2m+1})^2}
q^{-\frac{1}{4m} \, [2mr+k+2mp]^2}
\nonumber
\\[1mm]
& & 
- \,\ \bigg[
\sum_{\substack{j, \, r \, \in \, \frac12 \, \zzz_{\rm odd} \\[1mm]
0 \, \leq \, r \, \leq \, j}}
-
\sum_{\substack{j, \, r \, \in \, \frac12 \, \zzz_{\rm odd} \\[1mm]
j \, < \, r \, < \, 0}}\bigg]
(-1)^{j-\frac12+k} \, 
q^{(m+\frac12)(j+\frac{2pm}{2m+1})^2}
q^{-\frac{1}{4m} \, [2mr-k+2mp]^2}
\nonumber
\\[1mm]
& & 
- \,\ 
\theta_{2mp+m+\frac12, m+\frac12}^{(-)}(\tau,0) \hspace{-3mm}
\sum_{\substack{r \, \in \zzz \\[1mm] -p \, < \, r \, \leq \, p}} \,\ 
(-1)^k q^{- \, \frac{1}{4m} \, \big(m(2r-1)+k\big)^2} 
\end{eqnarray*}}
\item[{\rm (ii)}] \,\ $2 \, g^{(1)[m,p] \, \ast}_0(\tau)$
{\allowdisplaybreaks
\begin{eqnarray*}
&= &
- \, 2 \, \bigg[
\sum_{\substack{j, \, r \, \in \, \frac12 \, \zzz_{\rm odd} \\[1mm]
0 \, < \, r \, \leq \, j}}
-
\sum_{\substack{j, \, r \, \in \, \frac12 \, \zzz_{\rm odd} \\[1mm]
j \, < \, r \, < \, 0}}\bigg]
(-1)^{j-\frac12} \, 
q^{(m+\frac12)(j+\frac{2pm}{2m+1})^2-m(r+p)^2}
\nonumber
\\[1mm]
& & 
- \, 2 \, 
\theta_{2mp+m+\frac12, m+\frac12}^{(-)}(\tau,0) 
\sum_{\substack{r \, \in \zzz \\[1mm] 0 \, \leq \, r \, \leq \, p-1}}  
q^{-m(r+\frac12)^2}
\end{eqnarray*}}
\item[{\rm (iii)}] \,\ $2 \, g^{(1)[m,p] \, \ast}_m(\tau)$
{\allowdisplaybreaks
\begin{eqnarray*}
&= & 
- \, 2 \, (-1)^m \, \bigg[
\sum_{\substack{j, \, r \, \in \, \frac12 \, \zzz_{\rm odd} \\[1mm]
0 \, < \, r \, \leq \, j}}
-
\sum_{\substack{j, \, r \, \in \, \frac12 \, \zzz_{\rm odd} \\[1mm]
j \, < \, r \, < \, 0}}\bigg]
(-1)^{j-\frac12} \, 
q^{(m+\frac12)(j+\frac{2pm}{2m+1})^2-m(r+p-\frac12)^2}
\nonumber
\\[0mm]
& & 
- \,\ (-1)^m \, \theta_{2mp+m+\frac12, m+\frac12}^{(-)}(\tau,0)
\sum_{\substack{r \, \in \zzz \\[1mm] -p \, \leq \, r \, \leq \, p}} 
q^{-mr^2} 
\end{eqnarray*}}
\end{enumerate}
\item[{\rm 2)}] 
\begin{enumerate}
\item[{\rm (i)}] \,\ $
\underset{\substack{\\[1mm] (0 < k < m)}}{g^{(2)[m,p] \, \ast}_k}(\tau)$
{\allowdisplaybreaks
\begin{eqnarray*}
&=&
- \,\ (-1)^p \, 
\bigg[\sum_{\substack{j, \, r \, \in \, \frac12 \, \zzz_{\rm odd} \\[1mm]
0 \, \leq \, r \, < \, j}}
-
\sum_{\substack{j, \, r \, \in \, \frac12 \, \zzz_{\rm odd} \\[1mm]
j \, \leq \, r \, < \, 0}}\bigg]
(-1)^{j-\frac12} \, 
q^{(m+\frac12)(j+\frac{2pm}{2m+1})^2}
q^{-\frac{1}{4m} \, [2mr+k+2mp]^2}
\nonumber
\\[1mm]
& & 
- \,\ (-1)^p \, \bigg[
\sum_{\substack{j, \, r \, \in \, \frac12 \, \zzz_{\rm odd} \\[1mm]
0 \, \leq \, r \, \leq \, j}}
-
\sum_{\substack{j, \, r \, \in \, \frac12 \, \zzz_{\rm odd} \\[1mm]
j \, < \, r \, < \, 0}}\bigg]
(-1)^{j-\frac12} \, 
q^{(m+\frac12)(j+\frac{2pm}{2m+1})^2}
q^{-\frac{1}{4m} \, [2mr-k+2mp]^2}
\nonumber
\\[1mm]
& & 
- \,\ (-1)^p \, 
\theta_{2mp+m+\frac12, m+\frac12}^{(-)}(\tau,0) \hspace{-3mm}
\sum_{\substack{r \, \in \zzz \\[1mm] -p \, < \, r \, \leq \, p}} \,\ 
q^{- \, \frac{1}{4m} \, \big(m(2r-1)+k\big)^2} 
\end{eqnarray*}}
\item[{\rm (ii)}] \,\ $2 \, g^{(2)[m,p] \, \ast}_0(\tau)$
{\allowdisplaybreaks
\begin{eqnarray*}
&=&
- \, 2 \, (-1)^p \, \bigg[
\sum_{\substack{j, \, r \, \in \, \frac12 \, \zzz_{\rm odd} \\[1mm]
0 \, \leq \, r \, \leq \, j}}
-
\sum_{\substack{j, \, r \, \in \, \frac12 \, \zzz_{\rm odd} \\[1mm]
j \, < \, r \, < \, 0}}\bigg]
(-1)^{j-\frac12} \, 
q^{(m+\frac12)(j+\frac{2pm}{2m+1})^2-m(r+p)^2}
\nonumber
\\[1mm]
& & 
- \, 2 \, (-1)^p \, 
\theta_{2mp+m+\frac12, m+\frac12}^{(-)}(\tau,0) 
\sum_{\substack{r \, \in \zzz \\[1mm] 0 \, \leq \, r \, \leq \, p-1}} 
q^{- m(r+\frac12)^2}
\end{eqnarray*}}
\item[{\rm (iii)}] \,\ $2 \, g^{(2)[m,p] \, \ast}_m(\tau)$
{\allowdisplaybreaks
\begin{eqnarray*}
&=&
- \, 2 \, (-1)^p \, \bigg[
\sum_{\substack{j, \, r \, \in \, \frac12 \, \zzz_{\rm odd} \\[1mm]
0 \, \leq \, r \, \leq \, j}}
-
\sum_{\substack{j, \, r \, \in \, \frac12 \, \zzz_{\rm odd} \\[1mm]
j \, < \, r \, < \, 0}}\bigg]
(-1)^{j-\frac12} \, 
q^{(m+\frac12)(j+\frac{2pm}{2m+1})^2-m(r+p-\frac12)^2}
\nonumber
\\[1mm]
& & 
- \,\ (-1)^p \, \theta_{2mp+m+\frac12, m+\frac12}^{(-)}(\tau,0)
\sum_{\substack{r \, \in \zzz \\[1mm] -p \, \leq \, r \, \leq \, p}} 
q^{-mr^2}
\end{eqnarray*}}
\end{enumerate}
\end{enumerate}
\end{prop}

\section{Indefinite modular forms $g^{[m,p] \, \ast}_j(\tau)$}
\label{sec:g(mp):indefinite}

By Proposition \ref{m1:prop:2022-924b} we observe that
the following formula 
\begin{equation}
g^{(2)[m,p] \, \ast}_j(\tau) \,\ = \,\ 
(-1)^{j+p} \, g^{(1)[m,p] \, \ast}_j(\tau) 
\label{m1:eqn:2022-924b}
\end{equation}
holds for all $m \in \nnn$ and $p, \, j \in \zzz$ such that 
$0 \, \leq \, p \, \leq \, 2m$ and $0 \, \leq \, j \, \leq \, m$.
Then, simplifying the notation, we define the functions \, 
$g^{[m,p] \, \ast}_j(\tau)$ \, by 
\begin{equation}
g^{[m,p] \, \ast}_j(\tau) \,\ = \,\ g^{(1)[m,p] \, \ast}_j(\tau) 
\label{m1:eqn:2022-924c}
\end{equation}
Then the modular transformation formulas for these functions 
\, $g^{[m,p] \, \ast}_j(\tau)$ \, are obtained from 
Proposition \ref{m1:prop:2022-924b} as follows:

\medskip

\begin{prop} 
\label{m1:prop:2022-924c}
Let $m \in \nnn$ and $p,j \in \zzz$ such that 
$0 \, \leq \, p \, \leq \, 2m$ and $0 \, \leq \, j \, \leq \, m$.
Then 
\begin{enumerate}
\item[{\rm 1)}] 
\begin{enumerate}
\item[{\rm (i)}] \,\ $
\underset{\substack{\\[1mm] (0<j<m)}}{g^{[m,p] \, \ast}_j
\Big(-\dfrac{1}{\tau}\Big)} 
\,\ = \,\
\dfrac{-i\tau}{\sqrt{m(m+\frac12)}}
\sum\limits_{p'=0}^{2m} \, 
\sum\limits_{\substack{k \in \zzz \\[1mm] 0<k<m}}
e^{\frac{2\pi i}{2m+1}pp'} \, 
\cos \dfrac{\pi jk}{m} \, g^{[m,p'] \, \ast}_k(\tau)$
$$
+ \,\ 
\frac{-i\tau}{\sqrt{m(m+\frac12)}}
\sum_{p'=0}^{2m} \, e^{\frac{2\pi i}{2m+1}pp'} \, \bigg\{
g^{[m,p] \, \ast}_0(\tau)
\, + \, 
(-1)^j \, g^{[m,p'] \, \ast}_m(\tau)\bigg\}
$$
\item[{\rm (ii)}] \,\ $
g^{[m,p] \, \ast}_0\Big(-\dfrac{1}{\tau}\Big)
\,\ = \,\ 
\dfrac{-i\tau}{\sqrt{2m(2m+1)}}
\sum\limits_{p'=0}^{2m} \, 
\sum\limits_{\substack{k \in \zzz \\[1mm] 0<k<m}}
e^{\frac{2\pi i}{2m+1}pp'} \, g^{[m,p'] \, \ast}_k(\tau)$
$$
+ \,\ 
\frac{-i\tau}{\sqrt{2m(2m+1)}}
\sum_{p'=0}^{2m} \, e^{\frac{2\pi i}{2m+1}pp'} \, \bigg\{
g^{[m,p'] \, \ast}_0(\tau)
\, + \, 
g^{[m,p'] \, \ast}_m(\tau)\bigg\}
$$
\item[{\rm (iii)}] \,\ $
g^{[m,p] \, \ast}_m\Big(-\dfrac{1}{\tau}\Big)
\,\ = \,\ 
\dfrac{-i\tau}{\sqrt{2m(2m+1)}}
\sum\limits_{p'=0}^{2m} \, 
\sum\limits_{\substack{k \in \zzz \\[1mm] 0<k<m}}
(-1)^k \, e^{\frac{2\pi i}{2m+1}pp'} \, 
g^{[m,p'] \, \ast}_k(\tau)$
$$
+ \,\ 
\frac{-i\tau}{\sqrt{2m(2m+1)}}
\sum_{p'=0}^{2m} \, e^{\frac{2\pi i}{2m+1}pp'} \, \bigg\{
g^{[m,p'] \, \ast}_0(\tau)
\, + \, 
(-1)^m \, g^{[m,p'] \, \ast}_m(\tau)\bigg\}
$$
\end{enumerate}
\item[{\rm 2)}] \,\ $
\underset{\substack{\\[0mm] (0 \, \leq \, j \, \leq \, m)
}}{g^{[m,p] \, \ast}_j(\tau+1)} \,\ = \,\ 
e^{\frac{\pi i}{2m+1}(p+\frac{2m+1}{2})^2- \frac{\pi i}{2m}(j+m)^2} \, 
g^{[m,p] \, \ast}_j(\tau)$
\end{enumerate}
\end{prop}

\section{An example \, $\sim$ \, the case $m=1$}
\label{sec:g(1):ast:ex:m=1}

The functions which take place in the case $m=1$ are 
$g^{[1,p] \, \ast}_k(\tau)$ \,\ $(p=0,1,2; \,\ k=0,1)$
and they are, by Proposition \ref{m1:prop:2022-924b}, as follows:
\begin{subequations}
{\allowdisplaybreaks
\begin{eqnarray}
2 \, g^{[1,p] \, \ast}_0(\tau)
&= &
- \, 2 \, \bigg[
\sum_{\substack{j, \, r \, \in \, \frac12 \, \zzz_{\rm odd} \\[1mm]
0 \, < \, r \, \leq \, j
}}
-\sum_{\substack{j, \, r \, \in \, \frac12 \, \zzz_{\rm odd} \\[1mm]
j \, < \, r \, < \, 0}}\bigg]
(-1)^{j-\frac12} \, 
q^{\frac32(j+\frac{2p}{3})^2-(r+p)^2}
\nonumber
\\[2mm]
& & 
- \, 2 \, 
\theta_{2p+\frac32, \frac32}^{(-)}(\tau,0) 
\sum_{\substack{r \, \in \, \zzz \\[1mm] 0 \, \leq \, r \, \leq \, p-1}} 
q^{-(r+\frac12)^2}
\label{eqn:2022-917a1}
\\[3mm]
2 \, g^{[1,p] \, \ast}_1(\tau) 
&= & 
2 \, \bigg[
\sum_{\substack{j, \, r \, \in \, \frac12 \, \zzz_{\rm odd} \\[1mm]
0 \, < \, r \, \leq \, j
}}
-\sum_{\substack{j, \, r \, \in \, \frac12 \, \zzz_{\rm odd} \\[1mm]
j \, < \, r \, < \, 0}}\bigg]
(-1)^{j-\frac12} \, 
q^{\frac32(j+\frac{2p}{3})^2-(r+p-\frac12)^2}
\nonumber
\\[2mm]
& & 
+ \,\ \theta_{2p+\frac32, \frac32}^{(-)}(\tau,0)
\sum_{\substack{r \, \in \, \zzz \\[1mm] -(p-1) \, \leq \, r \, \leq \, p-1}} 
q^{-r^2} 
\label{eqn:2022-917a2}
\end{eqnarray}}
\end{subequations}
Putting $\left\{
\begin{array}{ccc}
j-\frac12 &=& j' \\[1mm]
r-\frac12 &=& r'
\end{array}\right. $, 
the above formulas are rewritten as follows:
\begin{subequations}
{\allowdisplaybreaks
\begin{eqnarray}
g^{[1,p] \, \ast}_0(\tau)
&= &
- \,\ \bigg[
\sum_{\substack{j', \, r' \, \in \, \zzz \\[1mm] 
0 \, \leq \, r' \, \leq \, j'
}}
-\sum_{\substack{j', \, r' \, \in \, \zzz \\[1mm] 
j' \, < \, r' \, < \, 0}} \bigg]
(-1)^{j'} \, 
q^{\frac32(j'+\frac12+\frac{2p}{3})^2-(r'+\frac12+p)^2}
\nonumber
\\[2mm]
& & 
- \,\ 
\theta_{2p+\frac32, \frac32}^{(-)}(\tau,0) 
\sum_{\substack{r \, \in \, \zzz \\[1mm] 0 \, \leq \, r \, \leq \, p-1}}
q^{-(r+\frac12)^2}
\label{eqn:2022-917b1}
\\[3mm]
g^{[1,p] \, \ast}_1(\tau) 
&= & 
\bigg[
\sum_{\substack{j', \, r' \, \in \, \zzz \\[1mm] 
0 \, \leq \, r' \, \leq \, j'
}}
-\sum_{\substack{j', \, r' \, \in \, \zzz \\[1mm] 
j' \, < \, r' \, < \, 0}} \bigg]
(-1)^{j'} \, 
q^{\frac32(j'+\frac12+\frac{2p}{3})^2-(r'+p)^2}
\nonumber
\\[2mm]
& & 
+ \,\ \frac12 \, \theta_{2p+\frac32, \frac32}^{(-)}(\tau,0)
\sum_{\substack{r \, \in \, \zzz \\[1mm] -(p-1) \, \leq \, r \, \leq \, p-1}} 
q^{-r^2} 
\label{eqn:2022-917b2}
\end{eqnarray}}
\end{subequations}
These functions are written explicitly as follows:

\vspace{1mm}

\begin{note} \,\ 
\label{m1:note:2022-1002a}
\begin{enumerate}
\item[{\rm 1)}] \,\ $g^{[1,0] \, \ast}_0(\tau)
\, = \, 
- \,\ \bigg[
\sum\limits_{\substack{j, \, r \, \in \, \zzz \\[1mm] 
0 \, \leq \, r \, \leq \, j
}}
-\sum\limits_{\substack{j, \, r \, \in \, \zzz \\[1mm] 
j \, < \, r \, < \, 0}} \bigg]
(-1)^j \, q^{\frac32(j+\frac12)^2-(r+\frac12)^2}
\,\ = \,\ 
- \, q^{\frac18} + \cdots $
\item[{\rm 2)}] \,\ $g^{[1,0] \, \ast}_1(\tau) 
\, = \, 
\bigg[
\sum\limits_{\substack{j, \, r \, \in \, \zzz \\[1mm] 
0 \, \leq \, r \, \leq \, j
}}
-\sum\limits_{\substack{j, \, r \, \in \, \zzz \\[1mm] 
j \, < \, r \, < \, 0}} \bigg]
(-1)^j \, q^{\frac32(j+\frac12)^2-r^2}
\quad = \quad q^{\frac38}+ \cdots$
\item[{\rm 3)}] \,\ $g^{[1,1] \, \ast}_0(\tau)
\,\ = \,\ g^{[1,2] \, \ast}_0(\tau)$
$$ \hspace{-13mm}
= \,\ 
\bigg[
\sum\limits_{\substack{j, \, r \, \in \, \zzz \\[1mm] 0 \, \leq \, r \, < \, j}} 
-
\sum\limits_{\substack{j, \, r \, \in \, \zzz \\[1mm] j \, \leq \, r \, < \, 0}}
\bigg]
(-1)^j \, q^{\frac32(j+\frac16)^2-(r+\frac12)^2}
\,\ = \,\ q^{\frac{19}{24}}+\cdots 
$$
\item[{\rm 4)}] \,\ $g^{[1,1] \, \ast}_1(\tau)
\,\ = \,\ g^{[1,2] \, \ast}_1(\tau)$
$$= \,\
- \, \bigg[
\sum\limits_{\substack{j, \, r \, \in \, \zzz \\[1mm] 0 \, \leq \, r \, < \, j}} 
-
\sum\limits_{\substack{j, \, r \, \in \, \zzz \\[1mm] j \, \leq \, r \, < \, 0}}
\bigg]
(-1)^jq^{\frac32(j+\frac16)^2-r^2}
\, + \, 
\dfrac12 \, \theta_{\frac12, \frac32}^{(-)}(\tau,0)
\,\ = \,\ 
- \, \frac12 \, q^{\frac{1}{24}}+\cdots 
$$
\end{enumerate}
\end{note}

\medskip

The $S$-transformation of these functions is computed by Proposition 
\ref{m1:prop:2022-924a} as follows:
{\allowdisplaybreaks
\begin{eqnarray*}
g^{[1,p] \, \ast}_0\Big(-\dfrac{1}{\tau}\Big)
&=& 
\dfrac{-i\tau}{\sqrt{6}}
\sum\limits_{p'=0}^{2} \, e^{\frac{2\pi i}{3}pp'} \, \bigg\{
g^{[1,p'] \, \ast}_0(\tau)
\, + \, 
g^{[1,p'] \, \ast}_1(\tau)\bigg\}
\\[2mm]
g^{[1,p] \, \ast}_1\Big(-\dfrac{1}{\tau}\Big)
&=& 
\dfrac{-i\tau}{\sqrt{6}}
\sum\limits_{p'=0}^{2} \, e^{\frac{2\pi i}{3}pp'} \, \bigg\{
g^{[1,p'] \, \ast}_0(\tau)
\, - \, 
g^{[1,p'] \, \ast}_1(\tau)\bigg\}
\end{eqnarray*}}
Since $g^{[1,1] \, \ast}_k(\tau)=g^{[1,2] \, \ast}_k(\tau)$ \, $(k=0,1)$ by Note 
\ref{m1:note:2022-1002a}, these formulas are written explicitly as follows:

\begin{note} \,\ 
\label{m1:note:2022-1002b}
\begin{enumerate}
\item[{\rm 1)}] \,\ $g^{[1,0] \, \ast}_0\Big(-\dfrac{1}{\tau}\Big)
\,\ = \,\ 
\dfrac{-i\tau}{\sqrt{6}} \, \Big\{
g^{[1,0] \, \ast}_0(\tau) 
\, + \, g^{[1,0] \, \ast}_1(\tau)
\, + \, 2 \, g^{[1,1] \, \ast}_0(\tau) 
\, + \, 2 \, g^{[1,1] \, \ast}_1(\tau)\Big\}$
\item[{\rm 2)}] \,\ $g^{[1,0] \, \ast}_1\Big(-\dfrac{1}{\tau}\Big)
\,\ = \,\ 
\dfrac{-i\tau}{\sqrt{6}} \, \Big\{
g^{[1,0] \, \ast}_0(\tau) 
\, - \, g^{[1,0] \, \ast}_1(\tau)
\, + \, 2 \, g^{[1,1] \, \ast}_0(\tau) 
\, - \, 2 \, g^{[1,1] \, \ast}_1(\tau)\Big\}$
\item[{\rm 3)}] \,\ $g^{[1,1] \, \ast}_0\Big(-\dfrac{1}{\tau}\Big)
\,\ = \,\ 
\dfrac{-i\tau}{\sqrt{6}}\, \Big\{
g^{[1,0] \, \ast}_0(\tau)
\, + \, g^{[1,0] \, \ast}_1(\tau)
\, - \, g^{[1,1] \, \ast}_0(\tau)
\, - \, g^{[1,1] \, \ast}_1(\tau)\Big\}$
\item[{\rm 4)}] \,\ $g^{[1,1] \, \ast}_1\Big(-\dfrac{1}{\tau}\Big)
\,\ = \,\ 
\dfrac{-i\tau}{\sqrt{6}}\, \Big\{
g^{[1,0] \, \ast}_0(\tau)
\, - \, g^{[1,0] \, \ast}_1(\tau)
\, - \, g^{[1,1] \, \ast}_0(\tau)
\, + \, g^{[1,1] \, \ast}_1(\tau)\Big\}$
\end{enumerate}
\end{note}

\vspace{1mm}

From Notes \ref{m1:note:2022-1002a} and \ref{m1:note:2022-1002b},
we have the following: 

\vspace{1mm}

\begin{lemma} 
\label{m1:lemma:2022-1002a}
Define the functions $f_i(\tau)$ \,\ $(i=0,1,2,3)$ by
\begin{equation}
\left\{
\begin{array}{ccr}
f_0(\tau) &:=& \dfrac{- \, 2 \, g^{[1,1] \, \ast}_1(\tau)}{\eta(\tau)^2} \\[5mm]
f_3(\tau) &:=& \dfrac{2 \, g^{[1,1] \, \ast}_0(\tau)}{\eta(\tau)^2}
\end{array}\right. \hspace{10mm} \left\{
\begin{array}{ccr}
f_1(\tau) &:=& \dfrac{- \, g^{[1,0] \, \ast}_0(\tau)}{\eta(\tau)^2} \\[5mm]
f_2(\tau) &:=& \dfrac{g^{[1,0] \, \ast}_1(\tau)}{\eta(\tau)^2}
\end{array}\right. 
\label{m1:eqn:2022-1002a}
\end{equation} 
Then,
\begin{enumerate}
\item[{\rm 1)}] \,\ these functions $f_i(\tau)$ satisfy the following 
$S$-transformation properties:
\begin{enumerate}
\item[{\rm (0)}] \,\ $f_0\Big(-\dfrac{1}{\tau}\Big)
\,\ = \,\ 
\dfrac{1}{\sqrt{6}}\, \Big\{
f_0(\tau) 
\, + \, 2 \, f_1(\tau)
\, + \, 2 \, f_2 (\tau)
\, + \, f_3(\tau) \Big\}$
\item[{\rm (i)}] \,\ $f_1 \Big(-\dfrac{1}{\tau}\Big)
\,\ = \,\ 
\dfrac{1}{\sqrt{6}} \, \Big\{
f_0(\tau)
\, + \, f_1(\tau) 
\, - \, f_2(\tau)
\, - \, f_3(\tau)\Big\}$
\item[{\rm (ii)}] \,\ $f_2\Big(-\dfrac{1}{\tau}\Big)
\,\ = \,\ 
\dfrac{1}{\sqrt{6}} \, \Big\{
f_0(\tau)
\, - \, f_1(\tau) 
\, - \, f_2(\tau)
\, + \, f_3(\tau) \Big\}$
\item[{\rm (iii)}] \,\ $f_3\Big(-\dfrac{1}{\tau}\Big)
\,\ = \,\ 
\dfrac{1}{\sqrt{6}}\, \Big\{
f_0(\tau)
\, - \, 2 \, f_1(\tau)
\, + \, 2 \, f_2(\tau)
\, - \, f_3(\tau) \Big\} $
\end{enumerate}
\item[{\rm 2)}] \,\ the leading terms of $f_i(\tau)$ are as follows \,\ :
\quad $ \left\{
\begin{array}{lcrc}
f_0(\tau) &=& q^{-\frac{1}{24}} &  + \,\ \cdots 
\\[1mm]
f_1(\tau) &=& q^{\frac{1}{24}} &  + \,\ \cdots 
\\[1mm]
f_2(\tau) &=& q^{\frac{7}{24}} &  + \,\ \cdots 
\\[1mm]
f_3(\tau) &=& 2 \, q^{\frac{17}{24}} &  + \,\ \cdots 
\end{array} \right. $
\end{enumerate}
\end{lemma}

Next we consider the Jacobi's theta function
$$
\theta_{j,3}(\tau,z) \,\ = \,\ 
\sum_{n \, \in \, \zzz} \, q^{3(n+\frac{j}{6})^2} \, e^{6\pi i(n+\frac{j}{6})z}
$$
The $S$-transformation of $\theta_{j,3}(\tau,0)$ is computed by using 
Lemmas 1.2 and 1.3 in \cite{W2022c} as follows:
\begin{enumerate}
\item[{\rm (i)}] \,\ $\theta_{0,3}\Big(-\dfrac{1}{\tau},0\Big)
\,\ = \,\ 
\dfrac{(-i\tau)^{\frac12}}{\sqrt{6}} \Big\{
\theta_{0,3}(\tau,0)
\, + \, 2 \, \theta_{1,3}(\tau,0)
\, + \, 2 \, \theta_{2,3}(\tau,0)
\, + \, \theta_{3,3}(\tau,0)\Big\}$
\item[{\rm (ii)}] \,\ $\theta_{1,3}\Big(-\dfrac{1}{\tau},0\Big)
\,\ = \,\ 
\dfrac{(-i\tau)^{\frac12}}{\sqrt{6}} \Big\{
\theta_{0,3}(\tau,0)
\, + \, \theta_{1,3}(\tau,0)
\, - \, \theta_{2,3}(\tau,0)
\, - \, \theta_{3,3}(\tau,0)\Big\}$
\item[{\rm (iii)}] \,\ $\theta_{2,3}\Big(-\dfrac{1}{\tau},0\Big)
\,\ = \,\ 
\dfrac{(-i\tau)^{\frac12}}{\sqrt{6}} \Big\{
\theta_{0,3}(\tau,0)
\, - \, \theta_{1,3}(\tau,0)
\, - \, \theta_{2,3}(\tau,0)
\, + \, \theta_{3,3}(\tau,0)\Big\}$
\item[{\rm (iv)}] \,\ $\theta_{3,3}\Big(-\dfrac{1}{\tau},0\Big)
\,\ = \,\ 
\dfrac{(-i\tau)^{\frac12}}{\sqrt{6}} \Big\{
\theta_{0,3}(\tau,0)
\, - \, 2 \, \theta_{1,3}(\tau,0)
\, + \, 2 \, \theta_{2,3}(\tau,0)
\, - \, \theta_{3,3}(\tau,0)\Big\}$
\end{enumerate}
And the leading terms of $\theta_{j,3}(\tau,0)$ are \quad $
\left\{
\begin{array}{lcrc}
\theta_{0,3}(\tau,0) &=& q^0 & + \,\ \cdots 
\\[1mm]
\theta_{1,3}(\tau,0) &=& q^{\frac{1}{12}} &+ \,\ \cdots 
\\[1mm]
\theta_{2,3}(\tau,0) &=& q^{\frac13} &+ \,\ \cdots 
\\[1mm]
\theta_{3,3}(\tau,0) &=& 2 \, q^{\frac34} &+ \,\ \cdots 
\end{array} \right. $

\medskip

Then putting 
\begin{equation}
h_j(\tau) \,\ := \,\ \frac{\theta_{j,3}(\tau,0)}{\eta(\tau)} \hspace{10mm} 
(j=0,1,2,3)
\label{m1:eqn:2022-1002b}
\end{equation} 
we have

\medskip

\begin{lemma} 
\label{m1:lemma:2022-1002b}
\begin{enumerate}
\item[{\rm 1)}] \,\ Functions $h_j(\tau)$ satisfy the following 
$S$-transformation properties:
\begin{enumerate}
\item[{\rm (0)}] \,\ $h_0\Big(-\dfrac{1}{\tau}\Big)
\,\ = \,\ 
\dfrac{1}{\sqrt{6}}\, \Big\{
h_0(\tau) 
\, + \, 2 \, h_1(\tau)
\, + \, 2 \, h_2 (\tau)
\, + \, h_3(\tau) \Big\}$
\item[{\rm (i)}] \,\ $h_1 \Big(-\dfrac{1}{\tau}\Big)
\,\ = \,\ 
\dfrac{1}{\sqrt{6}} \, \Big\{
h_0(\tau)
\, + \, h_1(\tau) 
\, - \, h_2(\tau)
\, - \, h_3(\tau)\Big\}$
\item[{\rm (ii)}] \,\ $h_2\Big(-\dfrac{1}{\tau}\Big)
\,\ = \,\ 
\dfrac{1}{\sqrt{6}} \, \Big\{
h_0(\tau)
\, - \, h_1(\tau) 
\, - \, h_2(\tau)
\, + \, h_3(\tau) \Big\}$
\item[{\rm (iii)}] \,\ $h_3\Big(-\dfrac{1}{\tau}\Big)
\,\ = \,\ 
\dfrac{1}{\sqrt{6}}\, \Big\{
h_0(\tau)
\, - \, 2 \, h_1(\tau)
\, + \, 2 \, h_2(\tau)
\, - \, h_3(\tau) \Big\} $
\end{enumerate}
\item[{\rm 2)}] \,\ the leading terms of $h_j(\tau)$ are as follows \,\ :
\quad $ \left\{
\begin{array}{lcrc}
h_0(\tau) &=& q^{-\frac{1}{24}} &  + \,\ \cdots 
\\[1mm]
h_1(\tau) &=& q^{\frac{1}{24}} &  + \,\ \cdots 
\\[1mm]
h_2(\tau) &=& q^{\frac{7}{24}} &  + \,\ \cdots 
\\[1mm]
h_3(\tau) &=& 2 \, q^{\frac{17}{24}} &  + \,\ \cdots 
\end{array} \right. $
\end{enumerate}
\end{lemma}

From these formulas, we obtain the following :

\medskip

\begin{prop} \,\
\label{m1:prop:2022-1002a}
\begin{enumerate}
\item[{\rm (i)}] \,\ $
\bigg[\sum\limits_{\substack{j, \, r \, \in \, \zzz \\[1mm] 
0 \, \leq \, r \, \leq \, j }}
-\sum\limits_{\substack{j, \, r \, \in \, \zzz \\[1mm] 
j \, < \, r \, < \, 0}} \bigg]
(-1)^j \, q^{\frac32(j+\frac12)^2-(r+\frac12)^2} 
\,\ = \,\ 
\eta(\tau) \, \theta_{1,3}(\tau,0)$
\item[{\rm (ii)}] \,\ $
\bigg[\sum\limits_{\substack{j, \, r \, \in \, \zzz \\[1mm] 
0 \, \leq \, r \, \leq \, j}}
-\sum\limits_{\substack{j, \, r \, \in \, \zzz \\[1mm] 
j \, < \, r \, < \, 0}} \bigg]
(-1)^j \, q^{\frac32(j+\frac12)^2-r^2}
\,\ = \,\ 
\eta(\tau) \, \theta_{2,3}(\tau,0)$
\item[{\rm (iii)}] \,\ $
\bigg[
\sum\limits_{\substack{j, \, r \, \in \, \zzz \\[1mm] 
0 \, \leq \, r \, < \, j}} 
-
\sum\limits_{\substack{j, \, r \, \in \, \zzz \\[1mm] 
j \, \leq \, r \, < \, 0}}
\bigg]
(-1)^j \, q^{\frac32(j+\frac16)^2-(r+\frac12)^2}
\,\ = \,\ 
\dfrac12 \,  \eta(\tau) \, \theta_{3,3}(\tau,0)$
\item[{\rm (iv)}] \,\ $
\bigg[
\sum\limits_{\substack{j, \, r \, \in \zzz \\[1mm] 0 \, \leq \, r \, < \, j}} 
-
\sum\limits_{\substack{j, \, r \, \in \zzz \\[1mm] j \, \leq \, r \, < \, 0}}
\bigg]
(-1)^jq^{\frac32(j+\frac16)^2-r^2}
\,\ = \,\ 
\dfrac12 \, \Big\{ \eta(\tau) \, \theta_{0,3}(\tau,0)
\, + \, 
\theta_{\frac12, \frac32}^{(-)}(\tau,0)\Big\}$
\end{enumerate}
\end{prop}

\begin{proof} By Lemmas \ref{m1:lemma:2022-1002a}
and \ref{m1:lemma:2022-1002b}, we see that the functions 
$\{f_i(\tau)\}_{i=0,1,2,3}$ and $\{h_i(\tau)\}_{i=0,1,2,3}$ 
satisfy the same $S$-transformation
properties and have the same polar parts.
Then, by Lemma 4.8 in \cite{W2001}, we have  
$$ \hspace{-20mm}
f_i(\tau) \, = \, h_i(\tau) \qquad \text{for all $i$},
$$
namely
$$
\left\{
\begin{array}{rcl}
2 \, g^{[1,1] \, \ast}_0(\tau) &=& \eta(\tau) \, \theta_{3,3}(\tau,0)
\\[2mm]
- \, 2 \, g^{[1,1] \, \ast}_1(\tau) &=& \eta(\tau) \, \theta_{0,3}(\tau,0)
\end{array}\right. \hspace{10mm} \left\{
\begin{array}{rcl}
- \, g^{[1,0] \, \ast}_0(\tau) &=& \eta(\tau) \, \theta_{1,3}(\tau,0)
\\[2mm]
g^{[1,0] \, \ast}_1(\tau) &=& \eta(\tau) \, \theta_{2,3}(\tau,0)
\end{array}\right. 
$$
by \eqref{m1:eqn:2022-1002a} and \eqref{m1:eqn:2022-1002b}.
Rewriting $g^{[1,p] \, \ast}_k(\tau)$ by using Note \ref{m1:note:2022-1002a},
we obtain the formulas in Proposition \ref{m1:prop:2022-1002a}.
\end{proof}

\end{document}